\documentclass[10pt]{amsart}
\usepackage{amssymb,amsmath,amscd,amsthm}
\usepackage[top=3.2cm, bottom=3.2cm, left=3.4cm, right=3.4cm, includefoot]{geometry}
\usepackage[all]{xy}

\theoremstyle{plain}
\newtheorem*{theorem*}{Theorem}   
\newtheorem{theorem}{Theorem}[subsection]                    
\newtheorem{proposition}[theorem]{Proposition}            
\newtheorem{corollary}[theorem]{Corollary}                
\newtheorem{lemma}[theorem]{Lemma}

\theoremstyle{definition}
\newtheorem{remark}[theorem]{Remark} 

\newtheorem{example}[theorem]{Example}
\newtheorem{definition}[theorem]{Definition}

\newtheorem{notation}[theorem]{Notation}

\newcommand{\act}{\curvearrowright}
\newcommand{\hoplus}{\mathbin{\widehat{\oplus}}}
\newcommand{\toplus}{\mathbin{\widetilde{\oplus}}}

\newcommand{\botimes}{\mathbin{\bar{\otimes}}}

\newcommand{\bigtoplus}{\mathop{\widetilde{\bigoplus}}}

\newcommand{\op}{{\mathord{\rm op}}}  

\newcommand{\re}{{\mathord{\rm r}}}

\newcommand{\nuc}{{\mathord{\rm nuc}}}

\newcommand{\nor}{{\mathord{\rm nor}}}

\newcommand{\univ}{{\mathord{\rm univ}}}
\DeclareMathOperator{\id}{id}

\DeclareMathOperator{\Tr}{Tr}

\DeclareMathOperator{\Aut}{Aut}

\DeclareMathOperator{\lspan}{span} 
\DeclareMathOperator{\ospan}{\overline{\lspan}}

\DeclareMathOperator{\Imag}{Im} 

\DeclareMathOperator{\supp}{supp}

\DeclareMathOperator{\Cor}{Cor}
\DeclareMathOperator{\Corr}{Corr}
\DeclareMathOperator{\Rep}{Rep}
\DeclareMathOperator{\CP}{CP}
\DeclareMathOperator{\CCP}{CCP}
\DeclareMathOperator{\UCP}{UCP}

 \DeclareMathOperator{\Ball}{Ball}

\def\us#1_#2{\underset{#2}{#1}}
\def\os#1^#2{\overset{#2}{#1}}
 
\def\i<#1>{\langle #1 \rangle}
\def\l<#1>{\left\langle #1 \right\rangle} 
 
\makeatletter 
\@tempcnta\z@
\loop\ifnum\@tempcnta<26
\advance\@tempcnta\@ne
\expandafter\edef\csname f\@Alph\@tempcnta\endcsname{\noexpand\mathfrak{\@Alph\@tempcnta}}
\expandafter\edef\csname g\@alph\@tempcnta\endcsname{\noexpand\mathfrak{\@alph\@tempcnta}}
\expandafter\edef\csname l\@Alph\@tempcnta\endcsname{\noexpand\mathbb{\@Alph\@tempcnta}}
\expandafter\edef\csname c\@Alph\@tempcnta\endcsname{\noexpand\mathcal{\@Alph\@tempcnta}}
\expandafter\edef\csname r\@Alph\@tempcnta\endcsname{\noexpand\mathrm{\@Alph\@tempcnta}}
\repeat
\newcommand{\unit}{{(0)}}
\newcommand{\comp}{{(2)}}
\newcommand{\cpt}{{\mathord{\rm c}}} 
\DeclareMathOperator{\ev}{ev}
\allowdisplaybreaks[4]
\begin{document}
\title[Relative nuclearity for C$^*$-algebras]{Relative nuclearity for C$^*$-algebras and $KK$-equivalences of amalgamated free products}
\author{Kei Hasegawa}
\address{Graduate~School~of~Mathematics, Kyushu~University, Fukuoka~819-0395, Japan}
\email{ma213034@math.kyushu-u.ac.jp}
\begin{abstract}
We prove a relative analogue of equivalence between nuclearity and CPAP. In its proof, the notion of weak containment for C$^*$-correspondences plays an important role. As an application we prove $KK$-equivalence between full and reduced amalgamated free products of C$^*$-algebras under a strengthened variant of `relative nuclearity'.
\end{abstract}
\maketitle

\section{Introduction} \renewcommand{\thetheorem}{\Alph{theorem}}
A C$^*$-algebra $A$ is said to be {\it nuclear} if for any C$^*$-algebra $B$ there is a unique C$^*$-cross norm on the algebraic tensor product $A \odot B$.
The notion of nuclearity was introduced by Takesaki \cite{Takesaki} in the 1960's.
On the other hand,
by remarkable works of Lance \cite{Lance_JFA}, Choi--Effros \cite{Choi-Effros}, and Kirchberg \cite{Kirchberg}, the nuclearity of a given C$^*$-algebra is known to be equivalent to the {\it completely positive approximation property} ({\it CPAP\/}), that is, the identity map can be approximated by a net of completely contractive positive maps that factor through matrix algebras.
This characterization is useful and plays an important role in various situations.
In this paper, we study `relative counterparts' of nuclearity and CPAP for inclusions of C$^*$-algebras.

It seems natural to formulate the `relative CPAP' of a given inclusion $B \subset A$ of C$^*$-algebras by the following asymptotically commuting diagram:
\[
\xymatrix  @C=0.5cm @R=1.2cm
{
A
\ar[rr]^{\id} \ar[rd]_{\varphi_i} &&  A \\
& \lM_{n(i)} (B) \ar[ru]_{\psi_i}&
}
\]
Namely, $\varphi_i$ and $\psi_i$ are completely positive maps satisfying $a=\lim_i \psi_i \circ \varphi_i (a)$ for $a \in A$.
In this case, the CPAP of $B$ implies the one of $A$.
Moreover, there are many examples known to have this `relative CPAP'; for example,
crossed products by amenable discrete groups, group C$^*$-algebras of relative amenable discrete groups, tensor products with nuclear C$^*$-algebras, continuous fields of nuclear C$^*$-algebras, etc.
From the viewpoint of equivalence between nuclearity and CPAP, those inclusions should be `relatively nuclear', but the explicit formulation of `relative nuclearity' has never been established so far.
The aim of this paper is to formulate relative nuclearity for a given inclusion $B \subset A$ in such a way that it includes original nuclearity as a particular case when $B= \lC 1_A$, and is characterized by a kind of relative CPAP. 

\medskip
It is widely known that nuclearity is to C$^*$-algebras what amenability is to von Neumann algebras.
Thus, motivated by Popa's formulation \cite{Popa} of relative amenability for von Neumann algebras, we will develop our theory of `relative nuclearity' by the use of C$^*$-correspondences, which are C$^*$-counterparts of bimodules over von Neumann algebras.
For C$^*$-algebras $A$ and $B$, an $A$-$B$ C$^*$-correspondence is given by a pair $(X, \pi_X)$, where $X$ is a Hilbert $B$-module and $\pi_X$ is a $*$-homomorphism from $A$ into the C$^*$-algebra of adjointable (right $B$-linear) operators on $X$.
$A$-$A$ C$^*$-correspondences are also called C$^*$-correspondences over $A$.
To define `relative nuclearity', let us introduce the notion of {\it universal factorization property} ({\it UFP\/}).
We say that a C$^*$-correspondence $(X, \pi_X)$ over $A$ has the UFP if for any C$^*$-algebra $B$ and any $*$-representation $\sigma :A \otimes_{\max} B \to \lB (H)$,
$\sigma$ factors through the image of the natural representation $\phi_X^H :A \otimes_{\max} B \to \lB (X \otimes_B H)$ (see \S\S \ref{ss-rel-UFP} for its precise definition) as follows.
\[
\xymatrix @C=2cm
{
A \us\otimes_{\rm max} B
\ar[d]_{\phi_X^H} \ar[r]^{ \sigma} &   \lB( H) \\
\Imag \phi_X^H    \ar@{-->}[ru]
}
\]
In our theory, the notion of UFP plays a role of the original definition of nuclearity.
Indeed, the nuclearity of $A$ is naturally equivalent to the UFP of $(H \otimes A, \pi_H \otimes 1)$ for some faithful $*$-representation $\pi_H :A \to \lB (H)$.

For a given unital inclusion $B \subset A$ with conditional expectation $E$,
we denote by $(L^2(A, E), \pi_E)$ the $A$-$B$ C$^*$-correspondence associated with $E$ given by the GNS-construction.
For simplicity, let us assume that $E$ is nondegenerate (i.e., $\pi_E$ is faithful) in the rest of this section.
We say that $(A, B, E)$ is {\it nuclear} if $(L^2(A, E) \otimes_B A, \pi_E \otimes 1)$ has the UFP.
In the case when $B = \lC 1_A$, $L^2(A, E) $ is a Hilbert space, and hence the nuclearity of $(A, \lC 1_A, E)$ is equivalent to the nuclearity of $A$.
Moreover, we prove the following theorem, which is a relative analogue of `nuclearity $\Leftrightarrow$ CPAP'.

\begin{theorem}\label{thm-A}
Let $B \subset A$ be a unital inclusion of $\rC^*$-algebras with conditional expectation $E$.
Then, $(A, B, E)$ is nuclear if and only if for any finite subset $\fF \subset A$ and $\varepsilon >0$, there exist $n,m \in \lN$ and completely positive maps $\varphi_k : A \to \lM_n (B)$ and $\psi_k : \lM_n (B) \to A$, $1 \leq k \leq m$ such that $\| a - \sum_{k=1}^m \psi_k \circ \varphi_k (a) \| < \varepsilon$ for $a \in \fF$, and each $\varphi_k$ and $\psi_k$ are of the form
\[
\varphi_k : a \mapsto \left[ \rule{0pt}{10pt} E (x_i^* a x_j ) \right]_{i,j=1}^n
\quad
\psi_k : \left[ \rule{0pt}{0pt}\, b_{ij}\, \right]_{i,j=1}^n \mapsto \sum_{i,j=1}^n y_i^* b_{ij} y_j 
\]
for some $x_i, y_i \in A$, $1 \leq i \leq n$.
\end{theorem}

This theorem will be proved based on the following two observations concerning weak containment for C$^*$-correspondences:
The first one is that for given $A$-$B$ C$^*$-correspondences $(X, \pi_X)$ and $(Y, \pi_Y)$ with $A$ unital, the following are equivalent (see \S\S \ref{ss-weak-def} for the definition of weak containment):
\begin{itemize}
\item $(X, \pi_X)$ is weakly contained in $(Y, \pi_Y)$ with respect to the universal representation, written $(X, \pi_X) \prec_\univ (Y,\pi_Y)$.
\item For any $\xi \in X$, finite subset $\fF \subset A$ and $\varepsilon >0$, there exist $m \in \lN$ and $\eta_1, \dots, \eta_m \in Y$ such that $\| \i< \xi, \pi_X(a) \xi >  - \sum_{k=1}^m \i< \eta_k, \pi_Y(a) \eta_k> \| < \varepsilon$ for $a\in \fF$.
\end{itemize}
The second one is that any C$^*$-correspondence $(X, \pi_X)$ over $A$ has the UFP if and only if $(A, \lambda_A) \prec_\univ (X, \pi_X)$ holds, where $(A, \lambda_A)$ is the identity C$^*$-correspondence over $A$ (see \S\S \ref{ss-rel-UFP}).
These observations also say that the nuclearity of $(A, B, E)$ is characterized by the condition that $(A, \lambda_A) \prec_\univ (L^2(A, E) \otimes_B A,  \pi_E \otimes 1_A)$. 
We point out that this is parallel to Popa's formulation \cite{Popa} of relative amenability for von Neumann algebras: an inclusion of von Neumann algebra $N \subset M$ is amenable if and only if ${}_M L^2(M)_M \prec {}_M L^2(M) \otimes_N L^2(M)_M$ holds.

\medskip
We also introduce the notion of {\it strong relative nuclearity}.
Roughly speaking, the strong nuclearity of a given triple $(A, B, E)$ is the property that each $\psi_k \circ \varphi_k$ in Theorem \ref{thm-A} can be chosen to be $B$-bimodule maps (see \S\S \ref{ss-rel-rel}).
This stronger notion seems technical, but almost all the examples of nuclear triples investigated in this paper are in fact strongly nuclear.
For example, a triple $(A, B, E)$ is strongly nuclear in any of the following cases:
\begin{itemize}
\item $A$ is nuclear, $B$ is finite dimensional, and the embedding $B \hookrightarrow A$ is full (i.e., $\ospan A b A =A$ for $b \in B \setminus \{ 0 \}$);
\item $A = B \otimes C$ with $C$ nuclear;
\item $B =C(X) \subset A' \cap A$ and $A$ is a continuous field of nuclear C$^*$-algebras over $X$;
\item $E$ is of Watatani index finite type;
\item $A=B \rtimes_\alpha \Gamma$, where $\Gamma$ is a discrete amenable group;
\item $B=C(X)$ and $A =C(X) \rtimes_\alpha \Gamma$, where $\alpha : \Gamma \act C(X)$ is amenable (\cite{Delaroche}); 
\item $A=\rC^*_\re (\Gamma)$ and $B=\rC^*_\re (\Lambda)$, where $\Lambda \triangleleft \Gamma$ is co-amenable (i.e., $\Gamma / \Lambda$ is amenable);
\item $A=\rC^*_\re (\cG)$ and $B = C (\cG^\unit)$, where $\cG$ is a locally compact Hausdorff amenable \'{e}tale groupoid whose unit space $\cG^\unit$ is compact. 
\end{itemize}
We also show that a triple $(A, B, E)$ is nuclear in any of the following cases:
\begin{itemize}
\item $A$ is nuclear and the embedding $B  \hookrightarrow  A$ is full;
\item $A=\rC^*_\re (\Gamma)$ and $B=\rC^*_\re (\Lambda)$, where $\Lambda < \Gamma$ is co-amenable;
\item $E$ is of probabilistic index finite type;
\item $A=B \rtimes_\alpha \Gamma$, where $\alpha : \Gamma \act B$ is amenable.
\end{itemize}
Clearly, strong relative nuclearity implies relative nuclearity, but we do not know whether or not these two notions are actually different.

\medskip
As a (kind of) byproduct of our investigation of `relative nuclearity', we also obtain
Weyl--von Neumann--Voiculescu type results that partially generalize the ones due to Kasparov \cite{Kasparov} and Skandalis \cite{Skandalis}.
In particular, we prove the next characterization for weak containment (see Theorem \ref{thm-absorbing}).
\begin{theorem}\label{thm-B}
Let $A$ and $B$ be $\rC^*$-algebras with $A$ unital separable and $B$ $\sigma$-unital,
and $(X,\pi_X)$, $(Y,\pi_Y )$ be $A$-$B$ $\rC^*$-correspondences with $X$ countably generated and $\pi_X$ and $\pi_Y$ unital.
Then, $(X, \pi_X)\prec_\univ (Y,\pi_Y)$ if and only if $(X \oplus Y^\infty, \pi_X \oplus \pi_Y^\infty)$ and $(Y^\infty, \pi_Y^\infty)$ are approximately unitarily equivalent,
where $(Y^\infty, \pi_Y^\infty)$ is the countable infinite direct sum of $(Y, \pi_Y)$.
\end{theorem}

As applications, we obtain several results in $KK$-theory.
Firstly, we show that our strong relative nuclearity implies Germain's relative $K$-nuclearity \cite{Germain-fields}, which is a relative counterpart of Skandalis's $K$-nuclearity \cite{Skandalis}.
In \cite{Skandalis}, Skandalis proved that nuclearity implies $K$-nuclearity by using Kasparov's generalized Voiculescu theorem.
Similarly, we prove relative $K$-nuclearity by establishing a certain Weyl--von Neumann--Voiculescu type assertion under strong relative nuclearity (see \S\S \ref{ss-rel-K-K}). Then, we prove the following theorems:
\begin{theorem}\label{thm-C}
Let $\{(A_i, B, E_i) \}_{i \in \cI}$ be an at most countable family of unital inclusions of separable $\rC^*$-algebras $B \subset A_i$ with conditional expectations $E_i : A_i \to B$.
If each triple $(A_i, B, E_i)$ is strongly nuclear,
then the canonical surjection from
the full amalgamated free product $\bigstar_{B, i \in \cI} A_i$ onto the reduced one $\bigstar_{B, i \in \cI} (A_i, E_i)$ is a $KK$-equivalence.
\end{theorem}
\begin{theorem}\label{thm-D}
Let $\{(A_i, B, E_i) \}_{i \in \cI}$ be an at most countable family of unital inclusions of separable $\rC^*$-algebras $B \subset A_i$ with nondegenerate conditional expectations $E_i :A_i \to B$.
If each $A_i$ is nuclear and $B$ is finite dimensional, then the canonical surjection from $\bigstar_{B, i \in \cI} A_i$ onto $\bigstar_{B, i \in \cI} (A_i, E_i)$ is a $KK$-equivalence.
\end{theorem}
These two theorems follow from a somewhat more general and technical result (see \S\S \ref{ss-KK-CD}).
We note that Germain's result on free products of unital separable nuclear C$^*$-algebras \cite{Germain-duke} is a particular case of these theorems with $B=\lC$.
Finally, combing our result with Thomsen's result \cite{Thomsen} on $K$-theory of full amalgamated free products we obtain six term exact sequences in $K$-theory of reduced ones.

\medskip
This paper is organized as follows.
In \S \ref{sec-pre}, we recall basic facts on C$^*$-correspondences.
The definition of weak containment for C$^*$-correspondences is given in \S \ref{sec-weak}.
In that section, we also characterize weak containment by a certain approximation property for coefficients, mentioned above.
In \S \ref{sec-rel} we introduce the notion of UFP.
We then define relative nuclearity and strong one, and prove Theorem \ref{thm-A}. 
We also see that our relative nuclearity is related to relative WEP recently introduced by Jian and Sepideh \cite{Jian-Sepideh}, as well as relative amenability for von Neumann algebras (\cite{Popa}\cite{Delaroche2}\cite{Ozawa-Popa}).
In \S \ref{sec-exam}, we see that the examples listed above are actually (strongly) nuclear.
\S \ref{sec-voic} is devoted to proving Weyl--von Neumann--Voiculescu type results.
Applications in $KK$-theory are given in the last three sections.
In \S \ref{sec-rel-K} we prove that strong relative nuclearity implies relative $K$-nuclearity.
The proof of Theorem \ref{thm-C} and Theorem \ref{thm-D} are given in \S \ref{sec-KK}.
In the final section, we establish six term exact sequences in $KK$-theory of reduced amalgamated free products and a $KK$-equivalence result for HNN-extensions.

\subsection*{Notation}
We basically follow the notation of Brown and Ozawa's book \cite{Brown-Ozawa}.
For a C$^*$-algebra $A$ we denote by $1_A$ the unit of the multiplier algebra $\cM (A)$ of $A$.
The symbols $\lB (H)$ and $\lK (H)$ stand for the set of bounded operators and the set of compact ones on a Hilbert space $H$, respectively.
We use the symbol $\odot$ to denote the algebraic tensor product over $\lC$.
For C$^*$-algebras $A$ and $B$ we denote by $A\otimes B$ and $A\otimes_{\rm max} B$ the minimal and the maximal tensor products, respectively.
For a von Neumann algebra $M$, we denote by $M_*$ the (unique) predual of $M$.
The von Neumann algebraic tensor product of $M$ and another von Neumann algebra $N$ is denoted by $M \botimes N$.
We denote by $\CP (A,B), \CCP(A, B)$ and $\UCP(A, B)$ the c.p.\ (completely positive) maps, the c.c.p.\ (completely contractive positive) maps, and the u.c.p.\ (unital completely positive) maps from $A$ into $B$, respectively.
For a linear map $\varphi :  A \to B$ we denote by $\varphi^{(n)}$ the linear map $\varphi \otimes \id_{\lM_n } :\lM_n (A) \to \lM_n (B)$ with identification $\lM_n (A) = A \otimes \lM_n$, etc.
For these terminologies we refer the reader to \cite[Chapter 1-3]{Brown-Ozawa}.
For elements $x$ and $y$ in a normed space $X$ and $\varepsilon >0$ we write $x \approx_\varepsilon y$ when $\| x - y \| < \varepsilon$ holds.
The closed unit ball of $X$ is denoted by $\Ball (X)$.

\section*{Acknowledgment} The author wishes to express his gratitude to Professor Yoshimichi Ueda, who is his supervisor, for his continuous guidance and constant encouragement.
He is also grateful to Yuki Arano for many stimulating conversations.

\setcounter{theorem}{0}
\renewcommand{\thetheorem}{\arabic{section}.\arabic{theorem}}
\section{Preliminaries on C$^*$-correspondences}\label{sec-pre}
In this section, we fix notations and terminologies and recall basic facts on C$^*$-correspondences. 
We refer the reader to Lance's book \cite{Lance} for Hilbert C$^*$-module theory.
\begin{definition}
An {\it inner product} $A$-module is a linear space $X$ with a right $A$-action which is compatible with scalar multiplication, i.e., $\lambda (\xi  a)= (\lambda \xi ) a = \xi  (\lambda a)$ for $\lambda \in \lC, \xi \in X, a \in A$ and an $A$-valued inner product $\i< \cdot, \cdot > : X \times X \to A$ satisfying the following conditions:
\begin{itemize}
\item[(1)] $\i< \xi, \lambda \eta + \mu \zeta >= \lambda \i< \xi , \eta > + \mu \i< \xi, \zeta >$ for $\xi, \eta, \zeta \in X$ and $\lambda, \mu \in \lC$,
\item[(2)] $\i< \xi, \eta  a> = \i< \xi, \eta >a$ for $\xi, \eta \in X$ and $a\in A$,
\item[(3)] $\i< \xi, \eta >^*= \i< \eta, \xi >$ for $\xi, \eta \in X$,
\item[(4)] $\i< \xi, \xi >\geq 0$ for $\xi \in X$,
\item[(5)] $\xi =0$ if and only if $\i< \xi, \xi >=0$ for $\xi \in X$.
\end{itemize}
When $X$ is complete with respect to the norm $\| \xi \| = \| \i< \xi, \xi > \|^{1/2}$,
we call $X$ a {\it Hilbert $A$-module} or {\it Hilbert $\rC^*$-module over $A$}.
\end{definition}

Let $X$ be a Hilbert C$^*$-module over a C$^*$-algebra $A$.
When $A=\lC$, $X$ is a usual Hilbert space.
When we would like to emphasize the C$^*$-algebra $A$ of coefficients, we will write $\i< \cdot, \cdot >_A$.
A Hilbert $A$-module $X$ is said to be {\it countably generated} if there exists a countable subset $\{ \xi_n \}_{n=1}^\infty \subset X$ such that $\ospan \{ \xi_n a \mid a \in A, n \in \lN \} =X$.

Let $X$ and $Y$ be Hilbert $A$-modules.
A linear map $x : X \to Y$ is said to be {\it adjointable} if there exists a linear map $x^* :Y \to X$ which enjoys $\i< \eta, x \xi>= \i< x^* \eta, \xi>$ for all $\xi \in X, \eta \in Y$.
Note that adjointability implies $A$-linearity.
We denote by $\lL_A (X,Y)$ the set of adjointable maps from $X$ into $Y$ and set $\lL_A (X):=\lL_A (X,X)$.
Every adjointable map is automatically bounded and $\lL_A (X)$ forms a unital C$^*$-algebra with respect to the operator norm and the involution $x \mapsto x^*$.

For given vectors $\xi, \eta \in X$ we define the `rank one' operator $\theta_{\xi, \eta} \in \lL_A (X)$ by $\theta_{\xi, \eta} (\zeta ) =\xi  \i< \eta, \zeta>$.
We denote by $\lK_A (X)$ the C$^*$-subalgebra of $\lL_A (X)$ generated by $\{ \theta_{\xi, \eta} \mid \xi, \eta \in X \}$.
Operators in $\lK_A (X)$ are called compact operators on $X$.
It is known that $\lK_A (X)$ is a C$^*$-ideal of $\lL_A (X)$ and $\lL_A(X)$ is isomorphic to the multiplier algebra of $\lK_A(X)$. We denote by $1_X$ the identity operator on $X$. 

\begin{definition}
Let $A$ and $B$ be C$^*$-algebras.
An $A$-$B$ C$^*$-{\it correspondence} is a pair $(X, \pi_X)$ consisting of a Hilbert $B$-module $X$ and a $*$-homomorphism $\pi_X :A \to \lL_A (X)$, called the left action.
$A$-$A$ C$^*$-correspondences are also called {\it $\rC^*$-correspondences over} $A$.
We denote by $\Corr (A,B)$ the set of $A$-$B$ C$^*$-correspondences and set $\Corr (A):=\Corr(A,A)$.
\end{definition} 
A $\rC^*$-correspondence $(X, \pi_X) \in \Corr (A, B)$ is said to be {\it unital} if $A$ is unital and $\pi_X$ is also a unital map, and {\it countably generated} if $X$ is countably generated as a Hilbert $B$-module.
We denote by $\Rep (A)$ the set of nondegenerate $*$-representations of $A$.
$A$-$B$ C$^*$-correspondences $(X,\pi_X)$ and $(Y,\pi_Y)$ are said to be {\it unitarily equivalent}, denoted by $(X, \pi_X) \cong (Y, \pi_Y)$, if there exists a unitary $U \in \lL_B (X, Y)$ such that $\pi_X(a)= U^* \pi_Y (a) U$ for $a\in A$.
When no confusion may arise, we will write $X$ for short instead of $(X, \pi_X)$.
\begin{definition}\label{def-coeff}
For $X \in \Corr (A, B)$ and $\xi, \eta \in X$ the mapping $A \ni a \mapsto \i< \xi, \pi_X(a) \eta>$ is called a {\it coefficient} of $X$.
We define the c.p.\ map $\Omega_\xi : A \to B$ by $\Omega_\xi (a)=\i< \xi, \pi_X(a ) \xi >$.
For a subset $S \subset X$ we denote by $\cF_S$ the convex hull of $\{ \Omega_\xi \mid \xi \in S \}$ in $\CP (A,B)$.
\end{definition}

\begin{definition}\label{def-int-tensor}
Let $X$ and $Y$ be Hilbert C$^*$-modules over $A$ and $B$, respectively, and $\varphi : A \to \lL_B (Y)$ be a c.p.\ map.
Then we can construct the Hilbert $B$-module $X\otimes_\varphi Y$ by separation and completion of  $X\odot Y$ with respect to the $B$-valued semi-inner product (i.e., it satisfies the axiom of $B$-valued inner products except (5))
$
\i< \xi \otimes \eta, \xi' \otimes \eta' >:= \i< \eta, \varphi(\i< \xi, \xi'> ) \eta' >$ for  $\xi,\xi'\in X$ and $\eta,\eta' \in Y.$ 
There are two $*$-homomorphisms:
\begin{align*}
\lL_A (X) \to \lL_B (X\otimes_\varphi Y); \quad x \mapsto x \otimes 1_Y\\
\varphi(A)'\cap \lL_B (Y) \to \lL_B (X\otimes_\varphi Y); \quad y \mapsto 1_X \otimes y
\end{align*}
satisfying that $(x\otimes 1_Y) (\xi \otimes \eta) =( x\xi )\otimes \eta$ and $(1_X\otimes y)( \xi \otimes \eta) =\xi \otimes (y \eta),$
for $\xi \in X$ and $\eta \in Y$.
Since these $*$-homomorphisms have mutually commuting ranges,
we will write $x\otimes y: =(x\otimes 1_Y) (1_X\otimes y)=(1_X \otimes y)(x\otimes 1_Y)$.
When $\varphi$ is a $*$-homomorphism,
the module $X\otimes_{\varphi} Y$ is called the {\it interior tensor product} of $X$ and $(Y, \varphi)$.
When no confusion may arise,
we may write $X\otimes_B Y=X\otimes_{\varphi} Y$.
Further assume that $Y=B$ and $\varphi :A \to B$ is surjective.
In this case, $X \otimes_\varphi B$ is called the {\it pushout} of $X$ by $\varphi$ and denoted by $X_\varphi$.
We also write $x_\varphi:=x \otimes 1_B $ for $x \in \lL_A (X)$.
\end{definition}

\begin{definition}
Let $X$ and $Y$ be Hilbert C$^*$-modules over $C$ and $D$, respectively.
The {\it exterior tensor product} of $X$ and $Y$ is the Hilbert $C\otimes D$-module given by separation and completion of $X \odot Y$ with respect to the $C \otimes D$-valued semi-inner product $\i< \xi \otimes \eta , \xi' \otimes \eta'> =\i< \xi, \xi'> \otimes \i< \eta, \eta '> \in C \otimes D$ for $\xi, \xi ' \in X$ and $\eta, \eta' \in Y$.
It is known that there exists a $*$-homomorphism $\iota :\lL_C (X) \otimes \lL_D ( Y) \to \lL_{C \otimes D} (X\otimes Y)$ such that $\iota (x \otimes y) (\xi \otimes \eta)= x\xi \otimes y \eta$ for $x \in \lL_C (X)$, $y \in \lL_D (Y)$, $\xi \in X$, $\eta \in Y$.
We will write $\iota (x\otimes y) = x\otimes y$ for short. 
We note that when $(X, \pi_X) \in \Corr (A,C)$ and $(Y, \pi_Y) \in \Corr (B,D)$, we have $(X \otimes Y, \pi_X \otimes \pi_Y) \in \Corr (A \otimes B, C\otimes D)$. 
\end{definition}

\begin{remark}
Let $X \in \Corr (A, B)$ and $Y \in \Corr (B, C)$ be given.
To simplify the notation, we use the same symbol $\pi_X \otimes 1_Y$ for the $*$-homomorphisms from $A$ into $\lL_C (X \otimes_B Y)$ and $\lL_{B \otimes C} (X \otimes Y)$.
\end{remark}

\begin{example}\label{ex-id}
Every C$^*$-algebra $A$ forms a Hilbert $A$-module with respect to the inner product $\i< a, b> =a^*b$.
It is not hard to see that $A\cong \lK_A (A)$.
Let $\lambda_A :A \to \lL_A (A)$ the canonical $*$-homomorphism given by the left multiplication.
The $(A, \lambda_A) \in \Corr (A)$ is called the {\it identity $\rC^*$-correspondence over} $A$.
\end{example}
\begin{example}
Let $\pi_H :A \to \lB (H)$ be a $*$-representation.
The C$^*$-correspondence $(H \otimes B, \pi_H \otimes 1_B)$ is called a {\it scalar representation}.
\end{example}

\begin{example}\label{ex-cond}
Let $B \subset A$ be an inclusion of C$^*$-algebras with a conditional expectation $E$ from $A$ onto $B$.
Then, $A$ naturally forms a right $B$-module by right multiplication.
We denote by $L^2 (A, E)$ the Hilbert $B$-module obtained from $A$ by separation and completion with respect to the $B$-valued semi-inner product $\i< a, b > := E(a^* b)$ for $a, b \in A$.
The left action $\pi_E : A \to \lL_B (L^2 (A,E))$ is given by the left multiplication.
When $A$ is unital, we denote by $\xi_E$ be the vector in $L^2(A, E)$ corresponding to $1_A$.
The triple $(L^2(A,E), \pi_E, \xi_E)$ is called {\it the GNS representation associated with} $E$.
The conditional expectation $E$ is said to be {\it nondegenerate} when $\pi_E$ is injective (or, equivalently, $a=0$ if and only if $E(xay)=0$ for all $x,y \in A$).
The conditional expectation $E$ is also said to be {\it faithful} when for any $a\in A$, $a=0$ if and only if $E(a^*a)=0$.
\end{example}
Let $\{\delta_i \}_{i=1}^n$ be the standard basis of $\lC^n$ and $\{ e_{ij} \}_{i,j=1}^n$ be the corresponding system of matrix units in $\lM_n$.
We denote by $\lC_n$ the Hilbert $\lM_n$-module $\lC^n$ equipped with the right action $\lC^n \times \lM_n \ni ( \xi, x) \mapsto {}^t x \xi \in \lC^n$, where ${}^t x$ is the transposed matrix of $x$, and the $\lM_n$-valued inner product defined by $\i<\delta_i, \delta_j>=e_{ij}$ for $1 \leq i,j \leq n$.
\begin{example}\label{ex-tensor}
We call $H_A:=\ell^2(\lN) \otimes A$ {\it the standard Hilbert module over} $A$.
Clearly, $H_A$ is isomorphic to the infinite direct sum $\bigoplus_{n=1}^\infty A$ of $A$ as Hilbert $A$-module.
For $(X, \pi_X) \in \Corr (A,B)$ and $n\in \lN$ we set
$(X^\infty, \pi_X^\infty) := (\ell^2(\lN) \otimes X, 1_{\ell^2(\lN)} \otimes \pi_X)$ and $(X^n, \pi_X^n):=(\lC^n \otimes X, 1_{\lC^n} \otimes \pi_X)$.
We also define $(X_n, \pi_{X_n}):= (X \otimes \lC_n, \pi_X \otimes 1_{\lC_n} ) \in \Corr (A, \lM_n (B) )$.
\end{example}
The next observation is standard, but important for us since it illustrates how c.p.\ maps that factors through matrix algebras (over C$^*$-algebras) appear.
\begin{remark}\label{rem-factor}
Let $(X, \pi_X) \in \Corr (A, B)$ and $(Y, \pi_Y) \in \Corr (B, C)$ be given.
For a given vector $\zeta \in X \otimes_B Y$ of the form $\sum_{i=1}^n \xi_i \otimes \eta_i$,
the coefficient $\Omega_\zeta : A \to C$ is equal to the composition of the c.p.\ maps $\varphi : A \to \lM_n (B)$ and $\psi : \lM_n (B) \to C$ defined by
\[
\varphi : a \mapsto \left[ \rule{0pt}{10pt} \i< \xi_i, \pi_X( a) \xi_j >_B \right]_{i,j=1}^n, \quad 
\psi : \left[ \rule{0pt}{0pt}\, b_{ij}\, \right]_{i,j=1}^n \mapsto \sum_{i,j=1}^n \i< \eta_i, \pi_Y ( b_{ij} ) \eta_j >_C.
\]
\end{remark}

\renewcommand{\thetheorem}{\arabic{section}.\arabic{subsection}.\arabic{theorem}}

%
%
\section{Weak containment for C$^*$-correspondences}\label{sec-weak}
In this section we develop some general theory of weak containment for C$^*$-correspondences.
\subsection{Weak containment with respect to representations}\label{ss-weak-def}
\begin{definition}
Let $A$ be a C$^*$-algebra and $(H, \pi_H)$ and $(K, \pi_K)$ be $*$-representations of $A$.
We say that $(H,\pi_H)$ is {\it weakly contained in} $(K, \pi_K)$, written $(H, \pi_H) \prec (K,\pi_K)$ if $\ker \pi_K \subset \ker \pi_H$.
When no confusion may arise, we may write $H \prec K$ or $\pi_H \prec \pi_K$ for short.
\end{definition}
\begin{definition}\label{def-weak}
Let $A$ and $B$ be C$^*$-algebras.
For $(X,\pi_X) \in \Corr(A,B)$ and $(H, \pi_H) \in \Rep (B)$
we define the $*$-representation $\theta_X^H : A \otimes_{\rm max} \pi_H(B)' \to \lB (X\otimes_B H)$ by
\[
\theta_X^H (a\otimes x):= \pi_X(a)\otimes x, \quad a\in A, x\in \pi_H(B)'.
\]
We say that $(X, \pi_X )$ is {\it weakly contained in} $(Y,\pi_Y)  \in \Corr(A,B)$ {\it with respect to} $(H,\pi_H)$, written $(X, \pi_X) \prec_{(H,\pi_H)} (Y,\pi_Y)$, if $(X\otimes_B H, \theta^H_X)$ is weakly contained in $(Y \otimes_B H, \theta^H_Y)$.
When no confusion may arise, we will write $X\prec_H Y$ for short.
In the case that $(H,\pi_H)$ is the universal representation of $B$, we write $X \prec_\univ Y$.
\end{definition}

\begin{remark}
The reader may think our definition of weak containment rather technical.
Hence, we will briefly explain why we formulated it as above.
Let $A$ and $B$ be C$^*$-algebras.
Fix $H \in \Rep (B)$ arbitrarily and set $M:=\pi_H(B)''$.
Then, weak containment for $A$-$B$ C$^*$-correspondences with respect to $H$ can be characterized by the one for corresponding $A$-$M$ bimodules in the following way.
Let $X, Y \in \Corr (A, B)$ be arbitrary.
Thanks to Lemma \ref{lem-normalrepn} below,
we can assume that $M \subset \lB (H)$ is of standard form (see, e.g., \cite{Haagerup}).
Then, the commutant $\pi_H(B)'$ is canonically isomorphic to the opposite algebra $M^\op$, and $\theta_X^H$ and $\theta_Y^H$ factor through the (right) normal tensor product $A \otimes_\nor M^\op$ (see \cite{Effros-Lance}).
Let $\rho_X$ and $\rho_Y$ be the representations of $A \otimes_\nor M^\op$ corresponding to $\theta_X^H$ and $\theta_Y^H$, respectively.
Then, it is clear that $\theta_X^H \prec \theta_Y^H$ if and only if $\rho_X \prec \rho_Y$.
We also note that this observation says that, in the case when $B=M$ and $X$ and $Y$ are selfdual,
our definition agrees with the one defined by Anantharaman-Delaroche and Havet \cite[Definition 1.7]{Delaroche-Havet}.
\end{remark}

\begin{remark}\label{rem-faithful}
Let $X, Y \in \Corr (A, B)$ and $(H, \pi) \in \Rep (B)$ be arbitrary.
Recall that the pushout of $X$ by $\pi$ is the Hilbert $\pi (B)$-module  $X_\pi$.
Since $X \otimes_B H \cong X_\pi \otimes_{\pi(B)} H$, it follows that $X \prec_H Y$ if and only if $X_\pi \prec_H Y_\pi$ as $A$-$\pi(B)$ C$^*$-correspondences.
Thanks to this observation, we can reduce to the case when $(H, \pi_H)$ is faithful in some cases. 
\end{remark}


\begin{lemma}[cf.\ {\cite[Lemma 3.8.4]{Brown-Ozawa}}]\label{lem-normalrepn}
Let $A$ and $B$ be $\rC^*$-algebras, and $X,Y \in \Corr(A,B)$ and $(H, \pi_H) \in \Rep(B)$ be given. Set $M=\pi_H(B)''$.
If $X \prec_H Y$ and $(K,\pi_K)$ is a normal representation of $M$,
then it follows that $X \prec_{(K,\pi_K\circ \pi_H)} Y$.
\end{lemma}
\begin{proof}
First we deal with the case that $K= H \otimes G$ and
$\pi_K : M \to \lB (H \otimes G); x\mapsto x\otimes 1_G$ for a Hilbert space $G$.
Fix $\sum_{k=1}^m a_k \otimes x_k \in A\odot \pi_K(M)'$ arbitrarily.
Let $P \in \lB (G)$ be an orthogonal projection of rank $n$.
Note that $\pi_K(M)' = M'\bar{\otimes}\lB (G)$ and $(1 \otimes P) \pi_K (M)' (1 \otimes P) \cong M' \otimes \lM_n$.
Let $\{ e_{ij} \}_{i,j=1}^n$ be a system of matrix units $\lM_n$ and $\sum_{i,j=1}^n x^{(k)}_{ij}\otimes e_{ij} \in M' \otimes \lM_n$ be the matrix representation of $(1\otimes P)x_k(1\otimes P)$ via the above isomorphism.
Since $\lB (X \otimes_B (H \otimes PG) ) \cong \lB (( X \otimes_B H) \otimes \lC^n ) \cong \lB (X \otimes_B H) \otimes\lM_n$, we have
\begin{align*}
& \left\| ( 1_X \otimes (1_H \otimes P) ) \sum_{k=1}^m \theta_X^{K}(a_k \otimes x_k)  ( 1_X \otimes (1_H \otimes P) )  \right\|
_{\lB (X\otimes_B K)} \\
&\qquad =
\left\|  \sum_{k=1}^m \theta_X^{K}(a_k \otimes ( (1_H \otimes P ) x_k ( 1_H \otimes P )  ) )  \right\|
_{\lB (X\otimes_B (H \otimes PG) )} \\
&\qquad =
\left\| \sum_{k=1}^m \sum_{i,j=1}^n \theta_X^H \left( a_k \otimes x_{ij}^{(k)} \right)  \otimes e_{ij} \right\|
_{\lB (X\otimes_B H) \otimes \lM_n } \\
&\qquad =
\left\| (\theta_X^{H} \otimes \id_{\lM_n} ) \left( \sum_{k=1}^m \sum_{i,j=1}^n ( a_k \otimes x^{(k)}_{ij}) \otimes e_{ij} \right) \right\|
_{\lB (X\otimes_B H) \otimes \lM_n } \\
&\qquad \leq
\left\| (\theta_Y^{H } \otimes \id_{\lM_N} )  \left( \sum_{k=1}^m \sum_{i,j=1}^n ( a_k \otimes x^{(k)}_{ij}) \otimes e_{ij} \right) \right\|
_{\lB (Y \otimes_B H) \otimes \lM_n } \\
&\qquad \leq
\left\| \sum_{k=1}^m \theta_Y^{K}(a_k \otimes x_k)  \right\|
_{\lB (Y \otimes_B K)}.
\end{align*}
Let $\{P_i \}_i \subset \lB (G)$ be a net of finite rank projections converges to $1_{G}$ strongly.
Then $\{1_X \otimes (1_H \otimes P_i) \}_i $ also converges to $1_{X\otimes K}$ strongly.
By the lower semi-continuity of operator norm we have
$\left\| \sum_{k=1}^m \theta_X^{K}(a_k \otimes x_k)  \right\|
\leq
\left\| \sum_{k=1}^m \theta_Y^{K}(a_k \otimes x_k)  \right\|$,
and hence we have $\theta_X^H \prec \theta_Y^H$.

Since every normal representation of $M$ is the cut-down of $\pi_K$ above by some projection in $\pi_K (M)'$, we are done.
\end{proof}

As we mentioned above, the following definition includes \cite[Definition 1.7]{Delaroche-Havet}.
\begin{definition}
Let $A$ be a C$^*$-algebra and $M$ be a von Neumann algebra.
For any two C$^*$-correspondences $X,Y \in \Corr (A,M)$ we say that $X$ is {\it weakly contained in} $Y$, written $X \prec Y$, if $X$ is weakly contained in $Y$ with respect to any (or some) faithful normal representations of $M$.
\end{definition}

\subsection{Characterization in terms of coefficients}\label{ss-weak-coe}
In this subsection, we prove Theorem \ref{thm-weak} below, which contains \cite[Proposition 2.3]{Delaroche-Havet} as a particular case that $B=M$ and $X$ is selfdual.
The proof below is based on the same idea as in Kirchberg's proof \cite{Kirchberg} for showing that C$^*$-nuclearity implies CPAP (see also \cite[Theorem 3.8.5]{Brown-Ozawa}). 

\begin{theorem}\label{thm-weak}
Let $A$ and $B$ be $\rC^*$-algebras with $A$ unital.
Let $(X,\pi_X), (Y, \pi_Y) \in \Corr (A,B)$ and $(H, \pi_H ) \in \Rep (B)$ be given and
set $M:=\pi_H(B)''$.
Then, the following are equivalent:
\begin{itemize}
\item[$(1)$] $(X,\pi_X) \prec_{(H,\pi_H)} (Y,\pi_Y)$.
\item[$(2)$] For any $\xi \in X$ there exists a net of c.p.maps $\{\psi_i \}_i$ in $\cF_Y$ (see Definition \ref{def-coeff}) such that $\pi_H \circ \psi_i$ converges to $\pi_H \circ \Omega_\xi$ in the point $\sigma$-weak topology.
\item[$(3)$] For any $\xi \in X$ there exists a net of c.p.maps $\{\psi_i \}_i$ in $\cF_Y$ such that $\psi(1_A) \leq \Omega_\xi (1_A)$ and $\pi_H \circ \psi_i$ converges to $\pi_H \circ \Omega_\xi$ in the point $\sigma$-weak topology.
\item[$(4)$] $X\otimes_B M \prec Y \otimes_B M$.
\end{itemize}
\end{theorem}

The following technical lemmas originate in \cite[Lemma 2.2]{Delaroche-Havet} and are used to prove the implication (1) $\Rightarrow$ (2) in Theorem \ref{thm-weak}.

\begin{lemma}\label{prop-CS}
If $A$ and $B$ are unital $\rC^*$-algebras, $\varphi : A\to B$ is completely positive, and $f$ be a state on $B$,
then for any $a \in A$ and $b,c \in B$ it follows that
\[
|f( b^*\varphi(a) c)| \leq \min \left\{ |f(b^* \varphi(1)b )|^{1/2} |f(c^*\varphi (a^*a)c )|^{1/2},
f(b^* \varphi(aa^*)b )|^{1/2} |f(c^*\varphi (1)c )|^{1/2} \right\}.
\] 
\end{lemma}
\begin{proof}
Consider the $A$-$B$ C$^*$-correspondence $X = A\otimes_\varphi B$.
Then we have
$
f( b^*\varphi(a) c) = f (b^* \i< 1 \otimes 1, a\otimes 1 > c ) = f ( \i< 1 \otimes b, a\otimes c >).
$
Since $X\times X \ni (\xi, \eta) \mapsto f ( \i< \xi, \eta >) \in \lC$ defines a  sesquilinear form, by the Cauchy--Schwarz inequality, we have
\[
|f ( \i< 1 \otimes b, a\otimes c >) |\leq |f( \i< 1\otimes b, 1 \otimes b >)|^{1/2} |f( \i< a  \otimes c, a \otimes c >)|^{1/2} = | f(b^* \varphi(1)b ) |^{1/2} | f(c^* \varphi(a^*a)c ) |^{1/2}.
\]
Since $f ( \i< 1 \otimes b,  a\otimes c >) = f ( \i< a^* \otimes b, 1\otimes c > )$ holds,
by the Cauchy--Schwartz inequality again, we get the desired inequality. 
\end{proof}

\begin{lemma}
Let $A$ and $C$ be unital $\rC^*$-algebras and $\varphi : A \to C$ be a u.c.p.\ map.
Fix a faithful $*$-representation $C \subset \lB (H)$.
Let $\{ \phi_i \}_{i \in \cI}$ be a net in $\CP (A, C)$ which converges to $\varphi$ in the point $\sigma$-strong topology on $\CP (A, \lB (H))$ and set $c_i:= 2(1 + \phi_i (1_A))^{-1}$.
Then, $\phi'_i : A \to C; a \mapsto c_i \phi_i (a) c_i$ converges to $\varphi$ in the point $\sigma$-weak topology and satisfies that $\phi'_i (1_A) \leq 1_C$.
\end{lemma}
\begin{proof}
We first note that $\phi'_i (1_A) =4 \phi_i (1_A)(1 + \phi_i (1_A))^{-2} \leq 1$. 
Let $\fF \subset A$ be a finite subset, $\cX \subset M_*$ be a finite subset of normal states,
and $\varepsilon>0$ be arbitrarily chosen.
By assumption, there exists $i_0 \in \cI$ such that
\[
| f  ( \varphi (a) - \phi_i (a) ) | < \varepsilon /2, \quad  2 \| a \| |f ( (1- \phi_i (1_A))^2 ) |^{1/2} < \varepsilon/2, \quad | f (\phi_i (1_A)) - 1 | <1 
\]
for all $a \in \fF$ and $f \in \cX$ as long as $i >i_0$.
By the previous lemma we have
\begin{align*}
| f (\phi'_i(a) - \phi_i(a) )| 
&=
| f ( c_i^*  \phi_i (a) c_i  - \phi_i(a) ) | \\
&\leq
| f (c_i^* \phi_i(a) (c_i -1_C ) ) | + | f ((c_i^* -1_C ) \phi_i (a)) | \\
&\leq
 |f(\phi_i (1_A) (1_C - c_i)^2)| ^{1/2}( | f (c_i^* \phi_i (aa^*) c_i)|^{1/2}+ |f (\phi_i (a^*a))|^{1/2}) \\
&\leq
|f(\phi_i (1_A) (1_C -c_i)^2)| ^{1/2} ( \| a \| + \|a \| f (\phi_i (1_A))^{1/2} )\\
&\leq  2 \| a \| |f(\phi_i (1_A) (1_C-c_i)^2)| ^{1/2} .
\end{align*}
Since $\phi_i (1_A) (1_C-c_i)^2=(1_C-\phi_i (1_A))^2 (1_C+\phi_i (1_A))^{-2} \phi_i (1_A) \leq (1_C-\phi_i (1_A))^2$,
we have $|f(\phi_i (1_A) (1_C-c_i)^2)| \leq | f ((1_C-\phi_i (1_A))^2) |.$
Hence, we get
$| f(\varphi (a) - \phi'_i(a)) | \leq | f (\varphi (a) - \phi_i (a) )| + | f (\phi_i(a) - \phi'_i (a)) | <  | f (\varphi (a) - \phi_i (a) )| + 2 \| a \| |f ( (1_C- \phi_i (1_A))^2 ) |^{1/2} < \varepsilon$ as long as $i>i_0$.
\end{proof}

The next lemma specialized to the case that $B=B''$ is exactly \cite[Lemma 2.2]{Delaroche-Havet}.
\begin{lemma}\label{lem-cone}
Let $A$ and $B$ be $\rC^*$-algebras with $A$ unital.
Fix a nondegenerate faithful $*$-representation $B \subset \lB (H)$.
Let $\cF$ be a convex subset of $\CP(A,B)$ such that for any $b\in B$ and $\psi \in \cF$, the c.p.\ map $b^*\psi( \cdot )b$ also belongs to $\cF$.
If $\varphi \in \CP(A,B)$ belongs to the point $\sigma$-weak closure of $\cF$ in $\CP (A,\lB (H) )$,
then there exists a net $\{ \psi_i \}_i$ in $\cF$ such that $\psi_i(1_A) \leq \varphi(1_A)$ and it converges to $\varphi$ in the point $\sigma$-weak topology.
\end{lemma}
\begin{proof}
Set $b:=\varphi (1_A)$.
For each $n \in \lN$, we define the continuous function $f_n \in C_0 (0, \| b \| ]$ by
\[
f_n(x):=
\begin{cases}
n \sqrt{2x} & 0 \leq x \leq \frac{1}{2n},\\
(x + \frac{1}{2n} )^{-1/2} & \text{otherwise},
\end{cases}
\]
and set $b_n:= f_n (b) \in B$.
Let $e \in B''$ be the support projection of $b$ and set $C:=e B''e$.
Define $\theta_n \in \CP (A,B)$ by $\theta_n (a) := b_n \varphi (a) b_n$ for $a \in A$.
We first claim that $\theta_n$ converges to a u.c.p.\ map $\phi$ from $A$ into $C$ in the point $\sigma$-strong topology.
To see this, we fix a positive contraction $a \in A$ arbitrarily.
Since we have $0 \leq \varphi(a) \leq b$,
by the Douglas decomposition theorem (see e.g.\ \cite[Theorem 17.1]{Conway}),
there exists $c \in \lB (H)$ such that $\varphi(a)^{1/2} =c b^{1/2}$.
Since $\{ b_n b^{1/2} \}_n$ is an increasing sequence which converges $e$ strongly,
$\theta_n (a) = b_n b^{1/2} c^*c b^{1/2} b_n$ also converges to $e c^*c e$ in the same topology.
In the case when $a = 1_A$, the $c$ is equal to $e$, and hence $\phi (a):= e c^* c e$ defines the desired u.c.p.\ map.

We note that $\varphi = b^{1/2}\phi(\cdot) b^{1/2}$ holds.
Indeed, for any $a \in A$ we have
\[\varphi (a)=e \varphi (a) e =\lim_n  b^{1/2}b_n \varphi (a) b_n b^{1/2} = \lim_n b^{1/2} \theta_n (a) b^{1/2} = b^{1/2}\phi (a) b^{1/2}.
\]
Thus, if we find a net $\phi_i'$ in $\cF$ in such a way that $\phi_i' (1_A) \leq \phi (1_A)=e$ and $\phi_i' (a)$ converges to $\phi (a)$ $\sigma$-weakly for $a \in A$, then $\psi_i:= b^{1/2} \phi_i'(\cdot) b^{1/2}$ gives the desired net.

Since $\cF$ is convex, the point $\sigma$-weak closure of $\cF$ in $\CP (A, \lB (H) )$ coincides with the point $\sigma$-strong closure of $\cF$.
This follows from the fact that for any finitely many elements $a_1, \dots, a_n \in A$ the set $\{ (\psi (a_1), \dots, \psi (a_n) )\in  \bigoplus_{i=1}^n \lB (H)  \mid \psi \in \cF \}$ is convex, and hence its $\sigma$-weak and $\sigma$-strong closures coincide.
Thus,
by the claim above, we can find nets $\varphi_i \in \cF$ and $n (i) \in \lN$ in such a way that $\phi_i:= b_{n(i)} \varphi_i (\cdot) b_{n(i)}$ converges to $\phi$ point $\sigma$-strongly.
The image of each $\phi_i$ is contained in $B \cap C$ since $b_n = e b_n$ holds for every $n \in \lN$.
By the preceding lemma, the net $\phi_i' =c_i \phi_i (\cdot) c_i $ with $c_i= 2 ( e+ \phi_i (1_A) )^{-1} \in C$ converges to $\phi$ in the point $\sigma$-weak topology and satisfies that $\phi_i' (1_A) \leq e$.
We show that $\phi_i'$ belongs to $\cF$.
Indeed, we have $ \phi'_i  = c_i \phi_i (\cdot) c_i  = c_i b_{n(i)} \varphi_i (\cdot) b_{n(i)}c_i $.
By the fact that $c_i$ is the norm limit of commutative polynomials of $e$ and $b_{n(i)} \varphi_i (1_A) b_{n(i)}$ and $b_{n(i)}e=b_{n(i)}$,
each $c_i b_{n(i)}$ is in $B$.
By the assumption of $\cF$, we get $\phi_i' \in \cF$.
\end{proof}

\begin{proof}[Proof of Theorem \ref{thm-weak}]
As in Remark \ref{rem-faithful}, replacing $B$ and $X, Y$ by $\pi_H (B)$ and $X_{\pi_H}, Y_{\pi_Y}$ we may assume that $\pi_H$ is faithful and identify $B$ with $\pi_H (B)$.

We prove $(1) \Rightarrow (2):$
Fix $\xi \in X$, a finite subset $\fF \subset A$, a finite subset $\cX$ of normal states on $M$ and $\varepsilon >0$ arbitrarily.
Set $f:= | \cX |^{-1} \sum_{g \in \cX} g$ and let $(K, \pi_f, \xi_f)$ be the GNS-representation associated with $f$.
Since each $g \in \cX$ enjoys $g \leq |\cX| f$,
by the Radon--Nikodym theorem for states (see e.g. \cite[Proposition 3.8.3]{Brown-Ozawa}), there exists $x_g \in \pi_f(M)' = \pi_f (B)'$ such that $g (y) = \i< \xi_f, \pi_f (y) x_g \xi_f >$ for $y\in M$.
By Lemma \ref{lem-normalrepn} we have $X \prec_K Y$, i.e., $\theta_X^K \prec \theta_Y^K$.
Since the image of $Y \odot \xi_f$ is dense in $Y \otimes_B K$, applying Fell's characterization of weak containment \cite[Theorem 1.2]{Fell} to these $*$-representations of $A \otimes_{\max} \pi_f (B)'$ and the vector $\xi \otimes \xi_f \in X \otimes_B K$,
we find $\eta_1, \dots, \eta_n \in Y$ such that $| \i< \xi \otimes \xi_f, \theta_X^K (a \otimes x_g ) (\xi \otimes \xi_f) > - \sum_{i=1}^n \i< \eta_i \otimes \xi_f,  \theta_Y^K (a \otimes x_g )  (\eta_i \otimes \xi_f)> | < \varepsilon$ for all $a\in \fF$ and $g \in \cX$.
We then have for $a\in \fF$ and $g \in \cX$
\begin{align*} 
g \left(  \sum_{i=1}^n  \i< \eta_i, \pi_Y (a) \eta_i >_B \right)
&=\sum_{i=1}^n \i<  \xi_f,  \pi_f ( \i< \eta_i,  \pi_Y (a) \eta_i >_B ) x_g \xi_f > \\
&=\sum_{i=1}^n \i<  \eta_i \otimes \xi_f,  \pi_Y (a) \eta_i \otimes x_g \xi_f > \\
&=\sum_{i=1}^n \i<  \eta_i \otimes \xi_f,  \theta_Y^K (a \otimes x_g ) ( \eta_i \otimes \xi_f) >.
\end{align*}
Since $g ( \Omega_\xi (a) )  = \i< \xi_f, \pi_f  ( \i< \xi,  \pi_X(a) \xi >_B ) x_g \xi_f > = \i< \xi \otimes \xi_f,  \theta_X^K (a \otimes x_g )(\xi \otimes \xi_f) >$,
the c.p.\ map $\psi:=\sum_{k=1}^m \Omega_{\zeta_k} \in \cF_Y$ satisfies that $| g  (\Omega_\xi (a) -  \psi (a)) | < \varepsilon$ for $a \in \fF$ and $g \in \cX$.

The implication (2) $\Rightarrow$ (3) follows from Lemma \ref{lem-cone}.
We prove (3) $\Rightarrow$ (1):
Let $z \in \ker \theta^H_Y$ be given and show that $\theta^H_X (z)=0$.
It suffices to prove that $\theta^H_X (z) \xi \otimes \eta=0$ for every $\xi \in X$ and $\eta \in H$.
We fix $\varepsilon >0$ arbitrarily and take $w=\sum_{i=1}^n a_i \otimes x_i \in A \odot \pi_H(B)'$ in such a way that $\| z - w  \|_{\rm max} < \varepsilon $.
By (3) we can choose $\psi = \sum_{k=1}^m \Omega_{\zeta_k} \in \cF_Y$ in such a way that $\psi (1) \leq \Omega_\xi (1)$ and $|\i< \eta, \sum_{i,j=1}^n  \Omega_\xi (a_i^* a_j) x_i^* x_j \eta > -
 \i< \eta, \sum_{i,j=1}^n \psi (a_i^* a_j)  x_i^* x_j \eta >| < \varepsilon^2$.
Note that $ \sum_{k=1}^m \|  \zeta_k \otimes \eta \|^2 = \i< \eta, \psi(1) \eta > \leq \i< \eta,  \Omega_\xi (1) \eta > \leq \| \xi \otimes \eta \|^2$.
One has
\begin{align*}
&\| \sum_{i=1}^n \theta^H_X (a_i \otimes x_i) \xi \otimes \eta \|^2
=
\i< \eta, \sum_{i,j=1}^n\Omega_\xi (a_i^* a_j)  x_i^* x_j \eta >
\approx_{\varepsilon^2}
 \i< \eta, \sum_{i,j=1}^n \sum_{k=1}^m \Omega_{\zeta_k} (a_i^* a_j) x_i^* x_j \eta >\\
 &
 =
 \sum_{k=1}^m  \| \sum_{i=1}^n \theta^H_Y (a_i \otimes x_i) \zeta_k \otimes \eta \|^2 
 \leq
\varepsilon^2 \sum_{k=1}^m  \| \zeta_k \otimes \eta \|^2
\leq
 \varepsilon^2 \| \xi\otimes \eta \|^2.
\end{align*}
Thus, we have $\| \theta^H_X (z)(\xi \otimes \eta ) \| \leq \| z - w \|_{\rm max} + \| \theta^H_X (w) ( \xi \otimes \eta )\| \leq \varepsilon + \sqrt{2}\varepsilon$.
Since $\varepsilon$ is arbitrary, we get $\theta^H_X (z)=0$.

Finally, the canonical isomorphisms $(Y \otimes_B M) \otimes_M H \cong Y \otimes_B H$ and $(X \otimes_B M) \otimes_M H \cong X\otimes_B H$ imply that $\ker \theta_Y^H = \ker \theta_{Y\otimes_B M}^H $ and $\ker \theta_{X}^H= \ker \theta_{X \otimes_B M}^H$,
which proves $(1) \Leftrightarrow (4)$.
\end{proof}

\begin{corollary}\label{cor-weak}
Let $A$ and $B$ be $\rC^*$-algebras with $A$ unital and $X,Y \in \Corr(A,B)$ be given.
The following are equivalent:
\begin{itemize}
\item[$(1)$] $X \prec_\univ Y$.
\item[$(2)$] For any $\xi \in X$ the c.p.\ map $\Omega_\xi$ belongs to the point norm closure of $\cF_Y$.
\item[$(3)$] For any $\xi \in X$ there exists a net $\{ \psi_i \}_i$ in $\cF_Y$ such that $\psi_i(1_A) \leq  \Omega_\xi (1_A)$ and that $\lim_i \| \Omega_\xi (a) - \psi_i (a) \| =0$ for every $a\in A$.
\end{itemize}
\end{corollary}
\begin{proof}
Let $(H_u, \pi_u)$ be the universal representation of $B$.
Then, it is known that the enveloping von Neumann algebra $\pi_u (B)''$ is isomorphic to the second dual $B^{**}$.
Hence, the relative topology on $\pi_u(B)$ induced from the $\sigma$-weak topology on $\pi_u (B)''$ coincides with the weak topology.
Since $\cF_Y$ is convex, the point weak closure and the point norm closure of $\cF$ coincide (see e.g.\ \cite[Lemma 2.3.4]{Brown-Ozawa}).
Hence, the assertion follows from the previous theorem.  
\end{proof}
%
%
\subsection{Elementary properties of weak containment}\label{ss-weak-prop}
In this subsection we establish some basic facts on weak containment.
For a Hilbert C$^*$-module over a unital C$^*$-algebra $B$, a vector $\xi \in X$ is said to be {\it normal} if $\i< \xi, \xi > =1_B$ holds.
We note that every C$^*$-correspondence $(A \otimes_\varphi B, \lambda_A \otimes 1_B)$ arising from a u.c.p.\ map $\varphi : A \to B$ admits the normal vector $1_A \otimes 1_B$.
\begin{proposition}\label{prop-unital}
Let $A$ and $B$ be unital $\rC^*$-algebras and $X,Y \in \Corr (A, B)$ and $H \in \Rep (B)$ be given.
Assume that $Y$ is unital and admits a normal vector $\eta \in Y$.
If $X \prec_K Y$ (resp. $X \prec_\univ Y$),
then for any $\xi \in X$ there exists $\psi_i \in \cF_Y$ with $\psi_i (1_A) =\Omega_\xi(1_A)$ such that $\pi_K \circ \psi_i $ converges to $\pi_K \circ \Omega_\xi$ point $\sigma$-weakly (resp. $\psi_i$ converges to $\Omega_\xi$ in the point norm topology).
\end{proposition}
\begin{proof}
By Theorem \ref{thm-weak}, there exists a net $\psi_i \in \cF_Y$ with $\psi_i (1_A) \leq \Omega_\xi (1_A)$ such that $\pi_K \circ \psi_i$ approaches $\pi_K \circ \Omega_\xi$.
Set $b_i:=\Omega_\xi (1_A) - \psi_i (1_A) \in B$ and $\zeta_i:=\eta b_i^{1/2}$.
For $a\in A$ and normal state $f \in \lB (K)_*$ we have
$f( \Omega_{\zeta_i} (a) )= f (b_i^{1/2}\i< \eta, \pi_Y(a) \eta >b_i^{1/2}  ) \leq \| a \| f (b_i) \to 0$ and $\psi_i (1_A) + \Omega_{\zeta_i} (1_A)=\Omega_\xi(1_A)$. 
Thus $\psi_i + \Omega_{\zeta_i} \in \cF_Y$ is the desired one.
The assertion in the case that $X \prec_\univ Y$ follows from Corollary \ref{cor-weak}.
\end{proof}

Let $A$ and $B$ be C$^*$-algebras with $A$ unital, and let $X, Y, Z \in \Corr (A, B)$ and $H \in \Rep (B)$.
We note that $X \prec_H Z$ holds whenever $X \prec_H Y$ and $Y \prec_H Z$ hold thanks to Theorem \ref{thm-weak}.
\begin{proposition}\label{prop-tensor}
Let $A, B, C$ and $D$ be $\rC^*$-algebras with $A$ and $C$ unital.
For $X, Y \in \Corr (A, B)$, $Z, W \in \Corr (C, D)$, $H \in \Rep (B)$ and $K \in \Rep (D)$ if $X \prec_H Y$ and $Z \prec_K W$ hold,
then we have $X \otimes Z \prec_{H \otimes K} Y \otimes W$ as $(A\otimes C)$-$(B\otimes D)$ $\rC^*$-correspondences.
\end{proposition}
\begin{proof}
By the remark above, it suffices to show the case when $(Z, \pi_Z)=(W, \pi_W)$.
We show that $\ker \theta_{Y \otimes Z}^{H \otimes H} \subset \ker \theta_{X \otimes Z}^{H \otimes K}$.
Fix $z \in \ker \theta_{Y \otimes Z}^{H \otimes H}$ with $\| z \| \leq 1$ arbitrarily.
By the polarization trick, we only have to show that $\i<\xi \otimes \zeta \otimes h \otimes k,  \theta_{X \otimes Z}^{H \otimes H} (z) \xi \otimes \zeta \otimes h \otimes k > =0$ for all unit vectors $\xi \in X$, $\zeta \in Z$, $h \in H$, and $k \in K$.
For any $\varepsilon >0$, we can find a contraction $z_\varepsilon=\sum_l (a_l \otimes c_i ) \otimes x_l \in A \odot C \odot (\pi_H(B)' \botimes \pi_K (D)' )$ satisfying
$\| z -   \sum_l (a_l \otimes c_i ) \otimes x_l \| < \varepsilon$ in $( A \otimes C) \otimes_{\max} ( \pi_H(B)' \botimes \pi_K (D)' )$.
Note that $\| \theta_{Y \otimes Z}^{H \otimes K} ( z_\varepsilon ) \| < \varepsilon$.
Since $A$ is unital and $X \prec_H Y$ holds, Theorem \ref{thm-weak} enables us to find $\eta_1, \dots, \eta_p \in Y$ in such a way that $\sum_{r=1}^p \i< \eta_r, \eta_r > \leq \i< \xi, \xi>$ and
\[
 \left|
\sum_l  \l< h \otimes k, x_l ( \pi_H ( \Omega_\xi (a_l)  - \sum_{r=1}^p \Omega_{\eta_r} (a_l)) h ) \otimes (\pi_K ( \Omega_\zeta (c_l) ) k ) > 
\right|
< \varepsilon.
\]
Now we have
\begin{align*}
& | \i<\xi \otimes \zeta \otimes h \otimes k,  \theta_{X \otimes Z}^{H \otimes H} (z) \xi \otimes \zeta \otimes h \otimes k > |
\\ &\qquad \approx_{\varepsilon} 
| \i<\xi \otimes \zeta \otimes h \otimes k,  \theta_{X \otimes Z}^{H \otimes H} (z_\varepsilon) \xi \otimes \zeta \otimes h \otimes k >
\\ &\qquad =_{\phantom{\varepsilon}}
|\sum_l  \i< h \otimes k, x_l ( \pi_H ( \Omega_\xi (a_l) ) h ) \otimes (\pi_K ( \Omega_\zeta (c_l) ) k ) >  |
\\ &\qquad \approx_\varepsilon
| \sum_{r=1}^p  \sum_l  \i< h \otimes k, x_l ( \pi_H ( \Omega_{\eta_r} (a_l) ) h ) \otimes (\pi_K ( \Omega_\zeta (c_l) ) k ) > |
\\ &\qquad =_{\phantom{\varepsilon}}
 |\sum_{r=1}^p  \i< \eta_r \otimes \zeta \otimes h \otimes k, \theta_{Y \otimes Z}^{H \otimes K} ( z_\varepsilon ) \eta_r \otimes \zeta \otimes h \otimes k >
\\ &\qquad \leq_{\phantom{\varepsilon}}
\| \theta_{Y \otimes Z}^{H \otimes K} (z_\varepsilon ) \| \i< h \otimes k, \sum_{r=1}^p \i< \eta_r, \eta_r > h \otimes \i< \zeta, \zeta > k >
\\ &\qquad \leq _{\phantom{\varepsilon}}
\varepsilon \i< h \otimes k, \i< \xi, \xi> h \otimes \i< \zeta, \zeta > k >
\\ &\qquad \leq _{\phantom{\varepsilon}}
\varepsilon.
\end{align*}
Since $\varepsilon$ is arbitrary, we are done.
\end{proof}

The next useful proposition is a particular case of Proposition \ref{prop-tensor}.
\begin{proposition}\label{lem-tensor}
Let $A,B$ be $\rC^*$-algebras with $A$ unital and $X, Y \in \Cor (A, B)$ and $H \in \Rep (B)$.
If $X \prec_H Y$, then $(X_n, \pi_{X_n} ) \prec_{(H^n, \pi_H^{(n)})} (Y_n, \pi_{Y_n} )$ as $A$-$\lM_n(B)$ $\rC^*$-correspondences (see Example \ref{ex-tensor} for notations).
\end{proposition}

\begin{proposition}\label{prop-int-tensor}
Let $A, B$ and $C$ be $\rC^*$-algebras with $A$ and $B$ unital and $X, Y \in \Corr (A, B)$, $Z, W \in \Corr (B, C)$ and $K \in \Rep (C)$ be given.
Suppose that $Z\prec_K W$ and either $X \prec_{Z \otimes_C K} Y$ or $X \prec_{W \otimes_C K} Y$ holds.
Then $X \otimes_B Z \prec_{K} Y \otimes_B W$ as $A$-$C$ $\rC^*$-correspondences.
\end{proposition}
\begin{proof}
Suppose that $X \prec_{Z\otimes_C K} Y$.
We show that $X\otimes_B Z \prec_K Y \otimes_B Z \prec_K Y \otimes_B W$.
For any $\sum_i a_i \otimes x_i \in A \odot \pi_K(C)'$ we have
$\| \theta_{X\otimes_B Z}^K (\sum_i a _i \otimes x_i) \| = \| \theta_X^{Z\otimes_C K} (\sum_i a_i \otimes (1_Z \otimes x_i ) )\| \leq \|\theta_Y^{Z\otimes_C K} (\sum_i a_i \otimes (1_Z \otimes x_i ) )\| =\| \theta_{Y\otimes_B Z}^K (\sum_i a _i \otimes x_i) \|$.
Hence $X\otimes_B Z \prec_K Y \otimes_B Z$.
To see $Y \otimes_B Z \prec_K Y \otimes_B W$, let $\xi= \sum_{i=1}^n \eta_i \otimes \zeta_i \in Y \odot Z$ be arbitrary.
By Remark \ref{rem-factor},
the coefficient $\Omega_\xi$ is of the form $\Omega_\zeta \circ \Omega_\eta$ with $\eta=(\eta_1, \dots, \eta_n ) \in Y_n$ and $\zeta=(\zeta_1, \dots, \zeta_n ) \in Z^n$, where $\lM_n (B)$ acts on $Z^n$ by $\pi_Z \otimes \id_{\lM_n} : \lM_n (B) \to \lL_C (Z) \otimes \lM_n = \lL_ C( Z^n)$.

Since $(Z^n, \pi_Z \otimes \id_{\lM_n}) \prec_H (W^n, \pi_{W}\otimes \id_{\lM_n} )$ as $\lM_n (B)$-$C$ C$^*$-correspondences,
there exists a net $\psi_i \in \cF_{Z^n}$ such that $\pi_K \circ \psi_i \circ \Omega_\eta $ converges to $\pi_K \circ \Omega_\zeta \circ \Omega_\eta= \pi_K \circ \Omega_\xi$ in the point $\sigma$-weak topology.
Since $\psi_i \circ \Omega_\eta \in \cF_{X \otimes_B Z}$, we get $Y\otimes_B Z \prec_K Y \otimes_B W$.

When $X \prec_{W \otimes_C K} Y$, we can prove, in the same manner, that $X\otimes_B Z \prec_K X \otimes_B W \prec_K Y \otimes_B W$
\end{proof}

\begin{proposition}\label{prop-product-map}
Let $A$ and $B$ be $\rC^*$-algebras with $A$ unital.
Let $X, Y \in \Corr (A, B)$ and $H \in \Rep (B)$ be arbitrary and set $M=\pi_H (B)''$.
Then, $X \prec_H Y$ if and only if for any $\xi \in X$, the product map
$$
(\pi_H \circ \Omega_\xi ) \us \times_{\max} \iota_{M'} : A \us\otimes_{\max} M' \to \lB (H); \quad a \otimes x \mapsto \pi_H (\Omega_\xi (a) ) x
$$
factors through $\Imag \theta_Y^H$.
\end{proposition}
\begin{proof}
Thanks to Remark \ref{rem-faithful} we assume that $\pi_H$ is faithful and $B \subset \lB (H)$.
Suppose that $X \prec_H Y$.
For any $\sum_i a_i \otimes x_i \in A \odot M'$, $\xi \in X$, and $\eta, \zeta \in H$ one has
\begin{align*}
| \i< \eta, \sum_i \Omega_\xi (a_i) x_i \zeta >| 
&= | \i< \xi \otimes \eta, \theta_X^H (\sum_i a_i \otimes x_i  ) \xi \otimes \zeta >|
\leq \| \theta_Y^H (\sum_i a_i \otimes x_i  ) \| | \xi \otimes \eta \| \| \xi \otimes \zeta \| \\ 
& \leq \| \theta_Y^H (\sum_i a_i \otimes x_i  ) \| \| \xi \|^2 \| \eta \| \| \zeta \| ,
\end{align*}
which implies that $ \theta_Y ^H ( \sum_i a_i \otimes x_i ) \mapsto \sum_i \Omega_\xi (a_i) x_i$ is bounded and its norm is less than or equal to $\| \xi \|^2$.

Conversely, suppose that $\Phi: \Imag \theta_Y^H \to \lB (H); \theta_Y^H (a \otimes x)  \mapsto \Omega_\xi (a) x$ is bounded.
Let $z \in \ker \theta_Y^H$ be arbitrarily fixed.
We show that $\i< \xi \otimes \eta, \theta_X^H (z) \xi' \otimes \eta' > =0$ for all $\xi, \xi ' \in X$ and $\eta, \eta' \in H$.
By the polarization trick we may assume that $\xi = \xi'$.
Let $\varepsilon >0$ be arbitrary and take $\sum_i a_i \otimes x_i \in A \odot M'$ in such a way that $\| z - \sum_i a_i \otimes x_i \|_{\max} < \varepsilon$.
We then have
\begin{align*}
| \i< \xi \otimes \eta, \theta_X^H (\sum_i a_i \otimes x_i ) \xi \otimes \eta' >|
&=| \i< \eta, \sum_i \i< \xi, \pi_X( a_i) \xi > x_i \eta' > |
\leq \i< \eta, \Phi (\sum_i a_i \otimes x_i  ) \eta'> | \\
& \leq  \|\Phi \| \| \theta_Y^H (\sum_i a_i \otimes x_i   ) \| \| \eta \| \| \eta ' \| 
< \varepsilon  \|\Phi \| \| \eta \| \| \eta ' \|.
\end{align*}
Since $\varepsilon >0$ is arbitrary, we get $\theta_X^H (z)=0$, and hence $X \prec_H Y$.
\end{proof}

The next technical proposition will be used later.
The proof below is based on \cite[Proposition 3.6.5]{Brown-Ozawa}, called {\it The Trick}.

\begin{proposition}\label{prop-product-map2}
Let $A$ and $B$ be unital $\rC^*$-algebras, and $X,Y \in \Corr(A, B)$  and $H \in \Rep (B)$.
If $\pi_Y$ is unital and $X \prec_H Y$ holds, then for any normal vector $\xi \in X$ there exists a u.c.p.\ map $\Theta$ from $( \lL_B (Y) \otimes 1_H)'' \subset \lB (Y \otimes_B H)$ into $\pi_H (B)''$ such that $\Theta (\pi_Y (a) \otimes 1_H) = \pi_H \circ \Omega_\xi (a)$ for $a\in A$.
\end{proposition}
\begin{proof}
Thanks to the previous proposition there exists a u.c.p.\ map $\Phi : \Imag \theta_Y^H \to \lB (H)$ such that $\Phi ( \theta_Y^H ( a \otimes x) ) = \pi_H ( \i< \xi, \pi_X( a) \xi > )x$ for $ a\in A $ and $\xi \in X$.
Since $\Imag \theta_Y^H$ is a unital C$^*$-subalgebra of $\lB (Y \otimes_B H)$,
by Arveson's extension theorem we can extend $\Phi$ to a u.c.p.\ map $\Psi$ from $\lB (Y \otimes_B H)$ into $\lB (H)$.
Note that $\Imag \theta_Y^H$ is contained in the multiplicative domain of $\Psi$, i.e., $\Psi (abc) = \Phi (a) \Psi (b) \Phi(c)$ holds for $a, c \in \Imag \theta_Y^H$ and $b \in \lB (Y \otimes_B H)$.
Let $\Theta$ be the restriction of $\Psi$ to $(\lL_B (Y) \otimes 1_H )''$.
Then, for $x \in (\lL_B (Y) \otimes 1_H )''$ and $y \in M'$ we have $y\Theta (x) =\Phi ( \theta_Y^H (1_A \otimes y )) \Psi (x) = \Psi ( \theta_Y^H (1_A \otimes y )x) = \Psi (x \theta_Y^H (1_A \otimes y ))=\Psi (x) \Phi ( \theta_Y^H (1_A \otimes y ))  =\Theta (x) y$, which implies $\Theta (x) \in \pi_H(B)''$.
\end{proof}

\section{Relative nuclearity}\label{sec-rel}
In this section we give the definition of relative nuclearity and prove Theorem \ref{thm-A}.
\subsection{Universal factorization property}\label{ss-rel-UFP}
To define relative nuclearity, we need the notion of universal factorization property, which plays a role of the original definition of nuclearity.
Recall the fact that for $\rC^*$-algebras $A$ and $B$ every nondegenerate representation $\sigma : A \odot B \to \lB (H)$ is of the form of $\sigma_A \times \sigma_B$, where $\sigma_A : A \to \lB (H)$ and $\sigma_B : B \to \lB (H)$ are unique $*$-homomorphisms having mutually commuting ranges and satisfying $\sigma (a \otimes b ) =\sigma_A (a) \sigma_B (b)$ (see e.g., \cite[Theorem 3.2.6]{Brown-Ozawa}).
\begin{definition}\label{def-UFP}
Let $A$ be a $\rC^*$-algebra.
We say that $(X,\pi_X)\in \Corr(A)$ has the {\it universal factorization property} ({\it UFP\/} for short) if it satisfies the following universal property:
For any $\rC^*$-algebra $B$, every $*$-representation $(H,\sigma) \in \Rep (A \otimes_{\max} B)$ is weakly contained in the $*$-representation $\phi_X^H: A \otimes_{\max} B \ni a \otimes b \mapsto \pi_X(a) \otimes \sigma_B (b) \in \lB (X\otimes _{\sigma_A} H)$.
In other words, there exists $*$-homomorphism $\Phi : \Imag \phi_X^H \to \lB( H)$ such that the following diagram commutes:
\[
\xymatrix @C=2cm
{
A \us\otimes_{\rm max} B
\ar[d]_{\phi_X^H} \ar[r]^{ \sigma} &   \lB( H) \\
\Imag \phi_X^H    \ar@{-->}[ru]_\Phi
}
\]
\end{definition}
Firstly, the UFP of a given C$^*$-correspondence is characterized in terms of weak containment.
\begin{proposition}\label{prop-UFP}
For any $\rC^*$-algebra $A$ and any $\rC^*$-correspondence $(X, \pi_X)$ over $A$,
$(X, \pi_X)$ has the UFP if and only if $(A,\lambda_A) \prec_\univ (X, \pi_X)$ holds.
\end{proposition}

\begin{proof}
Suppose that $(X, \pi_X)$ has the UFP.
Let $(H, \pi_H)$ be the universal representation of $A$.
Applying the UFP to $\pi_H \times \iota : A \odot \pi_H(A)' \to \lB (H)$,
where $\iota : \pi_H (A)' \hookrightarrow \lB (H)$ is the inclusion map, we get
$\ker \theta_X^H = \ker \phi_X^H \subset \ker (\pi_H \times_{\rm max} \iota) = \ker \theta_A^H$.
Conversely, suppose that $(A, \lambda_A) \prec_\univ  (X, \pi_X)$ holds.
Let $B$ and $\sigma_A \times \sigma_B : A\odot B \to \lB (H)$ be given.
The condition that $(A, \lambda_A) \prec_\univ (X, \pi_X )$ implies
that $(A, \lambda_A) \prec_{(H, \sigma_A)} (X, \pi_X)$.
For any $\sum_i a_i \otimes b_i \in A \odot B $ we have
$ \| \sigma_A \times \sigma_B ( \sum_i a_i \otimes b_i ) \| =\| \theta_A^H (\sum_i a_i \otimes \sigma_B (b_i) ) \| \leq \| \theta_X^H ( \sum_i a_i \otimes \sigma_B (b_i) \| = \|\phi_X^H ( \sum_i a_i \otimes b_i ) \|$,
which implies $(H, \sigma) \prec (X \otimes_B H, \phi_X^H )$.
\end{proof}

We next translate the original definition of nuclearity into the language of C$^*$-correspondences by use of the notion of UFP.
We note that equivalence between nuclearity and CPAP follows from the next proposition together with Corollary \ref{cor-weak}.
\begin{proposition}\label{prop-nuclear}
Let $A$ be a $\rC^*$-algebra and $\pi_H : A \to \lB (H)$ be a faithful $*$-representation.
Then, the following are equivalent:
\begin{itemize}
\item[(1)] $A$ is nuclear.
\item[(2)] $(H \otimes A, \pi_H \otimes 1_A)$ has the UFP.
\item[(3)] $(A, \lambda_A) \prec_\univ (H \otimes A, \pi_H \otimes 1_A)$.
\end{itemize}
\end{proposition}
\begin{proof}
We prove (1) $\Rightarrow$ (2):
Suppose that $A$ is nuclear and fix a C$^*$-algebra $B$ and nondegenerate $*$-representation $\sigma: A \odot B \to \lB (K)$ arbitrarily.
We observe that $(H \otimes A )\otimes _{\sigma_A} K \cong H \otimes K$, and this isomorphism induces $\Imag \phi_{K \otimes A}^H \cong A \otimes \sigma_B (B) \subset \lB (K \otimes H)$.
The nuclearity of $A$ implies that $\Imag \phi_{K \otimes A}^H \cong A \otimes_{\max} \sigma_B (B)$.
By the universality of the maximal tensor product, the mapping $\Phi : \Imag \phi_{K \otimes A}^H \ni \phi_{K \otimes A}^H (a \otimes b ) \mapsto \sigma_A (a) \sigma_B (b) \in \lB (H)$ is bounded,
and hence $(H \otimes A, \pi_H \otimes 1_A)$ has the UFP.

Equivalence between (2) and (3) follows from the previous proposition.
We prove (2) $\Rightarrow$ (1):
Suppose that $(H \otimes A, \pi_H \otimes 1_A)$ has the UFP.
Let $\sigma : A\otimes_{\max} B \to \lB (H)$ be a faithful $*$-representation.
Since $\Imag \phi_{H \otimes A}^K \cong A \otimes \sigma_B(B)$ holds as above,
the UFP gives the inverse of the canonical surjection from $A\otimes_{\max} B$ onto $A \otimes B$, and hence $A$ is nuclear.
\end{proof}

\subsection{Relative nuclearity}\label{ss-rel-rel}

\begin{definition}\label{def-rel-nuc}
Let $B \subset A$ be an inclusion of $\rC^*$-algebras with conditional expectation $E :A \to B$.
We say that the triple $(A, B, E)$ is {\it nuclear via} $(Z,\pi_Z) \in \Corr (B)$ if the $\rC^*$-correspondence $(L^2(A,E) \otimes_B Z \otimes_B A, \pi_E \otimes 1_Z \otimes 1_A)$ has the UFP.
When $(Z, \pi_Z)=(B, \lambda_B)$, we say that $(A, B, E)$ is {\it nuclear}.
\end{definition}
Let $B \subset A$ be an inclusion of C$^*$-algebras and $E:A \to B$ be a nondegenerate conditional expectation.
As we will see in the next section, $(A, B, E)$ is nuclear whenever $A$ is nuclear and the embedding $B \hookrightarrow A$ is full (see Example \ref{ex-nuc}).
We do not know whether or not this still holds true when we remove the assumption of fullness,
but we can prove the nuclearity of $(A, B, E)$ via some C$^*$-correspondences over $B$.
This is the merit of considering nuclearity via C$^*$-correspondences over subalgebras.
We also mention that the `$(Z, \pi_Z)$' does not affect much in some cases.
For example, in the case when $B= \lC 1_A$, $L^2(A,E) \otimes_B Z$ is a usual Hilbert space.
Thus, the nuclearity of $(A, \lC 1_A, E)$ via some C$^*$-correspondence over $\lC$ is equivalent to the one of $A$ by Proposition \ref{prop-nuclear}.
Moreover, in \S\S \ref{ss-rel-WFP} and \S\S \ref{ss-rel-ame} we will see that the nuclearity of $(A, B, E)$ via some C$^*$-correspondence over $B$ implies the relative injectivity of $\pi_H (B)'' \subset \pi_H (A)''$ for any $(H, \pi_H) \in \Rep (A)$, and a relative weak expectation property of A.

The next theorem says that this `via version' of relative nuclearity is characterized by a kind of `relative CPAP'.
\begin{theorem}\label{thm-nuc-CPAP}
Let $B \subset A$ be a unital inclusion of $\rC^*$-algebras with conditional expectation $E$, and $(Z, \pi_Z)$ be a $\rC^*$-correspondence over $B$.
Then, $(A, B, E)$ is nuclear via $(Z, \pi_Z)$ if and only if for any finite subset $\fF \subset A$ and $\varepsilon >0$, there exist $n,m \in \lN$, $\varphi_k : A \to \lM_n (B)$ and $\psi_k : \lM_n (B) \to A$, $1 \leq k \leq m$ such that $\| a - \sum_{k=1}^m \psi_k \circ \varphi_k (a) \| < \varepsilon$ for $a \in \fF$, and each $\varphi_k$ and $\psi_k$ are of the form
\[
\varphi_k : a \mapsto \left[ \rule{0pt}{10pt} \i< \eta_i, (\pi_E(a) \otimes 1_Z) \eta_j > \right]_{i,j=1}^n
\quad
\psi_k : \left[ \rule{0pt}{0pt}\, b_{ij}\, \right]_{i,j=1}^n \mapsto \sum_{i,j=1}^n y_i^* b_{ij} y_j 
\]
for some $\eta_1, \dots, \eta_n \in L^2(A, E) \otimes_B Z$ and $y_1, \dots, y_n \in A$.
\end{theorem}
\begin{proof}
Suppose that $(A, B, E)$ is nuclear via $(Z, \pi_Z)$
and fix a finite subset $\fF \subset A$ and $\varepsilon >0$ arbitrarily.
By Proposition \ref{prop-UFP} and Proposition \ref{prop-nuclear},
we have $(A, \lambda_A) \prec_\univ (L^2(A,E) \otimes_B Z \otimes_B A, \pi_E \otimes 1_Z \otimes 1_A)$.
Since the identity map on $A$ is noting but $\Omega_{1_A} \in \cF_A$,
thanks to Theorem \ref{thm-weak}, we can find $m \in \lN$ and $\xi_1, \dots, \xi_m \in L^2(A, E) \otimes_B  Z \otimes_B A$ in such a way that $\| a - \sum_{k=1}^m \Omega_{\xi_k} (a ) \| < \varepsilon$ for all $a\in \fF$.
Here we may assume that each $\xi_k$ is of the form $\sum_{i=1}^n \eta^{(k)}_i \otimes y^{(k)}_i \in (L^2(A, E) \otimes_B Z ) \odot A$.
By Remark \ref{rem-factor}, letting $\varphi_k (a) := \left[ \i< \eta^{(k)}_i, (\pi_E (a) \otimes 1) \eta^{(k)}_j > \right]_{i,j=1}^n$ for $a \in A$ and $\psi_k ([ b_{ij}]_{ij=1}^n ):=\sum_{i,j=1}^n y_i^{(k)*} b_{ij} y_j^{(k)}$ for $[b_{ij}]_{i,j=1}^n \in \lM_n (B)$,
we get $\| a - \sum_{k=1}^m \psi_k \circ \varphi_k (a) \| < \varepsilon$ for all $a \in \fF$.
The converse implication follows from Proposition \ref{prop-UFP}, Proposition \ref{prop-nuclear}, and Corollary \ref{cor-weak} again.
\end{proof}

\begin{proof}[Proof of Theorem \ref{thm-A}]
Suppose that $(A, B, E)$ is nuclear and fix a finite subset $\fF \subset A$ and $\varepsilon >0$ arbitrarily.
By the preceding theorem, we can find $n, m \in \lN$ and c.p.\ maps $\varphi_k : A \to \lM_n (B), \psi_k : \lM_n (B ) \to A$ satisfying that $\| a - \sum_{k=1}^m \psi_k \circ \varphi_k (a) \| < \varepsilon$ for $a \in A$.
We only have to modify $\varphi_k$'s.
For each $k$, $\varphi_k$ is of the form $a \mapsto  [ \i< \eta_i, \pi_E(a) \eta_j > ]_{i,j=1}^n$ for some $\eta_1, \dots, \eta_n \in L^2(A, E)$.
Since $L^2(A,E)$ is the completion of $A$, we may assume that $\eta_i$ comes from an element $x_i \in A$.
Since $\i< x_i, \pi_E( a) x_j > = E (x_i^* a x_j)$ holds, we are done.
\end{proof}
The next proposition implies that if both $(A, B, E)$ and $B$ are nuclear, then so is $A$. 
\begin{proposition}\label{prop-nuc-exact}
Let $A$ and $B$ be $\rC^*$-algebras with $A$ unital and $B$ nuclear.
If there exist $X \in \Corr (A, B)$ and $Y \in \Corr (B, A)$ such that $(X \otimes_B Y, \pi_X \otimes 1_Y)  \in \Corr (A)$ has the UFP, then $A$ is nuclear.
\end{proposition}
\begin{proof}
Since $(X \otimes_B Y, \pi_X \otimes 1_Y)$ has the UFP,
by Proposition \ref{prop-UFP} we have $(A, \lambda_A) \prec_\univ (X \otimes_B Y, \pi_X \otimes 1_Y) $. 
If $(H, \pi_H) \in \Rep (B)$ is a faithful representation,
then thanks to Proposition \ref{prop-nuclear} we have $(B, \lambda_B) \prec_\univ ( H \otimes B, \pi_H \otimes 1_A).$
Proposition \ref{prop-int-tensor} implies that $(X \otimes_B Y, \pi_X \otimes 1_Y) \prec_\univ ( X \otimes_B H \otimes Y, \pi_X \otimes 1_H \otimes 1_Y)$.
Since $( X \otimes_B H \otimes Y, \pi_X \otimes 1_H \otimes 1_Y)$ is a scalar representation of $A$, this implies the nuclearity of $A$.
\end{proof}

We next introduce the notion of strong relative nuclearity.
Let $B\subset A$ be an inclusion of $\rC^*$-algebras and $(X,\pi_X)$ be a $\rC^*$-correspondence over $A$.
A vector $\xi \in X$ is said to be $B${\it -central} if $\xi$ enjoys $\pi_X(b) \xi = \xi  b$ for all $b \in B$.
We denote by $B' \cap X$ the set of $B$-central vectors in $X$.
For the identity C$^*$-correspondence $(A, \pi_A) \in \Corr (A)$ and $B\subset A$, the set of $B$-central vectors is nothing but the relative commutant $B' \cap A$.
We also note that every c.p.\ map in $\cF_{B'\cap X}$ forms a $B$-bimodule map.

\begin{definition}\label{def-CCPAP}
Let $B \subset A$ be an inclusions of $\rC^*$-algebras.
\begin{itemize}
\item We say that a $\rC^*$-correspondence $(X,\pi_X)$ over $A$ has the {\it $B$-central completely positive approximate property} ($B$-{\it CCPAP\/} for short) if there exists a net of c.c.p.\ maps $\psi_i \in \cF_{B'\cap X}$ such that $\lim_{i} \| a - \psi_i (a) \| =0$ for every $a\in A$.
\item Let $E :A \to B$ be a conditional expectation. We say that the triple $(A, B,E)$ is {\it strongly nuclear via} $(Z, \pi_Z) \in \Corr (B)$ if $(L^2(A,E) \otimes_BZ \otimes_B A, \pi_E \otimes 1_Z \otimes 1_A)$ has the $B$-CCPAP.
When $(Z, \pi_Z) =(B, \lambda_B)$, we say that $(A, B, E)$ is {\it strongly nuclear}.
\end{itemize}
\end{definition}
For a unital $A$ the $B$-CCPAP of a $\rC^*$-correspondence over $A$ implies the UFP.
Thus, every strongly nuclear triple $(A, B, E)$ is nuclear, but we do not know whether or not the converse is true.
%
%

\subsection{Relative WEP}\label{ss-rel-WFP}
We next discuss relative weak expectation property recently introduced by Jian and Sepideh \cite{Jian-Sepideh} in relation with our relative nuclearity.
\begin{definition}[{\cite[Proposition 3.3.6]{Brown-Ozawa}}] 
An inclusion $B\subset A$ is said to be {\it relatively weakly injective} if the following equivalent conditions hold;
\begin{itemize}
\item[(1)] there exists a c.c.p.\ map $\varphi:A \to B^{**}$ such that $\varphi(b)=b$ for every $b\in B$;
\item[(2)] for every $*$-homomorphism $\pi:B\to \lB (H)$ there exists a c.c.p.\ map $\varphi:A \to \pi(B)''$ such that $\varphi(b)=\pi(b)$ for every $b\in B$;
\item[(3)] for every {\rm C}$^*$-algebra $C$ there is a natural inclusion
$
B\otimes_{\rm max} C \subset A\otimes_{\rm max}C.
$ 
\end{itemize}
\end{definition}

\begin{definition}[\cite{Jian-Sepideh}]
Let $A$ and $B$ be $\rC^*$-algebras.
Then, $A$ is said to have the $B$-WEP$_1$ (resp. $B$-WEP$_2$) if there exists $(X, \pi_X) \in \Corr (A, B)$ (resp. $(X, \pi_X) \in \Corr (A, B^{**} )$ with $X$ selfdual) such that $\pi_X $ is injective and the inclusion $\pi_X(A) \subset \lL_B (X)$ (resp. $\pi_X(A) \subset \lL_{B^{**}}(X)$ is relatively weakly injective.
\end{definition}
Note that Lance's WEP is exactly $\lC$-WEP$_1$.
In \cite{Jian-Sepideh} it was proved that $B$-WEP$_1$ implies $B$-WEP$_2$.
We note that the next proposition can be applied to every triple $(A, B, E)$ that is nuclear via some C$^*$-correspondence over $B$,
which is an analogue of the fact that `nuclearity $\Rightarrow$ WEP'.
\begin{proposition}\label{prop-BWEP}
Let $A$ and $B$ be $\rC^*$-algebras with $A$ unital.
If there exist $X \in \Corr (A, B)$ and $Y \in \Corr (B,A)$ such that $\pi_X$ is unital injective, and $(X\otimes_B Y, \pi_X \otimes 1_Y)$ has the UFP,
then the inclusion $\pi_X( A) \subset \lL_B (X)$ is relatively weakly injective.
In particular, $A$ has the $B$-WEP$_1$. 
\end{proposition}

\begin{proof}
Let $(H, \pi_H) \in \Rep (A)$ be arbitrary.
We will show that there exists a u.c.p.\ map $\Psi : \lL_B (X) \to \pi_H (A)''$ such that $\Psi \circ \pi_X = \pi_H$.
The UFP of $X \otimes_B Y$ and Proposition \ref{prop-product-map2}
imply that there exists a u.c.p map $\Theta :  \lL_A (X \otimes_B Y) \otimes 1_H \to \pi_H (A)''$ satisfying $\Theta ( \pi_X (a) \otimes 1_Y \otimes 1_H) = \pi_H (a)$.
Hence, the mapping $\Psi: \lL_B (X) \ni x \mapsto \Theta (x \otimes 1_Y \otimes 1_H) \in \pi_H (A)''$ is the desired one.
\end{proof}
%
%

\subsection{Relative amenability for von Neumann algebras}\label{ss-rel-ame}
In this subsection, we see that our relative nuclearity is related to relative amenability for von Neumann algebras \cite{Popa}\cite{Delaroche2}\cite{Ozawa-Popa}.

Let $N \subset M$ be an inclusion of finite von Neumann algebras and $\tau$ be a faithful normal tracial state.
Let $L^2(M)$ and $L^2(N)$ be the GNS Hilbert spaces for $\tau$ and $\tau |_N$, respectively and $\xi_\tau \in L^2(M)$ be the corresponding cyclic vector.
We may assume that $M \subset \lB (L^2(M))$.
Denote by $E_N$ the unique faithful normal $\tau$-preserving conditional expectation from $M$ onto $N$.
Let $e_N \in \lB (L^2(M))$ be the orthogonal projection onto $L^2(N) \subset L^2(M)$,
which is called the Jones projection and satisfies that $e_N x e_N = E_N (x) e_N$ for $x \in M$.
The basic extension $\i<M, e_N>$ is the von Neumann subalgebra of $\lB (L^2(M))$ generated by $M$ and $e_N$.
It is known that $\i< M, e_N>$ is the $\sigma$-weak closure of $\lspan \{ x e_N y \mid x, y \in M \}$ and becomes semifinite with the canonical faithful normal semifinite tracial weight $\Tr :x e_N y \mapsto \tau (xy)$.
\begin{theorem}[{\cite[Theorem 2.1]{Ozawa-Popa}}]
Let $M$ be a finite von Neumann algebra with a faithful normal tracial state $\tau$ and $Q, N \subset M$ be von Neumann subalgebras.
Then, the following are equivalent.
\begin{itemize}
\item[$(1)$] There exists a $N$-central state $\varphi$ on $\i< M, e_Q>$ such that $\varphi |_M = \tau$.
\item[$(2)$]There exists a $N$-central state $\varphi$ on $\i< M, e_Q>$ such that $\varphi$ is normal on $M$ and faithful on $\cZ (N' \cap M)$.
\item[$(3)$]There exists a conditional expectation $\Phi : \i< M, e_Q>$ onto $N$ such that $\Phi|_M  = E_N$.
\item[$(4)$]There exists a net $\{ \xi_i \}_i$ in $L^2 \i< M, e_Q>$ such that $\lim_i \i< \xi_i, x \xi_i > =\tau (x)$ for $x \in M$ and $\lim_i \| [ u, \xi_i ] \|_2 =0$ for all $u \in \cU (N)$.
\end{itemize}
When any of these conditions holds, we say that $N$ is amenable relative to $Q$ inside $M$ and write $N \lessdot_M Q$.
\end{theorem}
When $M \lessdot_M N$ holds, we also say that $M$ is {\it amenable relative to} $N$.
The next lemma seems to be known among specialists, but we do give its proof for the reader's convenience.
\begin{lemma}
There exists a unitary $U$ from $L^2 \i< M, e_Q>$ onto $H:=L^2(M, E_Q) \otimes_Q L^2(M)$ that maps $x e_N y$ to $x \otimes y $ and satisfies that $ U \i< M, e_Q> U^*= (\lL_Q (L^2(M, E_Q) )\otimes 1_{L^2(M)})''$.
\end{lemma}
\begin{proof}
By \cite[Lemma 2.1]{Izumi-Longo-Popa} $ \lspan Me_N M$ is norm dense in $L^2\i< M, e_N>$.
For any finite sums $\sum_k x_k e_N y_k$, $\sum_l z_l e_N w_l  \in \lspan M e_N M$ we have
$$
\Tr ( (\sum_k x_k e_N y_k)^*(\sum_l z_l e_N w_l) )
= \sum_{k,l} \tau ( y_k^* E_N ( x_k^*y_l) z_l) = \i<\sum_k x_k \otimes y_k \xi_\tau, \sum_l z_l \otimes  w_l \xi_\tau >_H.
$$
Thus, the mapping $\lspan M e_N M \ni \sum_k x_k e_N y_k \mapsto \sum_k x_k \otimes y_k \xi_\tau \in H$ is bounded and extends to a unitary, which gives the isomorphism of Hilbert $M$-$M$ bimodules.
To see that $ U \i< M, e_Q> U^*= (\lL_Q (L^2(M, E_Q) )\otimes 1_{L^2(M)})''$, let $P_N \in \lL_N (L^2(M, E_N))$ the projection defined by $P_N x = E_N (x)$ for $x \in M$.
Then, it follows that $\lK (L^2(M, E_N)) = \ospan M P_N M$ and $U (x e_N y)U^*=x P_N y$ for $x, y \in M$.
Since $(\lK_N (L^2(M, E_N) )\otimes 1 )'' = (\lL_N (L^2(M, E_N) )\otimes 1 )''$,
we get $(\lL_Q (L^2(M, E_Q) )\otimes 1_{L^2(M)})''= ( \lspan M P_N M \otimes 1)'' =U(  \lspan M e_N M \otimes 1 )''U^* = U \i<M, e_N> U^*$.
\end{proof}
Here we give a characterization of relative amenability in terms of weak containment.
\begin{proposition}\label{prop-rel-ame}
Let $Q, N \subset M$ be inclusions of finite von Neumann algebras and $\tau$ be a faithful normal tracial state on $M$.
Then, $N \lessdot_M Q$ if and only if $L^2(M, E_N) \prec L^2 (M, E_Q) \otimes_Q L^2 (M, E_N)$ as $M$-$N$ $\rC^*$-correspondences.
\end{proposition}
\begin{proof}
Put $X_N:=L^2(M,E_N)$ and $X_Q:=L^2(M, E_Q)$ and let $\xi_N \in L^2 (M,E_N)$ be the vector corresponding to $1$.
We note that $X_N \otimes_N L^2 (N) \cong L^2(M)$.
Suppose that $X_N$ is weakly contained in $X_Q \otimes_Q X_N$ with respect to the standard representation $N \subset \lB (L^2 (N, \tau|_N) )$.
By Proposition \ref{prop-product-map2} there exists a u.c.p map $\Phi$ from $(\lL_Q (X_Q) \otimes 1 \otimes 1)'' \subset \lB (X_Q \otimes_Q X_N \otimes_N L^2 (N) ) = \lB (X_Q \otimes_Q L^2 (M) )$ into $N$ satisfying that $\Phi (x \otimes 1 \otimes 1) = \Omega_{\xi_N} (x) =E_N (x)$ for $x\in M$.
Since $(\lL_Q (X_Q) \otimes 1 \otimes 1 )'' $ is isomorphic to $\i< M, e_Q>$, we are done.

Suppose $N \lessdot_M Q$.
Let $\xi_i \in L^2 \i< M, e_Q >$ be a net in (4) of the theorem above.
Let $z \in \ker \theta_{X_Q \otimes_Q X_N}^{L^2(N)}$ and $\varepsilon >0$ be arbitrary.
Take $w = \sum_k x_k \otimes y_k^\op \in M \odot N^\op$ such that $\| z - w \|_{\max} < \varepsilon$.
Note that $X_N \otimes_N L^2(N) \cong L^2(M)$.
For any $a, b \in M$ we have
\begin{align*}
| \i< a \xi_\tau, \theta_{X_N}^{L^2(N)} (w ) b \xi_\tau > |&=  |\tau (\sum_k a^*x_k b y_k )|= \lim_i | \i< \xi_i,\sum_k a^*x_k b y_k \xi_i > | \\
&=\lim_i | \i< a \xi_i, \sum_k x_k b \xi_i y_k > | = \lim_i | \i< a \xi_i, \theta_{X_Q \otimes_Q X_N}^{L^2(N)} (w )b \xi_i >| \leq  \varepsilon \|a \| \|b \|,
\end{align*}
which implies $\theta_{X_N}^{L^2(N)} (z) =0$.
\end{proof}

Let $N\subset M$ be an inclusion of von Neumann algebras.
Recall that $M$ is said to be {\it injective relative to} $N$ if there exists a norm one projection from $JN'J \subset L^2 (M)$ onto $M$ (\cite[Definition 3.1]{Delaroche2}),
where $J : L^2 (M) \to L^2 (M)$ is the modular conjugation.
When $M$ is finite, this is the case that $M \lessdot_M N$, that is, $M$ is amenable relative to $N$ (inside $M$).
\begin{proposition}\label{prop-relative-injective}
Let $B\subset A$ be a unital inclusion of $\rC^*$-algebras.
If there exists a unital $X\in \Corr(A,B)$ such that $(X \otimes_B A, \pi_X \otimes 1_A)$ has the UFP,
then for any $*$-representation $\pi_H :A \to \lB(H)$,
there exists a norm one projection from $\pi_H(B)'$ onto $\pi_H(A)'$.
Consequently, $\pi_H(A)''$ is injective relative to $\pi_H(B)''$.
\end{proposition}
\begin{proof}
We note that $X\otimes_B A \otimes_A H =X\otimes_B H$.
Applying the UFP of $X \otimes_B A$ to $\pi_H \times \iota : A\odot \pi_H(A)' \to \lB(H)$
we get the $*$-homomorphism from $\Imag \phi_X^H \subset \lB (X \otimes_B H)$ to $\lB (H)$. By Arveson's extension theorem, we obtain a u.c.p.\ map $\Phi: \lB (X\otimes_B H ) \to \lB (H)$ satisfying $\Phi ( \pi_X (a) \otimes x ) = \pi_H (a)x$ for $a \in A$ and $x\in \pi_H (A)'$.
Define $\Psi: \pi_H (B)' \to \lB (H)$ by $\Psi (x):= \Phi (1_X \otimes x)$.
Then we can prove that $\Psi$ is a norm one projection onto $\pi_H(A)'$ as in the proof of Proposition \ref{prop-product-map2}.
To see the second assertion, let $M:=\pi_H(A)''$ and $N:=\pi_H(B)''$.
We may assume that $M$ and $N$ are acting on the standard Hilbert space $L^2 (M)$.
Let $J$ be the modular conjugation on $L^2 (M)$.
We then have a u.c.p.\ map $\Phi : N' \to M'$.
Thus, $JN'J \ni J x J \mapsto J \Psi (x) J \in J M'J=M$ is the desired map. 
\end{proof}

\subsection{Permanence properties}\label{ss-rel-per}
For Hilbert $\rC^*$-modules $X$ and $Y$ over $\rC^*$-algebras $A$ and $B$, respectively,
we denote by $X \toplus Y$ the Hilbert $A \oplus B$-module $X \oplus_{\rm alg} Y$ equipped with the inner product
$\i< \xi \oplus \eta, \xi' \oplus \eta >=\i< \xi, \xi' > \oplus \i< \eta, \eta' >$ for $\xi, \xi' \in X$ and $\eta, \eta' \in Y$.
The next proposition immediately follows from Theorem \ref{prop-UFP} and the definition of relative CCPAP.
\begin{proposition}
Let $B_i \subset A_i, i=1,\dots, n$ be inclusions of $\rC^*$-algebras with $A_i$ unital.
For $X_i \in \Corr (A_i)$ the following hold true.
\begin{itemize}
\item[$(1)$] The $X_1 \toplus X_2 \toplus \cdots \toplus X_n \in \Corr (\bigoplus_{i=1}^n A_i)$ has the UFP (resp. $\bigoplus_{i=1}^n B_i$-CCPAP) if and only if each $X_i$ has the UFP (resp. $B_i$-CCPAP).
\item[$(2)$] The $\bigotimes_{i=1}^n X_i \in \Corr (\bigotimes_{i=1}^n A_i)$ has the UFP (resp. $\bigotimes_{i=1}^n B_i $-CCPAP) if each $X_i$ has the UFP (resp. $B_i$-CCPAP).
\end{itemize}
\end{proposition}

\begin{proposition}[Direct sums and tensor products]\label{prop-sum-tensor}
Let $(A_i, B_i, E_i), 1\leq i \leq n$ be a finite family of unital inclusions of $\rC^*$-algebras with conditional expectations.
Then, the following hold true{\rm :}
\begin{itemize}
\item[$(1)$] $(\bigoplus_i A_i ,\bigoplus_i B_i, \bigoplus_i E_i )$ is nuclear (resp. strongly nuclear) if and only if so is each $( A_i, B_i, E_i)$.
\item[$(2)$] $(\bigotimes_i A_i ,\bigotimes_i B_i, \bigotimes_i E_i )$ is nuclear (resp. strongly nuclear) if so is each $( A_i, B_i, E_i)$.
\end{itemize}
\end{proposition}
\begin{proof}
Let $(A, B, E):=(\bigoplus_i A_i ,\bigoplus_i B_i, \bigoplus_i E_i )$ and $(\widetilde{A}, \widetilde{B}, \widetilde{E}):= (\bigoplus_i A_i ,\bigoplus_i B_i, \bigoplus_i E_i )$.
The assertions follow from the propositions above together with the next canonical isomorphisms 
$L^2 (A, E ) \otimes_{B} A \cong \bigotimes_i (L^2 (A_i, E_i) \otimes_{B_i} A_i )$ and $L^2(\widetilde{A}, \widetilde{E}) \otimes_{\widetilde{B}} \widetilde{A} \cong \bigtoplus_i (L^2 (A_i, E_i) \otimes_{B_i} A_i )$.
\end{proof}

\begin{proposition}[Inductive limits]
Let $B \subset A$ be a unital inclusion of $\rC^*$-algebras with a conditional expectation $E :A \to B$.
If there exists an increasing net $1_A \in A_i, i \in \cI$, of unital $\rC^*$-subalgebras of $A$
such that $E(A_i) =A_i \cap B$ and $(A_i, A_i \cap B, E|_{A_i} )$ is nuclear and $\bigcup_i A_i$ is norm dense in $A$, then $(A, B, E)$ is nuclear.
\end{proposition}
\begin{proof}
Take finitely many elements $\{ x_1, \dots, x_n \} \subset A$ and $\varepsilon >0$ arbitrarily.
Since $\bigcup_i A_i$ is norm dense in $A$, there exists $i \in \cI$ and $y_1, \dots, y_n \in A_i$ such that $\| x_i - y_i \| <\varepsilon /3$ for $1\leq i \leq n$.
Set $B_i:=B \cap A_i$ and $E_i := E|_{A_i}$.
By the nuclearity of $(A_i, B_i, E_i )$ we can find a c.c.p.\ map $\psi \in \cF_{L^2(A_i, E_i) \otimes_{B_i} A_i}$ such that $\| \psi (y_i) - y_i \| <\varepsilon/3$ for $1 \leq i \leq n$.
Since $L^2(A_i, E_i) \otimes_{B_i} A_i$ can be embedded in $L^2(A, E) \otimes_B A$ canonically and $1_A =1_{A_i}$, the $\psi$ is a c.c.p.\ on $A$.
Now we get $\| x_i - \psi (x_i) \| \leq \| x_i - y_i \| + \| y_i - \psi (y_i) \| + \| \psi (y_i - x_i ) \| < \varepsilon$ for $1\leq i \leq n$.
\end{proof}

The following proposition can be shown in the same manner.
\begin{proposition}[Inductive limits with common subalgebras]
Let $B \subset A$ be a unital inclusion of $\rC^*$-algebras with a conditional expectation $E :A \to B$.
If there exists an increasing net $A_i, i \in \cI$, of unital $\rC^*$-subalgebras of $A$ containing $B$ such that $(A_i, B, E|_{A_i} )$ is strongly nuclear and $\bigcup_i A_i$ is norm dense in $A$, then $(A, B, E)$ is strongly nuclear.
\end{proposition}

%
%
\section{Examples}\label{sec-exam}
In this section, we give several examples of nuclear and strongly nuclear triples.
\subsection{Inclusions of nuclear C$^*$-algebras}\label{ss-exam-inc}
\begin{example}
For any $\rC^*$-algebra $A$ the triple $(A, A, \id )$ is strongly nuclear.
Further assume that $A$ is unital and $\varphi :  A \to \lC$ is a nondegenerate state.
Then, the nuclearity of $A$ is equivalent to the strong nuclearity of $(A, \lC 1_A, \varphi )$ by Proposition \ref{prop-UFP} and Proposition \ref{prop-nuclear}. 
\end{example}
The following lemma can be shown in the same manner as \cite[Lemma 2.21]{Dadarlat-Eilers}.
\begin{lemma}
Let $B \subset A$ be a unital inclusion of $\rC^*$-algebras and $\varphi$ be a state on $B$.
If the embedding $\iota : B \hookrightarrow A$ is full, that is, $\ospan AbA =A$ holds for all $b \in B \setminus \{0 \}$,
then for any finite subset $\fF \subset B$ and $\varepsilon >0$, there exist $n \in \lN$ and $a_1, \dots, a_n \in A$ such that $\| \varphi (b)1_A - \sum_{i=1}^n a_i^* b a_i \|  < \varepsilon$ for $b \in \fF$.
\end{lemma}
\begin{example}[Inclusions of nuclear $\rC^*$-algebras]\label{ex-nuc}
Let $B \subset A$ be a unital inclusion of $\rC^*$-algebras with nondegenerate conditional expectation $E: A\to B$.
Take a faithful representation $\pi_H : B \to \lB (H)$.
Then, $A$ is nuclear if and only if $(A, B, E)$ is nuclear via $(H \otimes B, \pi_H \otimes 1)$.
Moreover, when the embedding $\iota : B \hookrightarrow A$ is full, the triple $(A, B, E)$ is nuclear.

Since $(L^2(A,E) \otimes_B H, \pi_E \otimes 1_H)$ is also a faithful representation of $A$,
$A$ is nuclear if and only if $(A, \id_A) \prec_\univ (( L^2(A,E) \otimes_B H) \otimes A , \pi_E \otimes 1_H \otimes 1_A)$ holds, thanks to Proposition \ref{prop-UFP} and Proposition \ref{prop-nuclear}.
The natural isomorphism $( L^2(A, E) \otimes_B H) \otimes A = L^2(A, E) \otimes_B (H \otimes B) \otimes_B A$ implies that this is the case when $(A, B,E)$ is nuclear via $(H\otimes B, \pi_H \otimes 1_A)$.

Now assume that $B$ is fully embedded.
Since every coefficient of $(H \otimes A, \pi_H \otimes 1_A)$ is approximated by c.p.\ maps of the form $b \mapsto \sum_{i,j} a_i^* \varphi (b_i^* b b_j) a_j$ for some state $\varphi$ and $a_i \in A, b_i \in B$,
by the preceding lemma and Corollary \ref{cor-weak}, we have $(H \otimes A, \pi_H \otimes 1_A) \prec_\univ (A, \iota)$.
Thus, thanks to Proposition \ref{prop-int-tensor} we have
\[
(A, \lambda_A) \prec_\univ (L^2(A, E) \otimes_{\pi_H \otimes 1_A} (H \otimes A), \pi_E \otimes 1_{H \otimes A}) \prec_\univ (L^2(A, E) \otimes_{ \iota} A, \pi_E \otimes 1_A).
\]
We note that if $A$ or $B$ is simple, then $\iota : B \to A$ is full.
\end{example}

\begin{example}[Continuous fields of nuclear C$^*$-algebras]
Let $A$ be a continuous filed of unital nuclear $\rC^*$-algebras on a compact Hausdorff space $X$,
that is, there is a unital embedding $C(X) \subset A' \cap A$, such that, the quotient maps $p^x : A \to A^x := A/ J_x$, where $J_x$ is the closed ideal generated by the kernel of the evaluation map $\ev_x :  C (X) \to \lC$ at $x$,
satisfies that $X \ni x \mapsto  \| p^x (a) \| \in \lC$ is continuous for each $a \in A$ and $A^x$ is nuclear for each $x \in X$.
Assume that there exists a conditional expectation $E : A \to C(X)$ such that for each $x \in X$ there exists a nondegenerate state $\varphi^x$ on $A^x$ satisfying $\ev_x \circ E = \varphi^x \circ p^x$.
Then, the triple $(A, C(X), E)$ is strongly nuclear (see \cite{Bauval}).

For the reader's convenience, we give a sketch of the proof of this observation.
Since every vector in $L^2(A, E) \otimes_{C(X) } A$ is $C(X)$-central, it suffices to show the nuclearity of $(A, B, E)$.
Fix $ a \in A$ and $\varepsilon >0$.
For each $x \in X$, the nuclearity of $A^x$ implies that there exist find finitely many vectors $\xi_{x,i} \in A \odot A \subset L^2(A, E) \otimes_{C(X) } A$ such that $\|p^x (a -  \sum_{i} \i< \xi_{x,i},  a \xi_{x,i} >) \| < \varepsilon$. 
Since $A$ is a continuous field, we can find an open neighborhood $U_x$ of $x$ such that $\|p^y (a -  \sum_{i} \i< \xi_{x,i},  a \xi_{x,i} > ) \| < \varepsilon$ for $y \in U_x$.
By the compactness of $X$, there exists a finite subset $F \subset X$ such that $X= \bigcup_{x \in F} U_x$.
Let $\{ h_x \}_{x \in F}$ be a partition of unity for this covering.
Then, the vectors $\{ \xi_{x,i} h_x^{1/2} \}_{x \in F, i}$ in $L^2(A, E) \otimes_{C(X)} A$ do the job.
Indeed, we have
\begin{align*}
 \| a - \sum_{x \in F} \sum_i  \i< \xi_{x,i}h_x^{1/2},  a\xi_{x,i} h_x^{1/2}> \| &=\sup_{y \in X}  \| \| p^y ( a - \sum_{x \in F}  \sum_i  \i< \xi_{x,i},  a \xi_{x,i} > h_x (y) ) \| \\ &\leq \sup_{y \in X} \sum_{x \in F} \| p^y ( a - \sum_i  \i< \xi_{x,i},  a \xi_{x,i} >  ) \|  h_x (y) < \varepsilon.
\end{align*}
\end{example}
When $B$ is finite dimensional, relative nuclearity implies strong one thanks to the following proposition.
\begin{proposition}
Let $B \subset A$ be a unital inclusions of $\rC^*$-algebras with $B$ finite dimensional.
If $X \in \Corr (A)$ has the UFP, then $X$ has the $B$-CCPAP.
\end{proposition}
\begin{proof}
Write $B =\bigoplus_{r=1}^p \lM_{n(r)}$.
Let $\{ e^{(r)}_{ij} \}_{i,j=1}^{n(r)}$ be a matrix unit system of $\lM_{n(r)}$.
Take a finite subset $\fF \subset A$ and $\varepsilon >0$ arbitrarily.
By the the UFP, we can find $\xi_1, \dots, \xi_m \in X$ in such a way that
$$
\| e_{1i}^{(r)} a e_{j1 }^{(s)} - \sum_{k=1}^m \i< \xi_k, e_{1i}^{(r)} a e_{j1 }^{(s)} \xi_k > \| < \frac{\varepsilon}{p^2  \max_{1 \leq r \leq p} n(r)^2} \quad 1 \leq r,s \leq p,\; 1 \leq i,j \leq n(r)
$$
and $\sum_{k=1}^m  \i< \xi_k, \xi_k > \leq 1$.
Set $ \eta_k := \sum_{r=1}^p \sum_{i=1}^{n(r)} e_{i1}^{(r)} \xi_k e_{1i}^{(r)}$ for $ 1 \leq k \leq m$.
Then, we have $e^{(r)}_{ij} \eta_k =e^{(r)}_{ij} e_{j1}^{(r)} \xi e_{1j}^{(r)} =e^{(r)}_{i1}\xi e_{1i}^{(r)}e_{ij}^{(r)} = \eta_k e_{ij}^{(r)}$, and hence $\eta_k$ is $B$-central.
For $a \in \fF$ we have
\begin{align*}
\sum_{k=1}^m \i< \eta_k, a \eta_k >
&= \sum_{k=1}^m\l< \sum_{r=1}^p \sum_{i=1}^{n(r)} e_{i1}^{(r)} \xi e_{1i}^{(r)}, a \sum_{s=1}^p \sum_{j=1}^{n(s)} e_{j1}^{(s)} \xi e_{1j}^{(s)} >
=\sum_{r, s=1}^p \sum_{i =1}^{n(r)} \sum_{j=1}^{n(s)} \sum_{k=1}^m e_{i1}^{(r)} \i< \xi, e_{1i}^{(r)} a e_{j1}^{(s)} \xi >  e_{1j}^{(s)}\\
& \approx_\varepsilon \sum_{r, s=1}^p \sum_{i =1}^{n(r)} \sum_{j=1}^{n(s)} e_{i1}^{(r)}  e_{1i}^{(r)} a e_{j1}^{(s)} e_{1j}^{(s)} 
= \left(\sum_{r=1}^p \sum_{i =1}^{n(r)}  e_{ii}^{(r)} \right) a  \left( \sum_{s=1}^p \sum_{j=1}^{n(s)}e_{jj}^{(s)} \right)
=a
\end{align*}
We also have $\sum_{k=1}^m \i< \eta_k, \eta_k >= \sum_{r=1}^p \sum_{i=1}^{n(r)} e_{i1}^{(r)} \i< \xi, e_{11}^{(r)} \xi >  e_{1i}^{(r)} \leq \sum_{i=1}^{n(r)} e_{ii}^{(r)} = 1_A$.
\end{proof}

The following proposition is important in terms of Theorem \ref{thm-D}.
\begin{proposition}\label{prop-fin-dim}
Let $B \subset A$ be a unital inclusion of $\rC^*$-algebras with nondegenerate conditional expectation $E$.
Suppose that $B$ is finite dimensional.
Then, $A$ is nuclear if and only if there exists a unital countably generated $(Z, \pi_Z) \in \Corr (B)$ such that $Z$ admits a $B$-central vector $\zeta \in Z$ with $\i< \zeta, \zeta > =1_B$ and $(A, B, E)$ is strongly nuclear via $(Z, \pi_Z)$.
\end{proposition}
\begin{proof}
The `only if' part follows from Proposition \ref{prop-nuc-exact}.
Suppose that $A$ is nuclear.
Take a faithful $*$-representation $\pi_H : B \to \lB (H)$.
By Example \ref{ex-nuc}, we have $(A, \lambda_A) \prec_\univ (L^2(A, E) \otimes_B (H \otimes B) \otimes_B A, \pi_E \otimes 1_{H \otimes B}  \otimes 1_A )$.
Thanks to the previous proposition and Proposition \ref{prop-int-tensor}, it enoughs to find a unital countably generated $(Z, \pi_Z) \in \Corr (B)$ containing a $B$-central normal vector and satisfying that $(H \otimes B, \pi_H \otimes 1) \prec_\univ (Z, \pi_Z)$.
Write $B = \bigoplus_{r=1}^p \lM_{n(r)}$ and set $n:=\sum_{r=1}^p n(r)$.
Let $\iota : B \hookrightarrow \lM_n$ be the natural embedding.
Since $\lM_n$ is simple, this embedding is full, and hence we have $(H \otimes \lM_n, \pi_H \otimes 1) \prec_\univ (\lM_n, \iota)$ as $B$-$\lM_n$ C$^*$-correspondences by Example \ref{ex-nuc}.
Let $\Phi : \lM_n \to B$ be the canonical trace preserving conditional expectation.
We then have 
\[
(H \otimes B, \pi_H \otimes 1_B) \prec_\univ (H \otimes L^2(\lM_n, \Phi ), \pi_H \otimes 1 ) \prec_\univ (L^2(\lM_n, \Phi), \pi_\Phi).
\]
Therefore, $(L^2(\lM_n, \Phi), \pi_\Phi \circ \iota ) \in \Corr (B)$ and $\xi_\Phi \in L^2(\lM_n, \Phi)$ are the desired ones.
\end{proof}
\subsection{Tensor products, subalgebras of finite index, and relative amenable groups}\label{ss-exam-other}
\begin{example}[Tensor products with nuclear C$^*$-algebras]
Let $A$ and $B$ be unital $\rC^*$-algebra and $f: A \to \lC$ be a nondegenerate state.
Then $A$ is nuclear if and only if the triple $(A \otimes B, \lC 1_A \otimes B, f \otimes \id_B)$ is strongly nuclear.
Indeed, if $A$ is nuclear, then there exists a net $\psi_i \in \cF_{H_f \otimes A}$ approximates $\id_{A \otimes B}$.
Since $X:=L^2 (A \otimes B, f \otimes \id_B) \otimes_B( A \otimes B )$ is isomorphic to $(H_f \otimes A) \otimes B$,
the $\psi_i \otimes \id_B$ belongs to $\cF_{B'\cap X}$.
Conversely, suppose that $(A \otimes B, B,f \otimes \id_B)$ is nuclear and take a net $\varphi_i \in \cF_X$ converges to $\id_{A\otimes B}$.
If $g$ be a state on $B$, then $A \ni a \mapsto (\id_A \otimes g ) \circ \psi_i (a \otimes 1_B)$ is nuclear and converges to $\id_A$.
\end{example}

\begin{example}[Finite Watatani index]
Let $1_A \in B \subset A$ be a unital inclusion of $\rC^*$-algebras
and assume that there exists a conditional expectation $E :A \to B$ of finite Watatani index,
that is, there exists a finite family of elements , called a quasi-basis, $\{ u_1, \dots, u_n \} \in A$ such that $a = \sum_{j=1}^n u_j E (u_j^* a) = \sum_{j=1}^n E(a u_j) u_j^* $ holds for $a \in A$.
Then $(A, B, E)$ is strongly nuclear.
Recall the fact that the element $e:=\sum_{i=1}^n u_i u_i^*$ (called the index of $E$) is an  invertible element in $\cZ (A)=A'\cap A$ \cite[Lemma 2.3.1]{Watatani}.
Letting $\xi:=\sum_{i=1}^n u_i \otimes u_i^* e^{-1/2} \in L^2 (A, E) \otimes_B A$ we have
\begin{align*}
(\pi_E (a) \otimes_B 1_A ) \xi
&= \sum_{i=1}^n a u_i \otimes u_i^* e^{-1/2}
=\sum_{i=1}^n \sum_{j=1}^n u_j E(u_j^* a u_i ) \otimes u_i^* e^{-1/2} \\
&= \sum_{j=1}^n u_j  \otimes \sum_{i=1}^n E(u_j^* a u_i ) u_i^* e^{-1/2} 
=\sum_{j=1}^n u_j \otimes u_j^* a e^{-1/2} = \xi  a,
\end{align*}
hence $\xi \in A' \cap ( L^2 (A, E) \otimes_B A)$.
Since $\i< \xi, \xi >= e^{-1} \sum_{i,j=1}^n u_i E(u_i^* u_j )u_j^*=1_A$, we get $\Omega_\xi (a) =a$ for all $a\in A$.
\end{example}
\begin{example}[Finite probabilistic index]
Let $B \subset A$ be a unital inclusion of C$^*$-algebras.
A conditional expectation $E : A \to B$ is said to be {\it finite probabilistic index finite} if there exists a constant $\lambda >0$ such that
$\lambda^{-1} E -\id _A$ is a positive map on $A$.
In this case, the triple $(A, B, E)$ is nuclear.
We show that $(X, \pi_X):= (L^2(A, E) \otimes_B A ,\pi_E \otimes 1_A)$ has the UFP.
By \cite[Theorem 1]{Frank-Kirchberg} there exists $\mu >0$ such that $\mu^{-1} E - \id _A$ is completely positive.
Let $C$ be any C$^*$-algebra and $\sigma: A \odot C \to \lB (H)$ be any nondegenerate representation. 
Fix $\sum_{i=1}^n a_i \otimes c_i \in A \odot C$ and $\xi \in H$ arbitrarily.
We observe that $[\mu^{-1}  E ( a_i^* a_j) - a_i^* a_j]_{i,j=1}^n \in \lM_n (A)$ is positive.
Hence, we get
\begin{align*}
\| \sum_{i=1}^n \sigma (a_i  \otimes c_i) \xi \| ^2 
&=  \i< \xi, \sum_{i,j=1}^n  \sigma_A (a_i^* a_j)  \sigma_C (c_i^* c_j) \xi >
\leq \mu^{-1} \i< \xi, \sum_{i,j=1}^n \sigma_A ( E ( a_i^* a_j) ) \sigma_C (c_i^* c_j) \xi >\\
&=\mu^{-1} \| \sum_{i=1}^n \pi_E (a_i ) \xi_E \otimes 1_A \otimes \sigma_C (c_i) \xi \|^2
\leq \mu^{-1} \| \sum_{i=1}^n \phi_X^H (a_i \otimes c_i ) \|^2 \| \xi \|^2.
\end{align*}
Hence, the mapping $\sum_{i=1}^n \phi_X^H (a_i  \otimes c_i ) \mapsto \sum_{i=1}^n  \sigma (a_i  \otimes c_i)$ is bounded.
\end{example}
\begin{remark}
We note that if $E :A \to B$ is Watatani index finite, then it is probabilistic index finite (\cite[Proposition 2.6.2]{Watatani}).
When $A$ is simple, the converse holds true by \cite[Theorem 3.2]{Izumi}.
Let $\Gamma$ be a finite group acting on a unital C$^*$-algebra $A$ by $\alpha$.
By averaging we obtain a conditional expectation $E_\alpha$ from $A$ into the fixed point algebra $A^\alpha$.
It is clear that $| \Gamma|^{-1} E_\alpha - \id_A$ is positive.
Hence, if $A$ is simple, then $(A, A^\alpha, E_\alpha)$ is strongly nuclear.
Thus, there are many example of strongly nuclear triples that are not ones arising from crossed products.
\end{remark}

\begin{example}[Co-amenable subgroups]
Let $\Lambda < \Gamma$ be an inclusion of discrete groups.
Recall that $\Lambda$ is said to be {\it co-amenable} in $\Gamma$ if there exists an left invariant mean on $\ell^\infty (\Gamma /\Lambda)$ for the action of $\Gamma$ on $\Gamma / \Lambda$ by the left multiplication.
When $\Lambda$ is a normal subgroup of $\Gamma$, this is equivalent to the amenability of the quotient group $\Gamma/ \Lambda$.
The compression map by the orthogonal projection onto $\ell^2(\Lambda) \subset \ell^2(\Gamma)$ gives a conditional expectation $E : \rC^*_\re (\Gamma) \to \rC^*_\re (\Lambda)$ such that $E(\lambda_g)=0$ whenever $g \notin \Lambda$.
Then, $( \rC^*_\re (\Gamma ), \rC^*_\re (\Lambda), E)$ is nuclear if and only if $\Lambda$ is co-amenable in $\Gamma$.
Moreover,
if $\Lambda$ is a co-amenable normal subgroup of $\Gamma$, then the triple is strongly nuclear.

Suppose that $\Lambda$ is co-amenable in $\Gamma$.
Fix a finite subset $\fF \subset \Gamma$ and $\varepsilon>0$ arbitrarily.
We denote by $\pi$ the quotient map from $\Gamma$ onto $\Gamma / \Lambda$.
By \cite[Theorem 4.1]{Greenleaf} there exist a finite subset $F \subset \Gamma / \Lambda$ such that
$|F|- |g^{-1}F \cap F | < \varepsilon | F| $ for $g \in \fF$.
Take a lift $G \subset \Gamma$ of $F$, i.e., $\pi (G) = F$ and $| F| =| G |$.
Letting $\xi=|F|^{-1/2}\sum_{f \in G} \lambda_{f} \otimes \lambda_{f}^* \in L^2 (\rC^*_\re (\Gamma) , E) \otimes_{\rC^*_\re (\Lambda)} \rC^*_\re (\Gamma)$ we have
$$
\i< \xi, \lambda_g \xi > = \frac{1}{|F|} \sum_{f, h\in G} \lambda_{f} E (\lambda_{f}^* \lambda_{g } \lambda_{h} ) \lambda_{h}^* = \frac{| F \cap g^{-1} F |}{| F|} \lambda_g
 \approx_\varepsilon \lambda_g$$
for all $g \in \fF$.
In the case that $\Lambda$ is a normal subgroup, the vector $\xi$ is $\rC^*_\re (\Lambda)$-central.

Conversely, if $( \rC^*_\re (\Gamma ), \rC^*_\re (\Lambda), E)$ is nuclear, then the group von Neumann algebra $L (\Gamma)$ is injective relative to $L(\Lambda)$ by Proposition \ref{prop-relative-injective}.
By \cite[Corollary 7]{Monod-Popa}, this implies the co-amenability of $\Lambda$ in $\Gamma$.
\end{example}

\begin{remark}
By \cite{Monod-Popa}\cite{Pestov} there are triples of groups $K < H < G$ such that $K$ is co-amenable in $G$ but not in $H$.
Thus, the example above says that $A:=\rC^*_\re (G)$, $B:=\rC^*_\re (H)$, and $C:=\rC^*_\re (K)$ satisfy that $(A, C, E)$ is nuclear and $(B, C, E|_B)$ is not nuclear, where $E:A \to C$ is the canonical conditional expectation as above.
We also note that there is a conditional expectation from $A$ to $B$.
\end{remark}

\subsection{Crossed products}\label{ss-exam-cro}
Let $B$ be a unital $\rC^*$-algebra and $\alpha : \Gamma \act B$ be an action of a discrete group.
Let $B \rtimes_\alpha \Gamma$ and $B \rtimes_{\alpha, \re} \Gamma$ denote the full and reduced crossed products, respectively.
We denote by $\alpha^{**}$ the action of $\Gamma$ on $B^{**}$ induced from $\alpha$.
\begin{definition}[{\cite[D\'{e}finition 4.1]{Delaroche}}]
Let $B$ be a unital $\rC^*$-algebra and $\Gamma$ be a discrete group.
An action $\alpha : \Gamma \to \Aut (B)$ is said to be {\it amenable}
if there exists a net of functions $\varphi_i :\Gamma \to \cZ(B^{**})$ with a finite support such that
\begin{itemize}
\item $\sum_{g\in \Gamma} \varphi_i (g)^* \varphi_i (g)\leq 1$ for all $i$,
\item $\sum_{g \in \Gamma} \varphi_i (g)^* \alpha^{**}_f (\varphi_i (f^{-1}g))$ converges to $1$ $\sigma$-weakly for all $f \in \Gamma$.
\end{itemize}
\end{definition}

\begin{proposition}
Let $\alpha : \Gamma \act B$ be an action of discrete group and $A:=B \rtimes_{\alpha, \re} \Gamma$.
Then, the following are equivalent{\rm :}
\begin{itemize}
\item[$(1)$] The action $\alpha : \Gamma \act B$ is amenable.
\item[$(2)$] There exists a net of vectors $\xi_i$ in the linear span of $\{ \xi  x \in L^2(A, E) \otimes_B A^{**} \mid \xi \in B' \cap ( L^2 (A,E) \otimes_B A), x \in \cZ (B^{**}) \}$ such that $\i< \xi_i, \xi_i > \leq 1$ and $\Omega_{\xi_i}(a)$ converges to $a$ $\sigma$-weakly for $a\in A$.
\item[$(3)$] The triple $(A, B, E)$ is nuclear.
\end{itemize}
When $B$ is commutative, these conditions are also equivalent to
\begin{itemize}
\item[$(4)$] The triple $(A, B, E)$ is strongly nuclear.
\end{itemize}
\end{proposition}
\begin{proof}
We first prove (1) $\Rightarrow$ (2): Take a net $\{ \varphi_i \}_i \in C_\cpt (\Gamma, \cZ (B^{**}))$ in the definition of amenable actions.
We set $\xi_i := \sum_{g \in \Gamma} u_g \otimes u_g^* \varphi_i (g) \in L^2 (A,E) \otimes_B A^{**}$.
Since $\varphi_i (g) \in \cZ (B^{**})$ we get $\xi_i \in B' \cap (L^2(A,E) \otimes_B A^{**})$.
For any $b \in B$ and $f \in \Gamma$ we have
\begin{align*}
\i< \xi_i, (\pi_X (b u_f) \otimes 1_{A^{**}}) \xi_i >
&= \sum_{g \in \Gamma} \sum_{h \in \Gamma} \i< u_g \otimes u_g^*\varphi_i (g) , b u_{fh} \otimes  u_h^* \varphi_i (h)> \\
&= \sum_{g \in \Gamma} \varphi_i (g)^* u_g \left ( u_g^* b u_g     \right ) u_{f^{-1}g}^*\varphi_i (f^{-1}g) \\
&= \left ( \sum_{g \in \Gamma} \varphi_i(g)^* \alpha_f (\varphi_i (f^{-1} g ) ) \right )b u_f,
\end{align*}
which converges to $b u_f$ by the choice of $\varphi_i$.
We also have $\i< \xi_i, \xi_i > =\sum_{g \in \Gamma} \varphi_i(g)^*\varphi_i (g) \leq 1$.
Since $A$ is the closed span of $\{ b u_g \mid b \in B, g \in \Gamma\}$, we get (2).

We next prove (2) $\Rightarrow$ (3):  Approximating $\xi_i$ in (2) by vectors in $L^2(A,E) \otimes_B A$ we get $\eta_i \in L^2(A, E) \otimes_B A$ such that $\Omega_{\eta_i}$ converges to $\id_A$ point $\sigma$-weakly.

We show (3) $\Rightarrow$ (1): Assume that $(A, B , E)$ is nuclear.
Let $B \subset \lB (H)$ be the universal representation.
Then, $A$ is realized as a $\sigma$-weakly dense subalgebra of $B^{**}\bar{\rtimes}_{\alpha^{**}} \Gamma \subset \lB (H \otimes \ell^2(\Gamma))$.
By Proposition \ref{prop-relative-injective} $B^{**}\bar{\rtimes}_{\alpha^{**}} \Gamma$ is injective relative to $B^{**}$,
and hence $\alpha : \Gamma \act B$ is amenable thanks to \cite[Proposition 3.4]{Delaroche2}.

Now we suppose that $B$ is abelian. The implication (4) $\Rightarrow$ (3) is trivial.
We show (2) $\Rightarrow$ (4):
Since $\cZ (B^{**})=B^{**}=B''$ holds,
we can approximate the $\xi_i$ in (2) by $B$-central vectors in $L^2(A, E) \otimes_B A$. 
\end{proof}
\begin{proposition}
Let $B$ and $\alpha : \Gamma \act B$ be as above.
If $\Gamma$ is amenable, then the $(A, B, E)$ is strongly nuclear.
\end{proposition}
\begin{proof}
Since $\Gamma$ is amenable, there exist F\o lner sets $F_i \subset \Gamma$.
If we put $\xi_i := | F_i |^{-1/2} \sum_{g \in F_i } u_g \otimes u_g^*$,
then $\xi_i$ is $B$-central and we have $\Omega_{\xi_i} (u_g) =|F|^{-1}|F \cap g^{-1} F|u_g $, which converges to $u_g$ in norm for all $g \in \Gamma$.
\end{proof}

\subsection{Groupoids}\label{ss-exam-grpd}
\begin{definition}
A {\it groupoid} $\cG$ is given by a unit space $\cG^{(0)}$ and range and source maps $r, s : \cG \to \cG^{(0)}$
and the multiplication $\cG^{(2)}:= \{ (g, h ) \in \cG \times \cG \mid s (g) = r (h) \} \to \cG; (g, h) \mapsto gh$
satisfying that
\begin{itemize}
\item $s (gh) =s(h)$ and $r (gh) =r (g)$ for $(g,h) \in \cG^\comp$,
\item $s(x) = r(x) = x$ for $x \in \cG^\unit$,
\item $g= r(g) g = g s (s)$ for $g \in \cG$,
\item every $g$ has the inverse $g^{-1} \in \cG$ such that $gg^{-1} =r (g)$ and $g^{-1}g=s(g)$.
\end{itemize}
\end{definition}
A {\it topological groupoid} is a groupoid with a topology such that range, source, multiplication and inverse maps are continuous.
Here the topology on $\cG^\comp$ is the relative topology in $\cG \times \cG$.
A topological groupoid is said to be {\it \'{e}tale} (or {\it $r$-discrete}) if $s$ and $r$ are local homermorphisms, i.e., for any $g \in \cG$ there exists a open neighborhood $U$ of $g$ such that $r|_U : U \to r (U)$ and $s|_U : U \to s(U)$ are homeomorphisms.

\medskip
In what follows, $\cG$ denotes a locally compact Hausdorff \'{e}tale groupoid.
For $x, y \in \cG^\unit$ we set $\cG_x := s^{-1}(x)$, $\cG^y:=r^{-1} (y)$, and $\cG_x^y:= \cG_x \cap \cG^y$.
We note that $\cG_x$ is discrete in $\cG$ and $\cG^\unit$ is clopen in $\cG$.
For $\varphi, \psi \in C_\cpt (\cG)$ we define the convolution product and the adjoint by
$$
\varphi * \psi (g) := \sum_{hk =g} \varphi (h) \psi (k), \quad \varphi^* (g) := \overline{\varphi (g^{-1})}
$$
Since $\cG^\unit$ is clopen in $\cG$, we have $C_\cpt ( \cG^\unit) \subset C_\cpt (\cG)$.
We note that for $\varphi \in C_\cpt (\cG)$ and $\psi \in C_\cpt (\cG^\unit)$ we have
$$
\varphi * \psi (g)  = \varphi (g) \psi ( s(g)) \quad \psi* \varphi (g) = \psi (r(g)) \varphi (g).
$$
In particular, for $\psi, \psi' \in C_\cpt ( \cG^\unit)$ we get $\psi* \psi'(x) = \psi (x)\psi'(x)$.
Therefore, if $\cG^\unit$ is compact, the $\chi_{\cG^\unit}$ forms the unit of $(C_\cpt (\cG), *)$.

We also note that $C_\cpt (\cG)$ forms an inner product $C_\cpt (\cG^\unit)$-module in the following way:
for $\xi, \eta \in C_\cpt (\cG)$ and $\varphi \in C_\cpt (\cG^\unit)$, define
$$
(\xi \varphi ) (g) = \xi (g) \varphi (s (g)) \quad \i< \xi, \eta >(x) = \sum_{g \in \cG_x} \overline{\xi (g)} \eta (g).
$$
We denote by $L^2(\cG)$ the Hilbert $C_\cpt (\cG)$-module given by completion of $C_\cpt (\cG)$.
Define the {\it left regular representation} $\lambda : C_\cpt (\cG )\ni \varphi \mapsto \lambda (\varphi) \to \lL_{C_\cpt (\cG)} ( L^2 (\cG))$ by $\lambda ( \varphi) \psi:=\varphi * \psi$.
The {\it reduced groupoid $\rC^*$-algebra} $\rC^*_\re (\cG)$ of $\cG$ is the completion of the image of $\lambda$ in $\lL_{C_\cpt (\cG)} ( L^2(\cG))$.
As noted above $\rC^*_\re (\cG)$ is unital if $\cG^\unit$ is compact.
The restriction map from $C_\cpt (\cG)$ onto $C_\cpt (\cG^\unit)$ gives the conditional expectation $E$ from $\rC^*_\re (\cG)$ onto $C_\cpt ( \cG^\unit)$.

\begin{definition}
We say that $\cG$ is {\it amenable} if there exists a net of positive functions $\mu_i \in C_\cpt (\cG)$ such that
$$
\lim_i \sup_{g \in K} | \sum_{h \in \cG_{r (g)}} \mu_i (h) - 1 | =0, \quad \lim_i \sup_{g \in K} \sum_{h \in \cG_{r(g)}} | \mu_i (hg ) - \mu_i (h) | =0 
$$
for all compact subset $K$ of $\cG$.
\end{definition}

\begin{lemma}[{\cite[Lemma 5.6.12]{Brown-Ozawa}}]
For $\varphi \in C_\cpt (\cG)$ we set $\| \varphi \|_{I,s}:= \sup_{x \in \cG^\unit} \sum_{ g \in \cG_x} | \varphi (g)|$.
Then, it follows that $\| \lambda (\varphi )\| \leq \max \{ \| \varphi \|_{I,s}, \| \varphi^* \|_{I,s} \}$.
\end{lemma}

To simplify our proofs, we use the following ad hoc notation and terminology.
For a subset $K \subset \cG$ and $x, y \in \cG^\unit$ we set $K_x:=K \cap \cG_x $, $K^y:=K \cap \cG^y$, and $K_x^y:=K_x \cap K^y$.
We also say that a subset $U \subset \cG$ is a {\it $\cG$-set} if $r|_U$ and $s|_U$ are homeomorphisms.
We note that if $U$ is a $\cG$-set, then $|U_x| \leq 1$ for all $x\in \cG^\unit$. 
\begin{lemma}
For any compact subset $K \subset \cG$ we have $C_K:= \sup_{x \in \cG^\unit} \max \{  |K_x|, |K^x| \} < \infty$.
\end{lemma}
\begin{proof}
By the compactness of $K$, there exists a finite family of open $\cG$-sets $\{ U_i \}_{i \in \cI}$ which covers $K$.
Fix $x\in \cG^\unit$.
If $g \in K_x$ belongs to $g \in U_i$, then we have $(U_{i})_x =\{ g \}$.
This implies that $\sup_{x \in \cG^\unit} |K_x| \leq |\cI|$.
\end{proof}

\begin{lemma}
Let $V \subset \cG$ be an open $\cG$-set and $\nu \in C_\cpt (\cG)$ be arbitrary.
If $\supp (\nu ) \subset V$ holds,
then $\lambda_\nu^* \otimes \lambda_\nu$ is a $C_0 (\cG^\unit)$-central vector in $L^2( \cG) \otimes_{C_0 (\cG^\unit)} \otimes \rC^*_\re (\cG)$.
\end{lemma}
\begin{proof}
Set $r_V:=r|_V : V \to r (V)$ and $s_V:=s|_V: V \to s(V)$.
Fix $\varphi \in C_\cpt (\cG^\unit)$ arbitrarily.
Since $\cG$ is locally compact and Hausdorff, we can find a open subset $U \subset \cG$ such that $\supp (\nu) \subset U \subset \overline{U} \subset V$.
By Urysohn's lemma we can find a function $\psi \in C_\cpt ( \cG^\unit)$ such that $0 \leq \psi \leq 1$, $\psi |_{s(U)} =1$, and $\supp (\psi) \subset s (V)$.
Let $\theta:=s_V \circ (r_V)^{-1} : r (V) \to s (V)$.
Since $\supp ((\varphi \psi ) \circ \theta ) \subset r_V \circ (s_V)^{-1}( \supp (\psi)) \subset r (V)$,
we can extend $(\varphi \psi ) \circ \theta$ to $\widetilde{\varphi} \in C_\cpt (\cG^\unit)$ in such a way that $\widetilde{\varphi}(x)=0$ for $x \in \cG^\unit \setminus r (V)$.
We will show that $\varphi * \nu^* = \nu^* * \widetilde{\varphi}$.
We observe that both of $\varphi * \nu^*$ and $\nu^* * \widetilde{\varphi}$ vanish on $\cG \setminus U^{-1}$.
For $g \in U$ we get $\varphi* \nu^* (g^{-1}) = \varphi (r (g^{-1})) \nu^* (g^{-1}) =\varphi (s (g))\nu^*(g^{-1}) =(\varphi \psi) (s (g))\nu^*(g^{-1})= \widetilde{\varphi}( r(g)) \nu^*(g^{-1}) =\nu^*(g^{-1}) \widetilde{\varphi} ( s(g^{-1}))= \nu^* * \widetilde{\varphi}  (g^{-1})$.
Thus, we have $\varphi * \nu^* = \nu^* * \widetilde{\varphi}$, and hence $\lambda( \varphi ) \lambda (\nu^*) \otimes \lambda (\nu) = \lambda ( \nu^* * \widetilde{\varphi})  \otimes \lambda (\nu ) = \lambda( {\nu}) ^* \otimes \lambda ( \widetilde{\varphi}*\nu)  =\lambda (\nu)^*\otimes \lambda (\nu)  \lambda (\varphi) $.
\end{proof}

\begin{theorem}
A locally compact Hausdorff \'{e}tale groupoid $\cG$ with $\cG^\unit$ compact is amenable if and only if $(\rC^*_\re (\cG), C (\cG^\unit), E)$ is strongly nuclear.
\end{theorem}
\begin{proof}
If $(\rC^*_\re (\cG), C (\cG^\unit), E)$ is strongly nuclear, then $\rC^*_\re (\cG)$ is nuclear by Proposition \ref{prop-nuc-exact}.
By \cite[Theorem 5.6.18]{Brown-Ozawa} this implies the amenability of $\cG$.
Conversely, suppose that $\cG$ is amenable.
Fix a finite subset $\fF \subset C_\cpt (\cG)$ and $0< \varepsilon <1/10$ arbitrarily.
We will show that there exists a $C (\cG^\unit)$-central vector $\xi \in L^2( \cG) \otimes_{C (\cG^\unit)} \rC^*_\re (\cG)$ such that
$ \| \lambda (\varphi ) - \Omega_\xi (\lambda (\varphi) ) \| < 9 \varepsilon $ for all $\varphi \in \fF$.
Since $\cG^\unit$ is compact, we may assume that $1 \in \fF$.
Set $K:=\bigcup_{\varphi \in \fF} \supp (\varphi) \cup \supp (\varphi)^{-1}$.
Then, $K$ is compact.
By the amenability of $\cG$, we can find $\mu \in C_\cpt (\cG)_+$ such that
\begin{equation*}
\sup_{ g \in K } | 1 - \sum_{h \in \cG_{r (g)} }\mu (h) | < \frac{ \varepsilon}{ C_K}, \quad
\sup_{ g \in K } \sum_{h \in \cG_{r(g)}} | \mu (h) - \mu (hg) | < \frac{\varepsilon^2}{ C_K^2}.
\end{equation*}
We note that for any $g \in K$ and $h \in \cG_{r (g)}$ it follows that $\mu (h), \mu (hg) \leq 3$.
Indeed, $\mu (h) \leq  \sum_{h \in \cG_{r (g)}} \mu (h) \leq | \sum_{h \in \cG_{r (g)}} \mu (h) - 1 | + 1 \leq 2$ and
$\mu (hg) \leq | \mu (hg) - \mu (h) | + \mu (h) \leq 3$.

Since $\cG$ is locally compact and Hausdorff, we can find a compact subset $K_\mu$ and open subset $O$ in such a way that
$\supp (\mu ) \subset O \subset K_\mu$.
By continuity of $\mu^{1/2}$, there exists a finite open covering $\{ U_i \}_{i \in \cI}$ of $\supp (\mu)$ consisting of $\cG$-sets and $g_i \in U_i$ such that
$| \mu (g_i )^{1/2} - \mu (h)^{1/2} | < \varepsilon / (C_{K_\mu} C_K)$ for all $h \in U_i$ and $ i \in \cI$. 
We may assume that $U_i \subset  O \subset K_\mu$.
Let $\{ \nu_i \}_{i \in \cI}$ be a corresponding partition of unity and set $\beta := \sum_{ i\in \cI} \mu (g_i)^{1/2} \nu_i$.
We note that $\| \beta - \mu^{1/2} \|_\infty <\varepsilon/ (C_{K_\mu} C_K).$

Set $\theta_i:=\nu_i^{1/2}$ and $\xi:=\sum_{i \in \cI} \mu (g_i)^{1/2} \lambda ( \theta_i )^* \otimes \lambda( \theta_i) $.
By the preceding lemma this $\xi$ is $C ( \cG^\unit)$-central.
For $\varphi \in \fF$ we denote by $\widehat{\varphi}$ the element in $C_\cpt (\cG)$ corresponding to $\Omega_\xi (\lambda_\xi)$.
We fix $g \in K$ arbitrarily and set $y:=r (g)$.
For any $\psi \in C_\cpt (\cG)$ we define $\psi (\star):=0$.
When $(U_i)_y \neq \emptyset$ we denote by $h_i$ a unique element in $(U_i)_y$.
When $(U_i)_y =\emptyset$, we set $h_i:=\star$.
We also define $\star g:=\star$.
We show that
\begin{equation}\widehat{\varphi} (g) = \left( \sum_{i, j \in \cI} \mu (g_i )^{1/2}\mu (g_j)^{1/2} \nu_i (h_i) \nu_j ( h_ig) \right) \varphi (g). \label{eq-grpd}
\end{equation}
Firstly, letting $\varphi_{ij}:=\theta_i * \varphi * \theta_j^* |_{\cG^\unit}$ we have 
$$\i< \xi, \lambda (\varphi)  \xi >=  \sum_{i,j \in \cI}\mu (g_i)^{1/2}\mu (g_j)^{1/2} \lambda (\theta_i) ^* E ( \lambda ( \theta_i * \varphi * \theta_j^* ) ) \lambda ( \theta_j)  = \sum_{i,j \in \cI} \mu (g_i)^{1/2}\mu (g_j)^{1/2} \lambda ( \theta_i^* *\varphi_{ij} * \theta_j ).$$
Hence, for $i,j \in \cI$ we have
\begin{equation}
\theta_i^* *\varphi_{ij} * \theta_j  (g)
=\sum_{h \in \cG^y} \sum_{k \in \cG^{s(h)}} \theta_i^* (h) \varphi_{ij}(k ) \theta_j (k^{-1}h^{-1}g) 
=\theta_i (h_i)  \varphi_{ij}(r(h_i)) \theta_j (h_ig). \label{eq-grpd-2}
\end{equation}
Set $z:=r (h_i)$.
The equation (\ref{eq-grpd-2}) implies that $h_i g$ belongs to $(U_j)^{z}$ whenever $\theta_i^* *\varphi_{ij} * \theta_j  (g) \neq 0$.
In this case, we have
\begin{align*}
\varphi_{ij} (z)
&=\theta_i * \varphi * \theta_j^* (z)
=\sum_{h \in \cG^z} \sum_{k \in \cG^{s(h)}} \theta_i (h ) \varphi (k) \theta_j^* (k^{-1} h^{-1}z )\\
&=\theta_i (h_i) \sum_{k \in \cG^y} \varphi (k) \theta_j (h_i k)
=\theta_i (h_i) \varphi (g) \theta_j (h_ig).
\end{align*}
Thus, we get (\ref{eq-grpd}).

Next, we take a subset $\cJ \subset \cI$ in such a way that the sets $\cI_i:= \{j \in \cI \mid h_i = h_j \}, i \in \cJ$ satisfy that $\cI = \bigsqcup_{i \in \cJ}  \cI_i$.
We then have
\begin{align*}
\sum_{i,j \in \cI}\mu (g_i)^{1/2}\mu (g_j)^{1/2}\nu_i (h_i) \nu_j (h_ig) 
&= \sum_{i \in \cI}\mu (g_i)^{1/2}\nu_i (h_i) \sum_{j \in \cI} \mu (g_j)^{1/2} \nu_j (h_ig) \\
&= \sum_{i_0 \in \cJ} \sum_{i \in \cI_{i_0}} \mu (g_i)^{1/2}\nu_i (h_i) \sum_{j \in \cI} \mu (g_j)^{1/2} \nu_j (h_{i_0}g) \\
&= \sum_{i_0 \in \cJ} ( \sum_{i \in \cI_{i_0}} \mu (g_i)^{1/2}\nu_i (h_{(i_0)})) ( \sum_{j \in \cI} \mu (g_j)^{1/2}\nu_j (h_{i_0}g) )\\
&= \sum_{i \in \cJ} \beta (h_i) \beta (h_ig).
\end{align*}
By the choice of $\mu$ we have
$
| 1 - \sum_{i \in \cJ} \mu ( h_i) |= | 1 - \sum_{h \in \cG_y} \mu (h) | <  \varepsilon/ C_K.
$
Since $\| \varphi \|_\infty \leq 1$, $|\cJ| \leq C_{K_\mu }$,
$| \beta (h_i) | \leq \| \beta - \mu^{1/2} \|_\infty + | \mu(u_{i, y} )^{1/2} | \leq 1 + \sqrt{3} \leq 3$, and $| \beta (h_ig) | \leq \| \beta - \mu^{1/2} \|_\infty + | \mu(u_{y}^{(i)} g )^{1/2} | \leq 1 + \sqrt{3} \leq 3$, 
we have
\begin{align*}
&| \widehat{\varphi} (g) - \varphi (g) |\\
&\quad \leq | \sum_{i \in \cJ} \beta (h_i) \beta (h_ig) - \mu(h_i) | + \frac{ \varepsilon}{C_K} \\
&\quad\leq | \sum_{i \in \cJ} \beta (h_i) \beta (h_ig) - \mu(h_i)^{1/2}\mu (h_ig)^{1/2} |
+  | \sum_{i \in \cJ} \mu (h_i)^{1/2} (  \mu (h_ig )^{1/2} - \mu (h_i)^{1/2} ) |
+ \frac{ \varepsilon}{C_K}\\
&\quad\leq  6 C_{K_\mu} \| \beta - \mu^{1/2}\| _\infty
+  | \sum_{i \in \cJ} \mu (h_i)^{1/2} (  \mu (h_ig )^{1/2} - \mu (h_i)^{1/2} ) |
+ \frac{ \varepsilon}{C_K}\\
&\quad\leq   \frac{  6 \varepsilon}{C_K}
+  \sum_{i \in \cJ} \mu (h_i)^{1/2} |  \mu (h_ig )^{1/2} - \mu (h_i)^{1/2} |
+ \frac{ \varepsilon}{C_K}.
\end{align*}
By the Cauchy--Schwartz inequality and the fact that $|a^{1/2} - b^{1/2} |^2 \leq | a^{1/2} - b^{1/2}| (a^{1/2} + b^{1/2}) =| a - b|$ for all positive real numbers $a,b >0$,
we have
\begin{align*}
\sum_{i \in \cJ} \mu (h_i )^{1/2} |  \mu (h_ig )^{1/2} - \mu (h_i)^{1/2} ) | 
& \leq \left( \sum_{i \in \cJ} \mu (h_i) \right)^{1/2}\left( \sum_{i \in \cJ} |  \mu (h_ig )^{1/2} - \mu (h_i)^{1/2}  |^2 \right)^{1/2} \\
& \leq \left( \sum_{h \in \cG_y} \mu (h) \right)^{1/2}\left( \sum_{k \in \cG_y} |  \mu (kg ) - \mu (k)  | \right)^{1/2}\\
& \leq \frac{2 \varepsilon}{C_K}.
\end{align*}
Therefore, we have
$| \widehat{\varphi} (g) - \varphi (g) | \leq  9\varepsilon/C_K$
for all $g\in K$.
Since $\supp (\varphi ) \subset K$ and $\supp (\widehat{\varphi} ) \subset K$,
we now have
\begin{align*}
\| \varphi - \widehat{\varphi} \|_{I,s}
&= \sup_{x \in \cG^\unit} \sum_{g \in \cG_x }| \varphi (g) - \widehat{\varphi} (g) |
= \sup_{x \in \cG^\unit} \sum_{g \in K_x }| \varphi (g) - \widehat{\varphi} (g) |\\
&= \sup_{g \in K} C_K | \varphi (g) - \widehat{\varphi} (g) |
\leq 9 \varepsilon.
\end{align*}
Since $K=K^{-1}$, we also have
$\| \varphi^* - \widehat{\varphi}^* \|_{I,s} \leq 9 \varepsilon$.
Therefore, we get $\| \lambda (\varphi) - \Omega_\xi (\lambda( \varphi) ) \| < 9 \varepsilon$ for all $\varphi \in \fF$.
Hence, we obtain a net of $C (\cG^\unit)$-central vectors $\{ \xi_i \}_i$ such that $\Omega_{\xi_i}$ converges to the identity map on the dense subspace $C_\cpt (\cG) \subset \rC^*_\re (\cG)$.
In particular, we have $\| \Omega_{\xi_i} \|^2 = \| \i< \xi_i, \xi_i > \| \leq 2$.
Thus, $\{ \Omega_{\xi_i} \}_i$ forms a bounded net, and hence converges on the whole $\rC^*_\re (\cG)$.
\end{proof}

\setcounter{theorem}{0}
\renewcommand{\thetheorem}{\arabic{section}.\arabic{theorem}}

%
%
\section{Weyl--von Neumann--Voiculescu type results}\label{sec-voic}

The main result in this section is Theorem \ref{thm-voic} below, which says that weak containment is characterized by a certain Weyl--von Neumann--Voiculescu type assertion.
Our proof is based on Arveson's argument \cite{Arveson}.
We use the condition (3) in Theorem \ref{thm-weak} and Corollary \ref{cor-weak} instead of Glimm's lemma.

We denote by $\{ \delta_n \}_{n=1}^\infty$ the canonical basis of $\ell^2( \lN)$ and by $p_n \in \lB (\ell^2( \lN))$ the orthogonal projection onto $\lC \delta_n$.
Recall that a $\rC^*$-algebra $B$ is said to be $\sigma$-{\it unital} if it admits a countable approximate unit. 
Recall that every separable $\rC^*$-algebra is $\sigma$-unital.
The following lemma can be found in the proof of \cite[Lemma 10]{Kasparov}.
Thus, we give only a sketch of proof for the reader's convenience.
\begin{lemma}\label{lem-quasicentral}
Let $A$ and $B$ be $\rC^*$-algebras with $A$ unital separable and $B$ $\sigma$-unital.
Then, for any finite subset $\fF \subset A$ and $\varepsilon >0$ there exists a sequence of positive elements $\{ e_n \}_{n=1}^\infty \subset \lK_B (H_B)$ such that each $e_n$ has the form of $\sum_{i=1}^{d(n)} p_i \otimes b_i$ for some $b_1, \dots, b_{d(n)} \in B$, $\sum_{n=1}^\infty e_n^2 =1_{H_B}$ strictly, and
\begin{align*}
a - \sum_{n=1}^\infty e_n a e_n \in \lK_B (H_B) \;\; \text{for}\; a \in A,
\quad 
\| b - \sum_{n=1}^\infty e_n b e_n \| < \varepsilon  \;\;\text{for}\; b \in \fF,
\end{align*}
where the infinite sums converge strictly.
\end{lemma}
\begin{proof}[Sketch]
Set $\fF_1:=\fF$ and take an increasing sequence of finite subsets $\fF_1 \subset \fF_2 \subset \cdots \subset A$ such that $\bigcup_{n=1}^\infty \fF_n$ is norm dense in $A$.
Take a countable approximate unit $\{ v_n \}_{n \geq 1}$ of $B$.
Consider the following two separable $\rC^*$-subalgebras of $\lL_B (H_B)$,
$$
\cA:=\rC^* (1_{H_B}, A, p_n \otimes v_m, n\in \lN, m \in \lN), \quad \cJ:= \rC^* (p_n \otimes v_m, n\in \lN, m \in \lN). 
$$
Since $\{ \sum_{i=1}^np_i \otimes v_n \}_{n=1}^\infty$ forms an approximate unit of $\cJ \subset \lK_B (H_B) \cong \lK (\ell^2 (\lN)) \otimes B$,
applying \cite[Theorem 1]{Arveson} to $\{ \sum_{i=1}^np_i \otimes v_n \}_n$ and $\cJ \triangleleft \cA$ we obtain a countable quasicentral approximate unit $\{ u_n \}_{n=1}^\infty$ contained in the convex hull of $\{ p_n \otimes v_m \}_{n,m} \subset \lK_B (H_B)$ satisfying that $\| [ a , (u_{n}- u_{n-1})^{1/2} ] \| < \varepsilon /2^n$ for $a \in \fF_n$.
Here we set $u_0:=0$.
Then, the $e_n:=(u_n-u_{n-1})^{1/2}$ is the desired one.
\end{proof}

\begin{lemma}\label{lem-glimm}
Let $A$ and $B$ be $\rC^*$-algebras with $A$ unital and $(Y, \pi_Y) \in \Corr(A,B)$ be unital.
Let $\varphi \in \CP ( A, \lL_B (H_B) )$, $(K,\pi_K) \in {\rm Rep}(B)$, and $b_1,\dots, b_n \in B$ be given.
For the compact operator $e = \sum_{i=1}^n p_i \otimes b_i \in \lK_B (H_B)$,
we define the c.p.\ map
$$
\psi : A\to \lK_B(H_B); \quad a \mapsto e^* \varphi (a) e.
$$
If $(A\otimes_\varphi H_B, \lambda_A \otimes 1_{H_B}) \prec_K Y$,
then for any finite subsets $\fF \subset A$, $\cX \subset \lB (H_B \otimes_B K)_*$, and $\varepsilon>0$
there exist $m \in \lN$ and $V \in \lL_B (H_B, Y \otimes \lC^m)$ such that $V^* V \leq e^*\varphi (1)e$ and
$$
\left| f ( ( \psi (a)- V^* \pi^m_Y (a) V)\otimes1_K ) \right| < \varepsilon, \quad a\in \fF, f\in \cX.
$$
When the given $(K, \pi_K)$ is the universal representation of $B$, the above $V$ can be chosen in such a way that $\| \psi(a) - V^*\pi_Y^m  (a)V \| <\varepsilon $ for $a \in \fF$.
\end{lemma}

\begin{proof}
For each $a \in A$, the support and range of $\psi (a)$ are contained in $\lC^n \otimes B$, we can regard $\psi(a)$ as an element in $\lK_B (\lC^n \otimes B)$.
Put $\zeta:=( 1_A \otimes (\delta_1 \otimes b_1), \dots,1_A \otimes  (\delta_n \otimes b_n) ) \in (A\otimes_\varphi H_B)_n$.
Then, via the isomorphisms $\lK_B (\lC^n \otimes B) \cong \lM_n (B)$ and $\lB ((\lC^n \otimes B) \otimes_B K) \cong \lB (K^n)$,
the operator $\psi (a)$ and $\psi (a) \otimes 1_K$ are identified with $\Omega_\zeta (a)$ and $\pi_{K}^{(n)} \circ \Omega_\zeta (a)$, respectively.
By Lemma \ref{lem-tensor}, we have $(A\otimes_\varphi H_B)_n \prec_{K^n} Y_n$.
Thus, there exist $\xi^{(1)}, \dots, \xi^{(m)} \in Y_n$ such that
\begin{equation}
 \left|f \circ \pi_{K}^{(n)} \left( \Omega_\zeta (a) - \sum_{k=1}^m \i< \xi^{(k)}, \pi_{Y_n} (a) \xi^{(k)} >_{\lM_n (B)} \right) \right| < \varepsilon,\quad a\in \fF, f \in \cX \label{eq-Glimm}
\end{equation}
and $\sum_{k=1}^m \i< \xi^{(k)}, \xi^{(k)} > \leq \Omega_\zeta (1) =e^*\varphi(1) e$.
Write $\xi^{(k)}=(\xi_1^{(k)}, \dots, \xi_n^{(k)} )\in Y_n$ with $\xi_i^{(k)} \in Y$ and set
$\xi_i:=(\xi_i^{(1)}, \dots, \xi_i^{(m)} ) \in Y^m$.
Define $V \in \lL_B (H_B, Y^m)$ by
$V (\delta_i \otimes b) = \xi_i  b$ for $1 \leq i \leq n$ and $V (\delta_i \otimes b)=0$ for $n<i$.
For $a \in A$, via the isomorphism $\lK_B (\lC^n \otimes B) \cong \lM_n (B)$, the operator $V^* \pi_Y^m (a) V$ is identified with $[ \i< \xi_i, \pi_Y^{m} (a) \xi_j >]_{i,j=1}^n$.
On the other hand, one has
\begin{align*}
\sum_{k=1}^m \i< \xi^{(k)} , \pi_{Y_n} (a) \xi^{(k)} >_{\lM_n (B)} 
=\left[ \sum_{k=1}^m \i< \xi_i^{(k)}, \pi_Y (a) \xi_j^{(k)} >_B \right]_{i,j=1}^n 
=\left[ \i< \xi_i, \pi_Y^{m} (a) \xi_j >_B \right]_{i,j=1}^n
\end{align*}
Hence, $V$ is the desired one.

In the case that $A \otimes_\varphi H_B \prec_\univ Y$,
by Corollary \ref{cor-weak},
we can choose $\xi^{(1)}, \dots, \xi^{(m)}$ in such a way that
$\| \Omega_\eta (a) - \sum_{k=1}^m \i< \xi^{(k)}, \pi_{Y_n} (a) \xi^{(k)} > \| < \varepsilon$ for $a \in \fF$ instead of (\ref{eq-Glimm})
\end{proof}

\begin{theorem}\label{thm-voic}
Let $A$ and $B$ be $\rC^*$-algebras with $A$ unital and $B$ $\sigma$-unital and $X$ be a countably generated Hilbert $B$-module.
For unital $(Y, \pi_Y) \in \Corr(A,B)$ and given $\varphi \in \CP ( A, \lL_B (X) )$ and $(K,\pi_K) \in {\rm Rep}(B)$, the following are equivalent:
\begin{itemize}
\item[$(1)$] $(A\otimes_\varphi X, \lambda_A \otimes 1_X) \prec_{(K,\pi_K)} (Y, \pi_Y)$.
\item[$(2)$] There exists a net $V_i \in \lL_B (X, Y^\infty)$ with $\|V_i^*V_i\| \leq \| \varphi (1_A) \| $ and $( V_i^*\pi^\infty_Y (a)V_i ) \otimes 1_K$ converges to $\varphi(a) \otimes 1_K$ in the $\sigma$-weak topology on $\lB (X \otimes_B K)$ for all $a \in A$.
\end{itemize}
Further suppose that $\varphi$ is unital, and $(K,\pi_K)$ is the universal representation of $B$.
Then, any of conditions above is also equivalent to
\begin{itemize}
\item[$(3)$] There exists a net of isometries $V_i \in \lL_B (X, Y^\infty)$such that $\lim_{i} \| V_i^* \pi^\infty_Y (a) V_i- \varphi (a) \| =0$ for all $a\in A$.
When $A$ is separable, we can choose the $V_i$ as a sequence satisfying that $V_i^*  \pi^\infty_Y (a)  V_i- \varphi (a) \in \lK_B (X)$ for all $a \in A$.
\end{itemize}
\end{theorem}
\begin{proof}
First, we show (1) $\Rightarrow$ (2):
It suffices to show the case that $X=H_B$.
Indeed, by Kasparov's stabilization theorem there exists a projection $P \in \lL_B (H_B)$ such that $PH_B\cong X$.
Hence, we can regard $\varphi$ as a map into $\lL_B (H_B)$.
Since $A \otimes_\varphi X =A\otimes_\varphi H_B$, we get $A\otimes_\varphi H_B \prec_K X$.
If we obtain a net $V_i \in \lL_B (H_B, Y^\infty)$ as in (2), then $V_iP \in \lL _B (X,Y^\infty)$ does the job.
Fix $\varepsilon>0$ and finite subsets $\fF \subset A$ and $\cX \subset \Ball ( \lB (H_B \otimes_B K)_*)$ arbitrarily.
It suffices to show that there exists $V \in \lL_B (H_B, Y^\infty )$ such that $\| V^*V  \| \leq \| \varphi(1) \| $ and
$$
| f ((\varphi (a) - V^* \pi_Y^\infty (a) V) \otimes 1_K )| < \varepsilon, \quad a\in \fF, f\in \cX.
$$
When $A$ is separable, take an increasing sequence of finite subsets $\fF=\fF_1 \subset \fF_2 \subset \dots \subset A$ such that $A $ is the norm closure of $\bigcup_{n=1}^\infty \fF_n$. When $A$ is non-separable, set $\fF_n:=\fF$ for all $n\in \lN$.
Let $\cA$ be the separable $\rC^*$-subalgebra of $\lL_B (H_B)$ generated by $\bigcup_{n=1}^\infty \varphi (\fF_n)$.
By Lemma \ref{lem-quasicentral} we obtain a sequence of positive operators $\{ e_n \}_{n=1}^\infty$ such that each $e_n$ has the form of $\sum_{i=1}^{d(n)} p_i \otimes b_i$, $\sum_{n=1}^\infty e_n^2=1_{H_B}$ strictly, and
\begin{align*}
x - \sum_{n=1}^\infty e_n x e_n \in \lK_B (H_B) \;\; \text{for}\; x \in \cA,
\quad 
\| \varphi (a) - \sum_{n=1}^\infty e_n \varphi (a) e_n \| < \varepsilon /2  \;\;\text{for}\; a \in \fF.
\end{align*}
Let $\psi_n : A \to \lL_B (H_B )$ be the c.p.\ map defined by $e_n$ as in Lemma \ref{lem-glimm}.
By Lemma \ref{lem-glimm} there exist $d(n) \in \lN$ and $W_n \in \lL_B ( H_B, Y^{d(n)})$ such that $W_n^*W_n \leq e_n \varphi (1) e_n$ and
\begin{equation}
\left| f \left(  ( \psi_n(a) - (W^*_n \pi_Y ^{d(n)} (a) W_n ) \otimes 1_K \right) \right| < 2^{-n} \varepsilon, \quad a \in \fF_n, \; f \in \cX. \label{eq-voic}
\end{equation}
For any $\xi \in H_B$ and $N \in \lN$ one has $\sum_{n=1}^N \i< W_n \xi ,W_n\xi > \leq \| \varphi(1) \| \i< \xi, \sum_{n=1}^N  e_n^2 \xi > \leq \| \varphi (1) \| \i< \xi, \xi>$, and hence $V:=\bigoplus_n W_n : H_B  \to \bigoplus_n Y^{d(n)} \cong Y^\infty$ is well-defined and satisfies that $V^*V \leq \| \varphi (1) \| 1$.
We observe that that $V^* \pi^\infty_Y (a)  V = \sum_{n=1}^\infty W_n^* \pi_Y^{d(n)} (a) W_n$ for $a\in A$.
Now for any $a\in \fF$ and $f\in \cX$ we have
\begin{align*}
&| f( ( \varphi(a) - V^*\pi_Y^\infty (a) V )\otimes 1_K) | \\
&\qquad \leq \| \varphi (a) - \sum_{n=1}^\infty e_n \varphi (a) e_n \|
+ |f(  \sum_{n=1}^\infty ( \psi_n (a) -  W_n^* \pi^{d(n)}_Y (a) W_n )\otimes 1_K  ) |\\
&\qquad  \leq \| \varphi (a) - \sum_{n=1}^\infty e_n \varphi (a) e_n \|
+ \sum_{n=1}^\infty \left| f \left( (  \psi_n(a) - W^*_n \pi_Y^{d(n)} (a) W_n) \otimes 1_K \right) \right|\\
&\qquad < \varepsilon.
\end{align*}

In the case that $A \otimes_\varphi H_B \prec_\univ Y$,
by Lemma \ref{lem-glimm} we can choose $W_n$ in such a way that
\begin{equation*}
\left \|  \psi_n(a) - W^*_n \pi_Y^{d(n)} (a) W_n \right\| < 2^{-n} \varepsilon, \quad a\in \fF_n
\end{equation*}
instead of (\ref{eq-voic}).
This implies that $\| \varphi(a) - V^*\pi^\infty_Y (a) V \| <\varepsilon$ for all $a\in \fF$.
We also have $\sum_{n=1}^\infty \| \psi_n (a) - W_n^* \pi_Y^{d(n)} (a)  W_n \| < \infty$ holds for all $a \in \bigcup_n \fF_n$.
Thus, if $A$ is separable, then $\varphi(a) - V^*\pi^\infty_Y (a) V \in \lK_B (H_B)$ for all $a\in A$. Moreover, if $\varphi$ is unital, then we can choose $V$ in such a way that $\| 1 - V^*V \| < \varepsilon<1$. Set $V_0:=V(V^*V)^{-1/2}$. Then $V_0$ is an isometry and enjoys that $\| \psi(a) - V_0^*\pi^\infty_Y(a) V_0  \| < \delta(\varepsilon)$ for $a\in \fF$, where $\delta(\varepsilon)$ is a positive number such that $\lim_{\varepsilon \searrow 0} \delta(\varepsilon)=0$, and hence we get (1) $\Rightarrow$ (3).

Finally, we prove (2) $\Rightarrow$ (1):
Suppose that we have a net $V_i \in \lL_B (X, Y^\infty)$ in (2).
Fix $\xi = \sum_{k=1}^m a_k \otimes \xi_k \in A\odot H_B$ arbitrarily.
Set $\zeta_i := \sum_{k=1}^m \pi_Y^{\infty} (a_k) V_i \xi_k \in Y^\infty$ and show that $\pi_K \circ\Omega_{\zeta_i}$ converges to $\Omega_\xi$ point $\sigma$-weakly.
Since $\{ \pi_K \circ\Omega_{\zeta_i} \}_i$ is norm bounded, it suffices to show convergence in the point weak operator topology.
For any $a\in A$ and $\eta \in K$ we have
\begin{align*}
\i< \eta, \pi_K\circ \Omega_{\zeta_i} (a) \eta >
&= \sum_{k,l=1}^m \i< \xi_k \otimes \eta, ( V_i^* \pi_Y^\infty (a_k a a_l ) V_i \xi_l ) \otimes \eta >\\
&\rightarrow \sum_{k,l=1}^m \i< \xi_k \otimes \eta, (\varphi (a_k a a_l ) \xi_l ) \otimes \eta >\\
&=\pi_K \circ  \Omega_\xi (a).
\end{align*}
By Theorem \ref{thm-weak} we get $A\otimes_\varphi H_B \prec_K Y^\infty$.
Since $Y^\infty \prec_\univ Y$ holds, we are done.
\end{proof}

Here is a characterization of weak containment for unital countably generated C$^*$-correspondences.
\begin{corollary}
Let $A$ and $B$ be $\rC^*$-algebras with $A$ unital and $B$ $\sigma$-unital.
For unital $\rC^*$-correspondences $X,Y \in \Corr(A,B) $ with $X$ countably generated and a representation $K \in \Rep (B)$, $X$ is weakly contained in $Y$ with respect to $K$ if and only if there exists a net of contractions $V_i \in \lL_B (X, Y^\infty)$ and $V_i^* \pi_Y^\infty (a) V_i \otimes 1_K$ converges to $\pi_X(a)  \otimes 1_K$ $\sigma$-weakly in $\lB (X \otimes_B K)$ for all $a\in A$.
\end{corollary}

\begin{corollary}\label{cor-voic}
Let $A$ and $B$ be $\rC^*$-algebras with $A$ unital and $B$ $\sigma$-unital.
For unital $\rC^*$-correspondences $(X,\pi_X), (Y,\pi_Y ) \in \Corr(A,B) $ with $X$ countably generated, $(X, \pi_X) \prec_\univ (Y,\pi_Y)$ if and only if there exists a net of isometries $V_i \in \lL_B (X, Y^\infty)$ such that $\lim_i \| V_i \pi_X(a) - \pi^\infty_Y (a) V_i \| =0$ for all $a\in A$.
When $A$ is separable, we may choose $V_i$'s as a sequence satisfies that $V_i \pi_X (a) - \pi^\infty_Y (a) V_i \in \lK_B (X, Y^\infty)$ for all $a\in A$.
\end{corollary}
\begin{proof}
The ``if part'' follows from the preceding theorem.
Let $V \in \lL_B (X, Y^\infty )$ be an isometry.
One has
\begin{align*}
(V \pi_X (a) - \pi^\infty_Y (a) V)^*(V \pi_X (a) - \pi^\infty_Y (a) V) 
&=\pi_X (a^*) (\pi_X (a)-  V^* \pi_Y^\infty (a) V)\\
&+ (\pi_X (a^*) - V^*\pi_Y^\infty (a^*)V )\pi_X (a) \\
&+ V^* \pi_Y^\infty (a^*a ) V - \pi_X (a^*a ).
\end{align*}
Thus, the assertion follows from Theorem \ref{thm-voic}.
\end{proof}
We note that $V_i$ in the preceding corollary satisfies $\pi^\infty_Y (a) V_iV_i^* - V_i V_i^* \pi^\infty_Y (a)$ converges to 0 in norm, and is compact if $A$ is separable.
Indeed,
$$
\pi_Y^\infty (a) V_iV_i^* - V_i V_i^* \pi_Y^\infty (a)
=(\pi_Y^\infty (a) V_i - V_i\pi_X (a) )V_i^* +V_i( \pi_X (a)  V_i^* - V_i^* \pi_Y^\infty (a)). 
$$

\begin{definition}
Let $A$ and $B$ are $\rC^*$-algebras. For $(X,\pi_X), (Y,\pi_Y) \in \Corr (A,B)$
we say that $(X,\pi_X)$ and $(Y,\pi_Y)$ are {\it approximately unitarily equivalent}, written $(X,\pi_X) \sim (Y,\pi_Y)$ if there exists a sequence of unitaries $U_n \in \lL_B (X,Y)$ such that $\pi_X (a) - U^*_n \pi_Y(a) U_n \in \lK_B (X)$ and $\lim_{n\to \infty} \|\pi_X (a) - U^*_n \pi_Y(a) U_n  \|=0$ for all $a \in A$.
\end{definition}

We introduce a notation which is useful to prove the next theorem.
Let $X$ and $Y$ be Hilbert $B$-modules and $\varphi : A \to \lL_B (X)$ and $\psi :A\to \lL_B(Y)$ be maps.
For a subset $\fF \subset A$ and $\varepsilon>0$ we denote by $\varphi \sim_{(\fF, \varepsilon)} \psi$ if
there exists a unitary $U \in \lL_B (X,Y)$ such that $\varphi (a)- U^* \psi (a) U \in \lK_B (X)$ and $\| \varphi(a)- U^* \psi (a) U \| <\varepsilon$ for all $a\in \fF$.
We also write $\varphi \sim \psi$ if $\varphi \sim_{(\fF, \varepsilon)} \psi$ for any finite subset $\fF \subset A$ and $\varepsilon>0$.

\begin{theorem}\label{thm-absorbing}
Let $A$ and $B$ be $\rC^*$-algebras with $A$ unital and $B$ $\sigma$-unital.
For unital $\rC^*$-correspondences $(X, \pi_X)$, $(Y, \pi_Y) \in \Corr(A,B)$ with $X$ countably generated,
$(X, \pi_X) \prec_\univ (Y,\pi_Y)$ holds if and only if there exists a net of unitaries $U_i \in \lL_B (X \oplus Y^\infty,Y^\infty)$ such that $\lim_{i} \|\pi_X (a) \oplus \pi_Y^\infty (a) - U^*_i \pi_Y^\infty(a) U_i\|=0$ for all $a \in A$.
When $A$ is separable, this is the case that $(X \oplus Y^\infty, \pi_X \oplus \pi_Y^\infty) \sim (Y^\infty, \pi_Y^\infty)$.
\end{theorem}
\begin{proof}
Suppose that $(X, \pi_X) \prec_\univ (Y,\pi_Y)$.
We only prove the case that $A$ is separable since the proof for general $A$ proceeds in the same manner.
It suffices to show that $\pi_X \oplus \pi_Y^\infty  \sim \pi_Y^\infty$.
Since $X^\infty \prec_\univ X \prec_\univ Y$ and $X^\infty$ is also countably generated,
we can apply the previous corollary and obtain a sequence of isometries $V_n \in \lL_B (X^\infty, Y^\infty)$ such that $V_n \pi_X^\infty (a) - \pi_Y^\infty (a)  V_n$ is compact and converges to 0 in norm for all $a \in A$.
Put $P_n:=1-V_nV_n^* \in \lL_B (Y^\infty)$.
Then $U_n:=V_n\oplus P_n : X^\infty \oplus P_n Y^\infty\to Y^\infty$ is a unitary.
Fix $a\in A$. We have
\begin{align*}
\pi^\infty_Y (a) U_n - U_n ( \pi_X^\infty (a) \oplus P_n \pi_Y^\infty (a) P_n ) 
&=\pi^\infty_Y (a) P_n + \pi^\infty_Y (a) V_n - P_n \pi^\infty_Y (a) P_n - V_n \pi_X^\infty (a)\\
&=\pi^\infty_Y (a) V_nV_n^*  - V_nV_n^* \pi^\infty_Y (a) + \pi^\infty_Y (a) V_n - V_n \pi_X^\infty (a),
\end{align*}
which is compact and converges to 0 by the remark above.
Hence, for any finite subset $\fF \subset A$ and $\varepsilon >0$ there exists $n \in \lN$ such that
$\pi_Y^\infty \sim_{(\fF, \varepsilon)} \pi_X^\infty \oplus P_n \pi_Y^\infty(\cdot) P_n$.
Now we have
$$
\pi_X \oplus \pi^\infty_Y \us\sim_{(\fF, \varepsilon)} \pi_X \oplus \pi_X^\infty \oplus P_n \pi_Y^\infty(\cdot) P_n \sim \pi_X^\infty \oplus P_n \pi_Y^\infty (\cdot) P_n \us\sim_{(\fF, \varepsilon)} \pi^\infty_X.
$$
This implies that $\pi_X\oplus \pi_Y^\infty \sim_{(\fF, 3\varepsilon)} \pi_Y^\infty$.

Conversely, suppose that $(X\oplus Y^\infty, \pi_X \oplus \pi_Y^\infty ) \sim (Y^\infty, \pi^\infty_Y)$.
Then, we have $(X\oplus Y^\infty , \pi_X \oplus \pi_Y^\infty ) \prec_\univ (Y^\infty, \pi^\infty_Y)$, and hence
$(X, \pi_X) \prec_\univ(X\oplus Y^\infty, \pi_X \oplus \pi_Y^\infty)  \prec_\univ  (Y^\infty, \pi_Y^\infty) \prec_\univ  (Y, \pi_Y).
$
\end{proof}

\begin{definition}[{\cite[Definition 1.6]{Skandalis}}]
An $A$-$B$ C$^*$-correspondence $X$ is said to be {\it nuclear} if for any $n \in \lN$ and any $\xi = ( \xi_1, \dots, \xi_n) \in X_n=X\otimes \lC_n$ with $\| \xi \| \leq 1$,
the c.c.p.\ map $\Omega_\xi : A \to \lM_n (B); a \mapsto [ \i< \xi_i, \pi_X(a) \xi_j >_B ] _{i,j=1}^n$ is nuclear.
\end{definition}

The nuclearity of a given C$^*$-correspondence is characterized in terms of our weak containment.
Recall that a c.p.\ map $\theta : A \to B$ is said to be {\it factorable} if there exist $n \in \lN$ and c.p.\ maps $\varphi : A \to \lM_n$ and $\psi : \lM_n \to A$ such that $\theta = \psi \circ \varphi$.
The set of factorable maps from $A$ into $B$ is known to be convex.
By \cite[Proposition 3.8.2]{Brown-Ozawa} any c.c.p.\ map is nuclear if and only if it can be approximated by factorable c.p.\ maps in the point norm topology.
\begin{proposition}\label{prop-nuclear-bimodule}
Let $A$ and $B$ be $\rC^*$-algebras with $A$ unital. 
Let $H \in \Rep (A)$ be faithful and $X \in \Corr (A, B)$ and $K \in \Rep (B)$ be given.
Then, $X \prec_\univ (H \otimes B, \pi_H \otimes 1_B)$ holds if and only if $X$ is nuclear. 
\end{proposition}
\begin{proof}
Suppose that $X \prec_\univ H \otimes B$.
Since every c.c.p.\ map in $\cF_{(H\otimes B)_n}$ can be approximated by factorable maps for $n \in \lN$, $(H \otimes B, \pi_H \otimes 1_H)$ is nuclear,
and hence so is $X$ by Corollary \ref{cor-weak} and Lemma \ref{lem-tensor}.
Conversely, suppose that $X$ is nuclear and show the condition (2) in Corollary \ref{cor-weak}.
Fix $\xi \in X$ with $\| \xi \| \leq 1$, a finite subset $\fF \subset A$, and $\varepsilon >0$ arbitrarily.
Since $\Omega_\xi$ is a nuclear map, we may and do assume that $\Omega_\xi$ is of the form $\beta \circ \alpha$ for some c.c.p.\ maps $\alpha : A \to \lM_n$ and $\beta : \lM_n \to B$.
Let $(H_\alpha, \pi_\alpha, V_\alpha)$ be the Strinespring dilation of $\alpha$.
Since $(H_\alpha, \pi_\alpha)$ is weakly contained in $(H, \pi_H)$ and $\alpha=\Omega_\eta$ with $\eta =  ( V_\alpha \delta_1, \dots, V_\alpha\delta_n ) \in (H_\alpha)_n$,
by Lemma \ref{lem-tensor} again, there exist $ \xi^{(1)}, \dots, \xi^{(m)} \in H \otimes \lC_n$ such that $\| \alpha (a) - \sum_{k=1}^m \Omega_{\xi^{(k)}} (a) \| < \varepsilon$ for $a\in \fF$.
Let $[ x_{ij} ]_{i,j=1}^n \in \lM_n (M)$ be the square root of the Choi matrix $[ \beta (e_{ij}) ]_{i,j=1}^n \in \lM_n (M)$.
Note that $\beta (e_{ij}) = \sum_{l=1}^n x_{li}^*x_{lj}$.
Write $\xi^{(k)}=(\xi^{(k)}_1, \dots, \xi^{(k)}_n )$ with $\xi_i^{(k)} \in H$.
Letting $\eta_l^{(k)}:= \sum_{i=1}^n \xi_k^{(i)} \otimes x_{li} \in H \otimes M$ we have
\begin{align*}
\beta \circ \alpha (a)
&\approx_\varepsilon \sum_{k=1}^m \beta ( \Omega_{\xi^{(k) }} (a) )
=\sum_{k=1}^m \beta \left( \left[ \l< \xi_i^{(k)}, \pi_H (a) \xi_j^{(k)} >\right]_{i,j=1}^n  \right) \\
&=\sum_{k=1}^m \sum_{l=1}^n \sum_{i,j=1}^n \l< \xi_i^{(k)}, \pi_H (a) \xi_j^{(k)} > x_{li}^*  x_{lj}
=\sum_{k=1}^m \sum_{l=1}^n \Omega_{\eta_l^{(k)}} (a)
\end{align*}
for every $a\in A$, which implies $\beta \circ \alpha$ belongs to the point norm closure of $\cF_{H \otimes A}$.
\end{proof}

We should remark that our results do not include original Voiculescu's theorem \cite{Voiculescu} as well as Kasparov's generalized one \cite[Theorem 5]{Kasparov} completely.
The following corollary follows from Theorem \ref{thm-voic} and is a particular case of \cite[Theorem 5]{Kasparov}.
\begin{corollary}
Let $A$ be unital separable $\rC^*$-algebra, $B$ be $\sigma$-unital, $\pi_H : A \to \lB (H)$ be a faithful representation with $H$ separable, and $\varphi :  A \to \lL_B (H_B)$ be unital completely positive. Then $(A \otimes_\varphi H_B, \lambda_A \otimes 1_{H_B} ) $ is nuclear if and only if
there exists a sequence of isometries $V_n \in \lL_B (H_B)$ such that $\varphi (a) - V_n^* (\pi_H^\infty (a) \otimes 1_B ) V_n \in \lK_B (H_B)$ and $\lim_{n \to \infty } \| \varphi (a) - V^*_n (\pi_H^\infty (a) \otimes 1_B ) V_n \|=0$ and $a\in A$.
\end{corollary}

\renewcommand{\thetheorem}{\arabic{section}.\arabic{subsection}.\arabic{theorem}}

%
%
\section{Relative $K$-nuclearity}\label{sec-rel-K}
\subsection{Preliminaries on $KK$-theory}\label{ss-rel-K-pre}
In this section we prove that our strong relative nuclearity implies Germain's relative $K$-nuclearity.
Firstly, we recall some definitions and facts on $KK$-theory.
We refer to \cite{Blackadar} and \cite{Jensen-Thomsen} for $KK$-theory.
\begin{notation}
For a trivially graded $\rC^*$-algebra $B$, a {\it graded Hilbert $B$-module} is a Hilbert $B$-module $X$ such that there exist closed submodules $X_0$ and $X_1$ of $X$, called {\it even} and {\it odd} parts of $X$, such that $X=X_0 \oplus X_1$.
To make the grading clear, we will write $X=X_0 \hoplus X_1$.
An operator $x \in \lL_B (X_0 \hoplus X_1)$ is said to be of {\it degree} $i \in \{ 0, 1\}$ if $x X_j \subset X_{i+j \, (\text{mod } 2)}$.
A $*$-homomorphism $\phi: A \to \lL_B (X_0 \hoplus X_1)$ is said to be of degree 0 if $\phi (a)$ is of degree 0 for all $a\in A$.
In this case, there exist $*$-homomorphism $\phi_0: A \to \lL_B (X_0)$ and $\phi_1 : A \to \lL_B (X_1)$ such that $\phi (a) = \phi_0 (a) \oplus \phi_1 (a)$ for $a\in A$.
We will write $\phi =\phi_0 \hoplus \phi_1$.
\end{notation}

\begin{definition}
For (trivially graded) $\rC^*$-algebras $A$ and $B$, a {\it Kasparov $A$-$B$ bimodule} is a triple $(X, \phi, F )$ such that $X$ is a countably generated graded Hilbert $B$-module, $\phi :A \to \lL_B ( X)$ is a $*$-homomorphism of degree 0,
and $F \in \lL_B (X)$ is of degree 1 and satisfies the following condition:
\begin{itemize}
\item $[F, \phi (a)] \in \lK_B (X)$ for $a\in A$,
\item $(F-F^*) \phi (a) \in \lK_B (X)$ for $a\in A$,
\item $(1 - F^2) \phi (a) \in \lK_B (X)$ for $a \in A$.
\end{itemize}
If $[F, \phi (a)]=(F-F^*) \phi (a)=(1 - F^2) \phi (a)=0 $ for all $a\in A$,
we call $(X, \phi, F )$ {\it degenerate}.
We denote by $\lE (A,B)$ and $\lD (A, B)$ the set of Kasparov $A$-$B$ bimodules and degenerate ones, respectively.
\end{definition}

We say that two $A$-$B$ Kasparov bimodules $(X, \phi, F)$ and $(Y, \psi, G)$ are {\it unitarily equivalent}, denoted by $(X, \phi, F) \cong (Y, \psi, G)$, if there exists a unitary $U \in \lL (X , Y)$ of degree $0$ such that $\psi = U \phi (\cdot ) U^*$ and $F= UGU^*$.

For a C$^*$-algebra $B$ we set $IB:=B \otimes C[0,1]$.
We identify $IB$ with $C([0,1], B)$, the space of $B$-valued continuous functions on $[0,1]$.
For $ t \in [0,1]$ the {\it evaluation at} $t$ is the surjective $*$-homomorphism, still written $t$, from $IB$ onto $B$ defined by $t (f):= f (t)$ for $f \in IB$.
If $X$ is a Hilbert $IB$-module, then the pushout $X_t$ of $X$ by $t$ is a Hilbert $B$-module.
For $(X, \pi_X) \in \Corr (A, IB)$ we denote by $(\pi_X)_t$ the $*$-homomorphism $A \ni a \mapsto \pi_X(a)_t \in \lL_B (X_t)$.
For a Hilbert $A$-module $X$, we set $IX:=X\otimes C[0,1]$.
We also identify $IX$ with the Hilbert $IA$-module $C([0,1], X)$ of $X$-valued continuous functions on $[0,1]$ equipped with the inner product $\i< f, g >_{IA} (t):= \i< f  (t), g(t) >_A$ for $f, g \in IX$.
The following proposition is probably well-known, but we give its proof for the reader's convenience.
\begin{proposition}\label{lem-path}
Let $X$ be a Hilbert $A$-module. 
Then, for any strict continuous, norm bounded path $ \{ x_t \}_{0 \leq t \leq 1} \in \lL_A (X)$,
there exists a unique operator $x \in \lL_{IA} (IX)$ of which the evaluation at $t$ is $x_t$.
Conversely, for any element $y \in \lL_{IA} (IX)$, the evaluations $\{y_t \}_{0 \leq t \leq 1}$ of $y$ defines a strict continuous, norm bounded path in $\lL_B (X)$.

Moreover, an operator $z$ is in $\lK_{IA} (IX)$ if and only if the corresponding path is a norm continuous path in $\lK_A (X)$.
\end{proposition}
\begin{proof}
Let $ \{ x_t \}_{0 \leq t \leq 1} \in \lL_A (X)$ be strict continuous and norm bounded.
Fix $f \in IX$ arbitrarily.
Indeed, for any $t, s\in [0,1]$ we have 
$
\| x_t f (t) - x_s f (s) \|
\leq \| (x_t - x_s )  f (t) \| +  \| x_s \| \| f (t) -  f (s) \| .
$
Since $ x_t $ is norm bounded and strict continuous, $t \mapsto x_t  f (t) $ and $t \mapsto x^*_t  f (t) $ define elements in $IX$.
Hence the mapping $x : f  \mapsto [ t\mapsto x_t  f (t) ]$ is the desired element in $\lL_{IA} (IX)$.

Conversely, let $y \in \lL_{IA} (IX)$ be given and $\{ y_t \}_{0\leq t\leq 1}$ be the corresponding evaluations. For $\xi \in X$, let $\xi \otimes 1 \in X\otimes C[0,1]= IX$ be the constant function. We then have $y_t \xi = [ y (\xi \otimes 1) ] (t)$ for all $t \in [0,1]$,
which implies $\{ y_t \}_{0 \leq t \leq 1}$ is strictly continuous.
The third assertion follows from the isomorphism $\lK (IX) \cong \lK (X) \otimes C[0,1] \cong C( [0,1], \lK (X) )$.
\end{proof}

Recall that for $\cX=(X, \phi, F) \in \lE (A, IB)$ the evaluation at $t \in [0, 1]$ is the Kasparov $A$-$B$ bimodule $\cX_t:=(X_t, \phi_t, F_t)$.
\begin{definition}
Two Kasparov $A$-$B$ bimodules $\cX$ and $\cY$ are said to be {\it homotopic} is there exists a Kasparov $A$-$IB$ bimodule $\cZ$ such that $\cX \cong \cZ_0$ and $\cY \cong \cZ_1$.
The $KK(A,B)$ is the set of homotopy equivalence classes of all $A$-$B$ Kasparov $A$-$B$ bimodules.
\end{definition}

Let $\cX=(X , \phi, F), \cY =(Y, \psi, G) \in \lE (A, B)$.
We denote by $[\cX]=[ X, \phi, F]$ and $[\cY]=[ Y, \psi, G]$ the elements in $KK (A, B)$ corresponding to $\cX$ and $\cY$, respectively.
The addition of $[\cX]$ and $[\cY]$ is defined by $[\cX]+[\cY]:=[\cX \oplus \cY]$, where $\cX\oplus \cY = (X \oplus Y, \phi \oplus \psi , F \oplus G)$.
All degenerate Kasparov bimodules are homotopic to the trivial bimodule $0=(0,0,0)$ and define the zero element in $KK(A, B)$.
Write $X=X_0 \hoplus X_1$ and set $-X:= X_1 \hoplus X_0$.
Let $U : X \to -X$ be the canonical isomorphism and set $\phi_-:=U \phi (\cdot) U^*$.
The inverse of $[\cX]$ is represented by $-\cX:=(-X, \phi_-, -F)$.
If $\phi : A \to B$ is a $*$-homomorphism, then $(B \hoplus 0, \phi \hoplus 0, 0 )$ defines a Kasparov $A$-$B$ bimodule.
We still denote by $\phi$ the element in $KK (A, B)$ corresponding to $(B \hoplus 0, \phi \hoplus 0, 0 )$. 

Let $A$, $B$ and $C$ be $\rC^*$-algebras.
For ${\bf x} \in  KK(A,B)$ and ${\bf y} \in KK (B, C)$, we denote by ${\bf x} \otimes_B {\bf y}$ the Kasparov product of ${\bf x}$ and ${\bf y}$.

Let $f : A \to B$ be a $*$-homomorphism.
Then, $f^* : KK (B, C) \to KK (A, C)$ is the group homomorphism given by $\lE(B,C) \ni (X, \phi, F) \mapsto f^*(X, \phi, F) :=(X, \phi \circ f, F) \in \lE (A, C)$.
Similarly, the $f_* : KK (C,A) \to KK (C, B)$ is defined by $\lE (C, A) \ni (Y, \psi, G) \mapsto f_* (Y, \psi, G):= (Y\otimes_f B, \psi \otimes 1, G \otimes 1) \in \lE (C, B)$.
Let ${\bf x} \in KK(B, C)$ be an element represented by $(X, \phi, F) \in \lE (B, C)$.
If $f : A \to B$ and $g : C \to D$ are $*$-homomorphisms, then we have $ f \otimes_B {\bf x} = f^* ({\bf x}) =[ X, \phi \circ f , F]$ and ${\bf x} \otimes_C g = g_* ({\bf x}) = [ X\otimes_g D,\phi \otimes 1, F \otimes 1]$.

\begin{definition}
An element ${\bf x} \in KK(A, B)$ is said to be a $KK$-{\it subequivalence} if there exists ${\bf y} \in KK( B, A)$ such that $\id_A ={\bf x} \otimes_B {\bf y}$.
When $\id_A ={\bf x}\otimes_B {\bf y}$ and $\id_B = {\bf y }\otimes_B {\bf x}$ hold, we say that  ${\bf x}$ is a $KK$-{\it equivalence}.
In this case, the $A$ and $B$ are said to be $KK$-{\it equivalent}.
\end{definition}
If ${\bf x} \in KK(A, B)$ is a $KK$-equivalence, then for any $C$ the mappings ${\bf x} \otimes_B (\cdot ) :KK(B, C) \to KK(A, C)$ and $( \cdot ) \otimes_A {\bf x} : KK (C, A) \to KK (C, B)$ are isomorphisms.
In particular, we have $K_* (A) \cong K_* (B)$ if $A$ and $B$ are $\sigma$-unital.
%
%
\subsection{Strong relative nuclearity implies relative $K$-nuclearity}\label{ss-rel-K-K}
The main result of this subsection is Theorem \ref{thm-K-nuclear}.
This technical theorem plays an important in \S\S \ref{ss-KK-CD}.

For a Hilbert $A$-module $X$ we denote by $J_X$ the degree 1 unitary $\left[ \begin{smallmatrix} 0&1\\ 1& 0 \end{smallmatrix} \right] \in \lL_A (X\hoplus X)$.
Since $I(X^\infty)=(X^\infty) \otimes C[0,1]=\ell^2(\lN) \otimes X \otimes C[0,1]=(IX)^\infty$, we will write $IX^\infty$.
For $\rC^*$-correspondence $(X, \pi_X)$ we set $(IX, \pi_{IX}):= (IX, \pi_X \otimes 1_{C[0,1]})$.
\begin{definition}[{\cite[Definition 3.4]{Germain-fields}}]
Let $1_A \in B \subset A$ be a unital inclusion of $\rC^*$-algebras with a conditional expectation $E : A \to B$.
We say that $A$ is {\it $K$-nuclear relative to} $(B, E)$ if the $\rC^*$-correspondence $(X, \pi_X):= (L^2 (A, E) \otimes_B A, \pi_E \otimes 1_A)$ satisfies the following:
\begin{itemize}
\item[(i)] there exist unital $*$-homomorphisms $\pi^+$ and $\pi^-$ from $A$ into $\lL_{IA} (IX^\infty)$ such that $\cX=(IX^\infty \hoplus IX^\infty, \pi^+ \hoplus \pi^-, J_{IX^\infty} ) \in \lE (A, IA)$, i.e., $\pi^+ (a)- \pi^- (a) \in \lK_{IA} (IX^\infty)$ for all $a \in A$;
\item[(ii)] $\pi^+ (b) = \pi^- (b) =\pi_{IX}^\infty (b)$ for all $b \in B$;
\item[(iii)] the evaluation of $\cX$ at $t=1$ is degenerate, i.e., we have $\pi^+_1 = \pi^-_1$;
\item[(iv)] there exists a unitary $U \in \lL_A (A \oplus X^\infty, X^\infty)$ such that
$\pi^+_0  = U(\lambda_A \oplus \pi_X^\infty )U^*$ , $\pi^-_0= \pi_X^\infty$, and
and $U(1_A \oplus 0)= \xi \otimes 1_A$ for some $\xi \in L^2(A,E)^\infty$.
\end{itemize}
\end{definition}

\begin{proposition}
If $A$ is unital separable $\rC^*$-algebra, $B\subset A$ is a $\rC^*$-subalgebra,
and $X \in \Corr (A)$ has the $B$-CCPAP, then we can chose a sequence $\psi_n \in \cF_{B'\cap X}$ such that $\psi_n(1_A) =1_A$ for $n\in \lN$ and $\lim_{n \to \infty} \| a - \psi_n (a ) \| =0 $ for $a\in A$. 
\end{proposition}
\begin{proof}
Take a countable dense set $\{ a_n \}_{n=1}^\infty$ in $A$ with $a_1 =1_A$. 
Take $\psi_{n} \in \cF_{B' \cap X}$ in such a way that $\| a_k - \psi_n (a_k) \| < 2^{-n}$ for $ 1\leq k \leq n $ and $\psi_n (1_A) \leq 1_A$.
Then $\{ \psi_n \}_n $ enjoys the second assertion.
We next replace $\psi_n$ by a u.c.p.\ map.
Since $\psi_n (1_A) b = \psi_n (b) =b \psi_n (1_A)$ for $b \in B$, it follows that $\psi_n (1_A) \in B' \cap A$.
Define the u.c.p.\ map $\varphi_n \in \cF_{X}$ by $\varphi_n (a) := \psi_n (1_A)^{-1/2} \psi_n (a) \psi_n (1_A)^{-1/2}$.
Since $\psi_n (1_A)^{-1/2}$ is in $B' \cap A$ and converges to $1_A$ in norm,
$\{ \varphi_n \}_{n=1}^\infty$ is the desired one.
\end{proof}

For $X, Y \in \Corr (A, C)$ and $B \subset A$ we define the set ${}_B \lL_C (X, Y):= \{ x \in \lL_C(X, Y) \mid x\pi_X(b) = \pi_Y (b) x \text{ for } b \in B \}$ and ${}_B\lK_C (X, Y):=\lK_C (X, Y) \cap {}_B\lL_C (X, Y)$.
\begin{theorem}\label{thm-strongvoic}
Let $A$ be a unital separable $\rC^*$-algebra, $B$ be a $\rC^*$-subalgebra of $A$, and $(X,\pi_X) \in \Corr (A)$ be unital and have the $B$-CCPAP.
Then, there exists a sequence of isometries $V_n \in {}_B\lL_A (H_A, X^\infty)$ such that
\begin{itemize}
\item[$(1)$] $\lambda_A^\infty (a) - V_n^* \pi_X^\infty (a) V_n \in \lK_A (H_A)$ for $a\in A$,
\item[$(2)$] $\lim_{n \to \infty}\|  \lambda_A^\infty (a) - V_n^* \pi_X^\infty (a) V_n \| =0$ for $a\in A$.
\end{itemize}
\end{theorem}

\begin{proof}
Since $A$ is separable, it suffices to show that for any finite subset $\fF \subset A$ and $\varepsilon>0$ there exists an isometry $V \in {}_B\lL_A (H_A, X^\infty)$ satisfying (1) and  $\|  \lambda_A^\infty (a) - V^* \pi_X^\infty (a) V \| <\varepsilon$ for $a\in \fF$.
Take an increasing sequence of finite subsets $\fF=\fF_1 \subset \fF_2 \subset \cdots$ of $A$ such that $A = \overline{\bigcup_{n=1}^\infty \fF_n}^{\| \cdot \|}$.
Write $H_A=\bigoplus_{n=1}^\infty A^{(n)}$ with $A^{(n)}:=A$.
Since $X$ has the $B$-CCPAP, for each $n \in \lN$ there exists $\psi_n \in \cF_{B'\cap X}$ such that $\psi_n(1_A) = 1_A$ and $\| a - \psi_n (a) \| < 2^{-n}\varepsilon $ for $a \in \fF_n$.
Write $\psi_n= \sum_{r=1}^{d(n)} \Omega_{\xi_n^{(r)}}$ with $\xi_n^{(r)} \in B'\cap X$.
Define $W_n \in \lL_A (A^{(n)}, X^{d(n)} )$ by $W_n 1_A =(\xi_n^{(1)}, \dots, \xi_n^{(d(n))})$.
We then have
$
\psi_n (a) = \sum_{r=1}^{d(n)} \i< \xi_n^{(r)}, \pi_X (a) \xi_n^{(r)} >_X = \i< W_n 1_A, \pi_X^{d(n)}(a) W_n1_A>_{A^{(n)}} = W_n^* \pi_X^{d(n)}(a) W_n
$
for $a\in A$.
Hence, $W_n$ is an isometry in ${}_B\lL_A (A^{(n)}, X^{d(n)})$.
Define $V \in \lL_A (H_A, X^\infty)$ by $V:=\bigoplus_{n=1}^\infty W_n : H_A \to \bigoplus_{n=1}^\infty X^{d(n)} \cong X^\infty$.
Then, it follows that $V$ is an isometry in ${}_B \lL_A (H_A, X^\infty)$.
Moreover, for any $a\in \bigcup_{n=1}^\infty \fF_n$, we have $\sum_{n=1}^\infty \| a - \psi_n(a) \| < \infty$, and hence $V$ satisfies (1).
By the construction of $V$ we also have $\|  \lambda_A^\infty (a) - V^* \pi_X^\infty (a) V \| <\varepsilon$ for $a\in \fF$.
\end{proof}

\begin{corollary}\label{cor-cvoic}
Let $A$ be a unital separable $\rC^*$-algebra, $B$ be a $\rC^*$-subalgebra of $A$, and $(X,\pi_X) \in \Corr (A)$ be unital and have the $B$-CCPAP.
Then, there exists a sequence of unitaries $U_n \in {}_B\lL_A (H_A \oplus X^\infty, X^\infty)$ such that
\begin{itemize}
\item[$(1)$] $\pi_X^\infty (a) - U_n( \lambda_A^\infty(a) \oplus \pi_X^\infty (a) ) U^*_n \in \lK_A (X^\infty)$ for $a\in A$ and $n\in \lN$,
\item[$(2)$] $\lim_{n \to \infty}\|  \pi_X^\infty (a) - U_n( \lambda_A^\infty(a) \oplus \pi_X^\infty (a)) U^*_n \| =0$ for $a\in A$.
\end{itemize}
\end{corollary}

\begin{proof}
Let $\{ V_n \}_n \subset \lL_A (H_A, X^\infty)$ be a sequence of isometries in the preceding theorem and put $P_n:=1 - V_nV_n^*$.
Note that $P_n \in \pi_X^\infty(B)'$.
Let $S_1,S_2 \in \lB (\ell^2(\lN))$ be the isometries defined by $S_1 \delta_n= \delta_{2n-1}$ and $S_2 \delta_n= \delta_{2x}$ for $n \geq 1$.
Then, the operator $T:=(S_1 \oplus S_2 ) \otimes 1_A : H_A \oplus H_A \to H_A$ is a unitary.
Define unitaries $W_n:=V_n \oplus P_n : H_A \oplus P_n X^\infty \to X^\infty$ and $U_n:=W_n (T\oplus 1_{X^\infty})(1_{H_A} \oplus W_n^*) : H_A\oplus X^\infty \to X^\infty$.
As in the proof of Theorem \ref{thm-absorbing} the sequence $\{ U_n\}_n$ satisfies (1) and (2).
Since each $U_n$ is the product of three unitaties intertwine the left actions of $B$, we have $U_n \in {}_B\lL_A (H_A \oplus X^\infty, X^\infty)$.
\end{proof}

For a Hilbert C$^*$-module $X$ and $\xi \in X$, we set $\xi^{(n)}:=\xi \otimes \delta_n \in X \otimes \ell^2 (\lN) =X^\infty$.
\begin{corollary}
Let $B \subset A$ and $X$ be as in the preceding corollary.
For any $\xi_0 \in  B' \cap X$ with $\i< \xi_0, \xi_0> =1_A$ there exists a unitary $U \in {}_B\lL_A (H_A \oplus X^\infty, X^\infty)$ such that
\begin{itemize}
\item[$(1)$] $\pi_X^\infty (a) - U( \lambda_A^\infty(a) \oplus \pi_X^\infty (a) ) U^* \in \lK_A (X^\infty)$ for $a\in A$,
\item[$(2')$] $U(1_A^{(1)}\oplus 0) =\xi_0^{(1)}$.
\end{itemize}
\end{corollary}

\begin{proof}
We use the notation in the proof of Theorem \ref{thm-strongvoic}.
Replace $W_1 : A_1 \to X^{d(1)}$ in the proof of Theorem \ref{thm-strongvoic} by $\widetilde{W}_1 : A_1 \ni 1_A \mapsto \xi_0 \in X$.
Then $\widetilde{V}=\widetilde{W}_1 \oplus \bigoplus_{n=2}^\infty W_n \in {}_B\lL_A ( H_A, X^\infty)$ is also an isometry enjoys (1) in Theorem \ref{thm-strongvoic} and that $\widetilde{V} 1_A^{(1)}  =\xi_0^{(1)}$.
Put $\widetilde{P}:=1 - \widetilde{V}\widetilde{V}^*$ and $\widetilde{W}:=\widetilde{V}\oplus \widetilde{P}$.
We note that the $T$ above satisfies that $T( 1_A^{(1)} \oplus 0)=1_A^{(1)}$.
Then $\widetilde{U}=\widetilde{W}(T \oplus 1_{X^\infty} )(1_{H_A}\oplus \widetilde{W}^*) \in {}_B\lL_A ( H_A \oplus X^\infty, X^\infty)$ is the desired unitary.
We only check $(2')$:
\begin{align*}
\widetilde{U}(1_A^{(1)}\oplus 0)
&=\widetilde{W}(T \oplus 1_{X^\infty} )(1_{H_A}\oplus \widetilde{W}^*)(1_A^{(1)}\oplus 0)
=\widetilde{W}(T \oplus 1_{X^\infty} )((1_A^{(1)}\oplus 0_{H_A}) \oplus 0_{X^\infty} ) \\
&=\widetilde{W}(1_A^{(1)}\oplus 0_{X^\infty}) 
=\widetilde{V}1_A^{(1)}
=\xi_0^{(1)}.
\end{align*}
\end{proof}

%
%
For a $\rC^*$-algebra $A$ we set $CA:=A\otimes C_0[0,1)$, where $C_0[0,1)=\{ f\in C[0,1] \mid f(1)=0 \}$. Note that $C_0[0,1)$ naturally forms a Hilbert $C[0,1]$-module.

\begin{theorem}\label{thm-K-nuclear}
Let $B \subset A$ be a unital inclusion of separable $\rC^*$-algebra,
$(X, \pi_X) \in \Corr(A)$ be countably generated, unital, and have the $B$-CCPAP,
$\xi \in B' \cap X$ be a fixed vector with $\i< \xi, \xi> =1_A$. 
Set $(IX, \pi_{IX}):=(X\otimes C[0,1], \pi_X \otimes 1_{C[0,1]}) \in \Corr (A, IA)$.
Then, there exists a unitary $U \in \lL_{IA} (CA \oplus  IX^\infty  , IX^\infty )$ such that
\begin{itemize}
\item[$(1)$] the triple $\cX= (IX^\infty  \hoplus IX^\infty, U ( \lambda_A \otimes 1_{C_0[0,1)} \oplus \pi_{IX}^\infty )U^* \hoplus \pi_{IX}^\infty, J_{IX^\infty })$ forms a Kasparov $A$-$IA$ bimodule.
\item[$(2)$] the unitary $U$ is in ${}_B\lL_{IA} (CA \oplus I X^\infty  , IX^\infty )$, i.e., $U_t ( \lambda_A(b) \oplus \pi^\infty_X(b))U_t^*=\pi_X^\infty(b)$ for all $b\in B$ and $t\in [0,1)$.
\item[$(3)$] the evaluation $U_1$ of $U$ at $1$ equals $1_{X^\infty}$, and hence the evaluation of $\cX$ at 1 is degenerate.
\item[$(4)$] the evaluations $\{ U_t \}_{0 \leq t \leq 1}$ satisfies that
$$
\begin{cases}
U_t (\cos (\pi t) 1_A \oplus \sin (\pi t) \xi^{(1)}) &= \xi^{(1)}  \quad\text{for} \quad 0 \leq t \leq 1/2 \\
U_t (0 \oplus \xi^{(1)})&= \xi^{(1)}   \quad\text{for}\quad 1/2 \leq t < 1
\end{cases}
$$
\end{itemize}
\end{theorem}

\begin{proof}
By the $B$-CCPAP of $X$ there exists a unitary $V \in {}_B\lL_A (H_A \oplus X^\infty, X^\infty)$ such that $V (\lambda_A^\infty(a) \oplus \pi_X^\infty (a))V^*- \pi_X^\infty \in \lK_A (X^\infty)$ for $a\in A$ and $V (1_A^{(1)} \oplus 0)=\xi^{(1)}$.
We set $H^\circ :=\ell^2(\lN) \ominus \lC \delta_1$ and $H^\circ_{C[0,1]}:=H^\circ \otimes C[0,1]$.
By Kasparov's stabilization theorem, there exists a unitary $T \in \lL_{C[0,1]} (C_0[0,1) \oplus H^\circ_{C[0,1]}, H^\circ_{C[0,1]})$.
Set $W:=(V\otimes 1_{C[0,1]}) ( T \otimes 1_A \oplus  1_{IA} \oplus 1_{IX^\infty} ) (1_{CA} \oplus  V^* \otimes 1_{C[0,1]}):$
\begin{align*}
CA \oplus IX^\infty
& \longrightarrow CA \oplus H_{IA} \oplus IX^\infty
= (C[0,1 ) \oplus H^\circ_{C[0,1]} ) \otimes A \oplus IA \oplus IX^\infty \\
&\longrightarrow (H_{C[0,1]}^\circ) \otimes A \oplus IA \oplus IX^\infty =H_{IA} \oplus IX^\infty \longrightarrow IX^\infty.
\end{align*}
By construction, it follows that $W \in {}_B\lL_{IA} (CA \oplus IX^\infty  ,I X^\infty )$ and
\begin{equation}
W( \lambda_A(a) \otimes 1_{C_0[0,1)} \oplus \pi_{IX}^\infty(a) )W^* - \pi_{IX}^\infty (a) \in \lK_{IA}( I X^\infty ) \text{ for } a\in A, \label{eq-path}
\end{equation}
and $W(0 \oplus \xi^{(1)}\otimes 1_{C[0,1]}) =\xi^{(1)} \otimes 1_{C[0,1]}$.
Let $W_1 \in {}_B\lL_A (X^\infty)$ be the evaluation of $W$ at 1 and put $U:=(W_1^* \otimes 1_{C[0,1]}) W$.
Then $U$ also enjoys $U(0 \oplus \xi^{(1)}\otimes 1_{C[0,1]})=(W_1^* \xi^{(1)}) \otimes 1_{C[0,1]} = \xi^{(1)} \otimes 1_{C[0,1]}$ and (\ref{eq-path}) in which replaced $W$ by $U$.
Setting $\eta:=U_{1/2}(1_A \oplus 0)$ we have $\eta \in B' \cap X^\infty$, $\i< \eta, \eta > =1_A$, and $\i< \xi^{(1)} , \eta >=0$.
Hence we have $X^\infty = \xi^{(1)}  A \oplus \eta  A \oplus Y$ for a submodule $Y \subset X^\infty$.
Note that $\pi_X^\infty (B)$ commute with projections $\theta_{\xi^{(1)},\xi^{(1)}}$ and $\theta_{\eta, \eta}$.
Define $U_t' \in \lL_A (\eta  A \oplus \xi^{(1)} A)$ for $0 \leq t \leq 1/2$ by
$U_t' \eta = \sin (\pi t) \eta  \oplus  \cos (\pi t) \xi^{(1)}$ and $U_t' (\xi^{(1)})= ( -\cos (\pi t) \eta ) \oplus \sin (\pi t ) \xi^{(1)}$.
($U'_t$ is the unitary given by the matrix $\left[ \begin{smallmatrix} \sin(\pi t) & - \cos (\pi t) \\ \cos (\pi t) & \sin (\pi t) \end{smallmatrix} \right] \in \lM_2 (A) \cong \lL_A (\eta  A \oplus \xi^{(1)}  A )$.)
Now $U'_t  \oplus 1_Y$ is a norm continuous path of unitaties in $\lC 1 + {}_B\lK_A (X^\infty)$ and $U'_{1/2} \oplus 1_Y =1$.
Let $\widetilde{U} \in \lL_{IA}(CA \oplus IX^\infty, IX^\infty )$ be the unitary defined by the path $\{ (U'_t \oplus 1_Y) U_{1/2} \}_{ 0\leq t \leq 1/2} \cup \{ U_t \}_{1/2 \leq t \leq 1}$.
By Lemma \ref{lem-path}, $\widetilde{U}$ is well-defined.
Moreover, by construction we get $\widetilde{U} \in {}_B\lL_{IA}(CA \oplus IX^\infty, IX^\infty )$. 
To see (1), it suffices to show that the path $\{ \widetilde{U}_t ( \lambda_A(a) \oplus \pi^\infty_X (a) )\widetilde{U}_t^* -  \pi_X^\infty(a) \}_{ 0\leq t \leq 1}$ is a norm continuous path in $\lK_A (X^\infty)$.
This follows from Lemma \ref{lem-path} and definition of $\widetilde{U}$ again.
Thus, $\widetilde{U}$ is the desired unitary.
\end{proof}

As a corollary, we get a relative analogue of  `nuclearity $\Rightarrow$ $K$-nuclearity' for {\it strongly} nuclear inclusions.
\begin{corollary}\label{cor-K-nuclear}
Let $B \subset A$ be unital inclusions of separable $\rC^*$-algebras with a conditional expectation $E :A \to B$.
If $(A, B, E)$ is strongly nuclear, then $A$ is $K$-nuclear relative to $(B,E)$.
\end{corollary}
\begin{proof}
Applying the previous theorem to $(X, \pi_X):=(L^2(A,E) \otimes_B A, \pi_E \otimes 1_A)$ and $\xi_E \otimes 1_A$ we get a unitary $U : CA \oplus IX^\infty \to IX^\infty$ satisfying the conditions in the theorem.
The $*$-homomorphisms $\pi^+:=U ( \lambda_A \otimes 1_{C_0[0,1)} \oplus \pi_{IX}^\infty )U^*$ and $\pi^-:= \pi_{IX}^\infty$ are the desired ones.
\end{proof}

%
%
\section{$KK$-equivalences of amalgamated free products}\label{sec-KK}
In this section, we prove Theorem \ref{thm-C} and Theorem \ref{thm-D}.

\subsection{Amalgamated free products of $\rC^*$-algebras}\label{ss-KK-ama}
Let $\cI$ be a set and $B$ be a unital $\rC^*$-algebra.
Let $\{(X_i, \pi_{X_i}) \}_{i \in \cI} \subset \Corr (B)$ be a family of unital $\rC^*$-correspondences over $B$ with $B$-central normal vectors $\xi_i$, i.e., one has $\pi_{X_i} (b) \xi_i = \xi_i b$ for $b \in V$ and $\i< \xi_i, \xi_i > =1_B$.
We define the index set $\cI_p$ for $p \in \lN$ by
$
\cI_p= \{ i: \{1, \dots, p \} \to \cI \mid i(k) \neq i(k+1) \text{ for }  1\leq k \leq p-1 \}. 
$
The free product of $\{(X_i, \xi_i ) \}_{i \in \cI}$ is the Hilbert $B$-module $(X, \xi_0)$ given by
$$
X=B \oplus \bigoplus_{p \geq 1} \bigoplus_{i \in \cI_p} X_{i(1)}^\circ \otimes_B \dots  \otimes_B X_{i(p)}^\circ,
$$
where $X_i^\circ = X_i \ominus \xi_i  B$.
We will denote by $\xi_0$ the unit of $B$ in the first direct summand of $X$
and write $(X, \xi_0 )= \bigstar_{i \in \cI} (X_i, \xi_i)$.
We also define complemented submodules $X(\lambda, j)$ and $X(\rho, j)$ of $X$ for $j \in \cI$ by
\begin{align*}
X(\lambda, j) = \xi_0  B \oplus  \bigoplus_{p \geq 1} \underset{i(1) \neq j}{\bigoplus_{i \in \cI_p}} X_{i(1)}^\circ \otimes_B \dots  \otimes_B X_{i(p)}^\circ,\\
X(\rho, j) = \xi_0  B \oplus  \bigoplus_{p \geq 1} \underset{i(p) \neq j}{\bigoplus_{i \in \cI_p}} X_{i(1)}^\circ \otimes_B \dots  \otimes_B X_{i(p)}^\circ,
\end{align*}
and unitaries $v_i \in \lL_B (X, X_i \otimes_B X(\lambda, i))$ and $w_i \in \lL_B (X, X(\rho, i) \otimes_B X_i)$ by
\begin{align*}
v_i \xi_0 &= \xi_i \otimes \xi_0, \\
v_i (\eta_1 \otimes \cdots \otimes \eta_p)
&= \begin{cases}
\eta_1 \otimes \xi_0 & \text{ for } \eta_1 \in X_i^\circ, p=1\\
\eta_1 \otimes (\eta_2 \otimes \cdots \otimes \eta_p)& \text{ for } \eta_1 \in X_i^\circ, p \geq 2\\
\xi_i \otimes (\eta_1 \otimes \cdots \otimes \eta_p) & \text{ for } \eta_1 \notin X_i^\circ
\end{cases}\\
w_i \xi_0 &= \xi_0 \otimes \xi_i \\
w_i (\eta_1 \otimes \cdots  \otimes \eta_p)
&= \begin{cases}
\xi_0 \otimes \eta_1& \text{ for } \eta_p \in X_i^\circ, p=1\\
( \eta_1 \otimes \cdots \otimes \eta_{p-1}) \otimes \eta_p  & \text{ for } \eta_p \in X_i^\circ, p \geq 2\\
(\eta_1 \otimes \cdots \otimes \eta_p) \otimes  \xi_i & \text{ for } \eta_1 \notin X_i^\circ .
\end{cases}
\end{align*}
We can now define $*$-homomorphisms $\lambda_i : \lL_B (X_i ) \to \lL_B(X)$ and $\rho_i : {}_B\lL_B (X_i) \to \lL_B (X)$ by
\begin{equation}
\lambda_i(x)=v_i^*( x\otimes 1_{X(\lambda,i)} )v_i,
\quad \rho_i(y) =w_i^*  (1_{X(\rho,i)} \otimes y) w_i. \label{eq-left-right}
\end{equation}
Hence, $(X, \lambda_i \circ \pi_{X_i})$ forms a C$^*$-correspondence over $B$.
Further assume that we have a Hilbert $B$-module $Y$ and $Z \in \Corr (B, C)$ for a C$^*$-algebra $C$.
By using unitaries $1_Y \otimes v_i \in \lL_B (Y \otimes_B X, Y \otimes_B X_i \otimes_B X( \lambda, i ) )$ and $w_i \otimes 1_Z \in \lL_C (X \otimes_B Z, X (\rho, i) \otimes_B X_i \otimes_B Z) $ we can define $*$-homomorphisms $\lambda_i^Y : \lL_B (Y\otimes_B X_i) \to \lL_B (Y \otimes_B X); x \mapsto (1_Y \otimes v_i)^* (x \otimes 1_{X(\lambda, i)} ) (1_Y \otimes v_i)$ and $\rho_i^Z : {}_B \lL_C (X_i \otimes_B Z) \to \lL_C (X \otimes_B Z); y \mapsto (w_i \otimes 1_Z)^* ( 1_{X(\rho, i)} \otimes y ) (w_i \otimes 1_Z)$.

A direct computation will prove the following proposition.
\begin{proposition}\label{prop-left-right}
Under the notation above, it follows that
$$
\left[ \lambda_i (x) \otimes 1_Z, \rho_j^{Z} (y) \right] = \delta_{ij}\rho_i^Z ( [ x \otimes 1_Z, y] )(P_{X_i} \otimes 1_Z) 
$$
for $x \in \lL_B (X_i)$, $y \in \lL_C (X_j \otimes_B Z)$, and $i,j \in \cI$, where $P_{X_i } \in \lL_B (X)$ is the orthogonal projection onto $X_i = \xi_0 B \oplus X_i^\circ \subset X$. 
\end{proposition}

\begin{definition}[\cite{Voiculescu-free}]
Let $\{ (A_i, E_i ) \}_{i \in \cI}$ be a family of unital $\rC^*$-algebras with nondegenerate conditional expectations $E_i$ from $A_i$ onto a common $\rC^*$-subalgebra $B$ containing $1_{A_i}$.
Let $(L^2(A_i, E_i ), \pi_{i}, \xi_i ) \in \Corr (A_i, B)$ be the GNS representation for $(A_i, E_i)$.
Let $(X, \xi_0)$ be the free product Hilbert module $\bigstar_{i \in \cI} (L^2 (A_i, E_i ), \xi_i )$.
The reduced amalgamated free product of $\{ (A_i, E_i ) \}_{i \in \cI}$ over $B$ is given by the pair $( \bigstar_{B, i \in \cI} (A_i , E_i ), E)$, where $\bigstar_{B, i \in \cI} (A_i , E_i )$ is the $\rC^*$-subalgebra of $\lL_B (X)$ generated by $\lambda_i \circ \pi_{E_i} (A_i), i\in \cI$ and $E $ is the conditional expectation from $A$ onto $B$ given by $\Omega_{\xi_0} \in \cF_{X}$.
\end{definition}
We note that $L^2(A, E) =\xi_E  B \oplus L^2(A,E)^\circ$, where $L^2 (A,E)^\circ$ the orthogonal complement of $\xi_E  B$, and $\pi_E (B)$ reduces these subspaces.

\begin{definition}
Let $\{ A_i \}_{i \in \cI}$ be a family of $\rC^*$-algebras containing a common $\rC^*$-subalgebra $B$.
Then, the {\it full amalgamated free product} of $A_i , i \in \cI$ over $B$ is the $\rC^*$-algebra $\bigstar_{B, i \in \cI} A_i$, equipped with injective $*$-homomorphisms $\iota_i : A_i \to \bigstar_{B, i \in \cI} A_i$ such that $\iota_i (b) =\iota_j (b)$ for $b \in B$ and $i, j \in \cI$ and satisfying the following universal property{\rm :} for any $\rC^*$-algebra $C$ and $*$-homomorphisms $\pi_i : A_i \to C$ such that $\pi_i (b) = \pi_j(b)$ for $b \in B$ and $i,j \in \cI$, there exists a unique $*$-homomorphism $\bigstar_{ i \in \cI} \pi_i : \bigstar_{B, i \in \cI} A_i \to C$ such that $(\bigstar_{ i \in \cI} \pi_i  )\circ \iota_i = \pi_i$ for $i \in \cI$.
\end{definition}
%
%
\subsection{Theorem C and Theorem D}\label{ss-KK-CD}
In this subsection we prove Theorem \ref{thm-KK} below, which contains Theorem \ref{thm-C} and Theorem \ref{thm-D} (see Proposition \ref{prop-fin-dim}).
(We refer to \S\S \ref{ss-rel-rel} for the definitions of strong relative nuclearity via C$^*$-correspondences and central vectors.)
\begin{theorem}\label{thm-KK}
Let $\{(A_i, B, E_i) \}_{i \in \cI}$ be an at most countable family of unital inclusions of separable $\rC^*$-algebras $B \subset A_i$ with conditional expectations $E_i : A_i \to B$.
If each triple $(A_i, B, E_i)$ is strongly nuclear via $\rC^*$-correspondence $(Z_i, \pi_{Z_i})$ over $B$ such that $Z_i$ is countably generated and admits a $B$-central vector $\zeta_i \in Z_i$ with $\i< \zeta_i, \zeta_i> =1_B$,
then the canonical surjection from
the full amalgamated free product $\bigstar_{B, i \in \cI} A_i$ onto the reduced one $\bigstar_{B, i \in \cI} (A_i, E_i)$ is a $KK$-equivalence.
\end{theorem}
We first establish several technical lemmas based on \cite{Germain-duke}\cite{Germain-fields}.
In what follows, $\cI$ denotes a countable set and $\{ (A_i, B, E_i) \}_{i \in \cI}$ is a family of unital inclusions of separable $\rC^*$-algebras with nondegenerate conditional expectations $E_i : A_i \to B$.
Let $A:= \bigstar_{i, \in \cI, B} A_i$ and $A_\re:= \bigstar_{i \in \cI, B} (A_i, E_i)$ be the full and the reduced amalgamated free products, respectively and let $\pi_\re : A \to A_\re$ be the canonical surjection.

The next lemma is probably well-known, but we give its proof for the reader's convenience.
\begin{lemma}\label{lem-free}
If $(Z,\pi_Z)$ is an $A$-$B$ $\rC^*$-correspondence such that there exists a subspace $\Gamma \subset Z$ such that $\pi_Z (A)\Gamma$ is norm dense in $Z$ and satisfying the freeness condition:
for any $p \in \lN$, $i \in \cI_p$, and $a_k \in \ker E_{i(k)} \subset A_{i(k)}, 1 \leq k \leq p$,
$$\i< \eta, \pi_Z ( a_1 \cdots a_p) \xi >=0 \quad \text{for } \xi, \eta \in \Gamma,$$
then the $\pi_Z$ factors through $A_\re$. 
\end{lemma}
\begin{proof}
It suffices to show that $\ker \pi_\re \subset \ker \pi_Z$.
Fix $a \in \ker \pi_\re$ arbitrarily and take a sequence $\{ a_n \}_n \in *\text{-Alg}(A_i, i\in \cI ) \subset A$ such that $\lim_{n \to \infty} \| a - a_n \| =0$.
Since $\Gamma$ is a cyclic subspace for $\pi_Z (A)$ we only have to prove that $\i<\eta, \pi_Z (b^* a c) \xi >=0 $ for all $b,c \in *\text{-Alg}(\iota_i (A_i), i\in \cI)$ and $\xi,\eta \in \Gamma$.
It is known that $\pi_\re (b^*a_n c)$ is a sum of an element $b_n \in B$ and finitely many elements having the form of $x_{i(1)} \cdots x_{i(p)}$ for some $p \in \lN$, $i \in \cI_p$, and $x_{i(k)}\in A_{i(k)} \cap \ker E \subset A_\re$.
Since $\xi, \eta \in \Gamma$ we now have
$\| \i< \eta, \pi_Z (b^* a c ) \xi > \|  \leftarrow \| \i< \eta, \pi_Z(b^*a_n c) \xi > \| = \|\i< \eta, \pi_Z (\iota  (b_n)) \xi > \| \leq \| \xi \| \| \eta \|  \| b_n \| = \|\xi \| \| \eta \| \| E (\pi_\re (b a_n c^*))\| \to 0$,
which implies that $\pi_Z (a) =0$.
\end{proof}

Let $Z_i \in \Corr (B)$ be unital faithful and $\zeta_i \in B' \cap Z_i$ be a normal vector, i.e., $\i< \zeta_i, \zeta_i>=1_B$.
Set $(X_i, \xi_i):= ( L^2(A_i, E_i) \otimes_B Z_i, \xi_{E_i} \otimes \eta_i )$ and $(X, \xi_0):=\bigstar_{i \in \cI} (X_i, \xi_i)$.
By the universality of $A$ we obtain the $*$-homomorphism $\pi_X:=\bigstar_{i \in \cI} \lambda_i \circ ( \pi_{E_i} \otimes 1_{Z_I}) : A \to \lL_B (X)$, where $\lambda_i : \lL_B (X_i) \to \lL_B (X)$ is as in (\ref{eq-left-right}).
We also set $(Y_i, \pi_{Y_i}):= (X_i \otimes_B A_i, \pi_{X_i} \otimes 1_{A_i} ) \in \Corr (A_i)$ and $(Y, \pi_Y):= (X \otimes_ B A, \pi_X \otimes 1_A) \in \Corr (A)$. 
\begin{lemma}\label{lem-throughAr}
Under the notation above, there exists a $*$-homomorphism $\bar{\pi}_X : A_\re \to \lL_B (X)$ such that $\pi_X = \bar{\pi}_X \circ \pi_\re$.
\end{lemma}
\begin{proof}
Set $Z_i^\circ:= Z_i \ominus \zeta_i B$ and
define the closed submodule $\Lambda \subset X$ by
$$
\Lambda := \bigoplus_{ p \geq 2} \bigoplus_{j \in \cI_p} ( \xi_{E_{j(1)}} \otimes_B Z_{j(1)}^\circ) \otimes_B X_{j(2)}^\circ \otimes_B \cdots \otimes_B X_{j(p)}^\circ .
$$
We will show that the closed submodule $\Gamma := \xi_0 B \oplus \bigoplus_{i \in \cI}
( \xi_{E_i} \otimes_B Z_i^\circ)
\oplus \Lambda $ satisfies the condition in Lemma \ref{lem-free}.
Indeed, it is not hard to see that $\pi_X ( a_1 \cdots a_p ) \Gamma \perp \Gamma$ for all $p \geq 1 $, $i \in \cI_p$, and $a_k \in A^\circ_{i (k)}$.
Set $\fX_0:=\xi_0B$ and $\fX_p:=\bigoplus_{i \in \cI_p} X_{i(1)}^\circ \otimes_B \cdots \otimes_B X_{i(p)}^\circ$ for $p \geq 1$.
We also set $X_i^\bullet:=L^2(A_i, E_i)^\circ \otimes_B \zeta_i B =\overline{\pi_X (A_i^\circ) \xi_0}$ and $\fX_p^\bullet:=\bigoplus_{i \in \cI} X_{i(1)}^\bullet \otimes_B X_{i(2)}^\circ \otimes_B \cdots \otimes_B X_{i(p)}^\circ$ for $p \geq 1$.
Since $X=\bigoplus_{p \geq 0} \fX_p$, it suffices to show that $\ospan \pi_X(A)\Gamma$ contains $\fX_p$ for all $p \geq 0$.
The case that $p=0$ is trivial.
Suppose that $p \geq 1$ and $\fX_p \subset \ospan \pi_X(A) \Gamma$.
We observe that $\fX_{p+1}^\bullet \subset \overline{\pi_X(A) \fX_p}$ and $\fX_{p+1}\ominus \fX_{p+1}^\bullet = \bigoplus_{i \in \cI_{p+1}} ( L^2(A_{i(1)}, E_{i(1)})\otimes_B Z_{i(1)}^\circ)  \otimes_B  X_{i(2)}^\circ \otimes_B \cdots \otimes_B X_{i(p+1)}^\circ \subset \overline{ \pi_X(A)\Gamma}$, 
which implies that $\fX_{p+1} \subset \ospan {\pi_(X) \Gamma}$.
By induction and Lemma \ref{lem-free}, we are done.
 \end{proof}
 
\begin{lemma}[{cf.\ \cite[Proposition 4.2]{Germain-duke}}] \label{lem-throughAr2}
Under the notation above,
suppose that there exist $Z'_i \in \Corr (B)$ and unitary $W_i : X_i \otimes_B Z'_i \to X_i \otimes_B X_i$ satisfying that
\begin{itemize}
\item[(i)] $W_i ( \pi_{X_i} (b) \otimes 1_{Z'_i} ) W_i^* = \pi_{X_i} (b) \otimes 1_{X_i}$ for all $b \in B$;
\item[(ii)] $\xi_i \otimes \xi_i \in W_i (\xi_i B \otimes_B Z'_i)$.
\end{itemize}
Define $\phi_i : A_i \to \lL_B (X \otimes_B X)$ by $\phi_i (a) = \lambda_i^{X_i} (W_i (\pi_{X_i} (a) \otimes 1_{Z'_i})W_i^* ) \oplus \tau_i (a) \otimes 1_X$.
Then, the $*$-homomorphism $\phi:=\bigstar_{i \in \cI} \phi_i : A \to \lL_B(X \otimes_B X)$ factors through $A_\re$.
\end{lemma}
\begin{proof}
We set $V_i:=1_{X_i} \otimes v_i \in \lL_B (X_i \otimes_B X,  X_i \otimes_B X_i \otimes_B X(\lambda, i ) )$.
Let $\fX_p, X_i^\bullet, \fX_p^\bullet$ and $\Lambda$ be as in the proof of Lemma \ref{lem-throughAr}.
Define the subspace $\Gamma_i \subset X_i\otimes_B X$ by
$$
\Gamma_i:=
V_i^* ( W_i \otimes 1_{X (\lambda, i)} )  \left(  \xi_{E_i} \otimes_B Z_i  \otimes_B Z'_i \otimes_B X (\lambda, i) \right).
$$
By (ii) the $\xi_0 \otimes \xi_0$ belongs  to $\Gamma_i$.
Set $\Gamma_i^\circ := \Gamma_i \ominus \xi_0 \otimes_B \xi_0 B$.
We will show that the subspace
$
\Gamma:=
\xi_0 \otimes_B \xi_0 B \oplus \bigoplus_{i \in \cI} \Gamma_i^\circ
\oplus  \Lambda \otimes_B X
$
satisfies the condition in Lemma \ref{lem-free}.
We note that 
\begin{equation}
\Gamma \subset \left[ \fX_0 \oplus \fX_1 \oplus  \bigoplus_{p=2}^\infty ( \fX_p \ominus \fX_p^\bullet )  \right] \otimes_B X.\label{eq-Gamma}
\end{equation}
Fix $i \in \cI$ and $a \in A_i^\circ$ arbitrarily and show that $\phi (a) \Gamma \perp \Gamma$.
By definition, we have
$$
\phi(a) = V_i^* (W_i \otimes 1_{X(\lambda, i)} ) ( \pi_{E_i} (a) \otimes 1_{Z_i} \otimes 1_{Z'_i} \otimes 1_{X(\lambda, i)} )(W_i \otimes 1_{X(\lambda, i)} )^* V_i \oplus \tau_i (a) \otimes 1_X,
$$
which implies that $\phi (a) \Gamma_i \perp \Gamma$.
For any $j \in \cI$ with $i \neq j$
the fact that $\phi (a) \Gamma_j^\circ \subset \fX_2^\bullet$ and the (\ref{eq-Gamma}) implies $\phi (a) \Gamma_j \perp \Gamma$.
We also have $\phi (\Lambda \otimes_B X) \perp \Gamma$ by (\ref{eq-Gamma}).
Next, let $q \geq 2$ and $a_{k} \in A_{i(k)}^\circ, 1 \leq k \leq q$ be arbitrary.
Because $\phi (a_1 \cdots a_q) \Gamma \subset \bigoplus_{p \geq 2} \fX_p^\bullet$ we get $\phi (a) \Gamma \perp \Gamma$ by (\ref{eq-Gamma}).

Finally, we prove that $ \Gamma_0:=\ospan \phi (A) \Gamma$ contains $\fX_p \otimes_B X$ for all $p \geq 0$.
The case of $p=0$ is trivial.
When $p=1$,
it follows that $X_i \otimes_B X \subset \Gamma_0$ since 
\begin{align*}
 \phi (A_i) \Gamma_i
=V_i  \left [ W_i  \left (  \pi_{X_i}(A_i)  \xi_{E_i} \otimes Z_i \otimes_B Z'_i \right)  \otimes_B X( \lambda, i) \right],
\end{align*}
which is dense in $V_i^* [ W_i (X_i \otimes_B Z'_i) \otimes_B X( \lambda, i) ] = X_i \otimes_B X.$
Since $\ospan ( \phi (A) \Lambda \otimes_B X )$ contains $\bigoplus_{p \geq 2} ( \fX_p \ominus \fX_p^\bullet) \otimes_B X$,
we only have to show that $\fX_p^\bullet \otimes_B X \subset \Gamma_0$ for $p \geq 2$.
This is done by induction.
Suppose that $\fX_p^\bullet \otimes_B X \subset \Gamma_0$ with $p \geq 2$.
We then have $\fX_p \otimes_B X \subset \Gamma_0$.
Since the restriction of $\phi (a)$ on $\fX_p \otimes_B X$ equals $\pi_X (a) \otimes 1_X$ for $a\in A$.
Thus, $\fX_{p+1}^\bullet \otimes_B X \subset  ( \ospan \pi_X (A) \fX_p )\otimes_B X \subset \Gamma_0$.
Therefore, by induction we get $\Gamma_0= X\otimes_B X$.
\end{proof}

Suppose that for each $i \in \cI$ there exists a unitary $U_i : A_i \oplus Y_i \to Y_i$ satisfies that $U_i ( 1_{A_i} \oplus 0 ) = \xi_i \otimes 1_{A_i}$ and $U_i ( b \oplus \pi_{Y_i} (b) )U_i^* = \pi_{Y_i} (b)$ for all $b \in B$.
Note that $Y \cong Y_i \otimes_{A_i} A \oplus (X \ominus X_i) \otimes_B A$ for all $i \in \cI$.
We define $\psi_i : A_i \to \lL_{A_i} (Y_i)$ by $\psi_i (a) = U_i ( a \oplus \pi_{Y_i} (a) )U_i^*$ and set $\psi:=\bigstar_{i \in \cI} ( \psi_i \otimes 1_A \oplus \tau_i \otimes 1_A)$, where $\tau_i : A_i \to \lL_B (X \ominus X_i)$ is defined by $\tau (a):=\pi_X(a)|_{X\ominus X_i}$.

\begin{lemma}\label{lem-unitary}
Under the notation above, let $S_i \in \pi_{Y_i} (B)' \cap \lL_{A_i} (Y_i)$ be the isometry defined by $U_i |_{ 0 \oplus Y_i}$.
Then, the isometries $\rho_i^{A} (S_i  \otimes 1_A) \in \lL_A (Y)$ satisfy that $P_A +\sum_{i \in \cI} \rho_i^A (S_i \otimes 1_A)\rho_i^A (S_i \otimes 1_A)^* =1_Y$, where $P_A \in \lL_A (Y)$ is the projection onto $\xi_0B \otimes_B A$.
Moreover, the unitary $U := P_A \oplus \bigoplus_{i \in \cI} \rho_i^A (S_i \otimes 1_A) : A \oplus Y\otimes \ell^2 (\cI) \to Y$ satisfies that $\psi (a) = U (a \oplus \pi_Y (a) \otimes 1_{\ell^2(\cI)} ) U^*$ for all $a \in A$.
\end{lemma}
 \begin{proof}
We observe that the image of $\rho_i^A (S_i \otimes 1_A)$ is
$
\bigoplus_{p \geq 1} \bigoplus_{j \in \cI_p, j(p)=i} X_{j(1)}^\circ \otimes_B \cdots \otimes_B X_{j(p)}^\circ \otimes_B A.
$
Since these subspaces are mutually orthogonal and span $Y \ominus P_A Y$, we get the first and the second assertions.

To see the second assertion, it suffices to see that $\psi (a) = U (\lambda_A (a) \oplus \pi_Y (a)  \otimes 1_{\ell^2 (\cI )}) U^*$ for all $i \in \cI$ and $a \in A_i$.
Thus, we fix $i \in \cI$ and $a\in A_i$.
Recall that $\psi (a)=[ ( U_i(\lambda_{A_i} (a) \oplus \pi_{Y_i} (a) ) U_i^* )\otimes 1_A ] \oplus \tau_i (a) \otimes 1_A = [ ( U_i \otimes 1_A)(\lambda_A(a) \oplus \pi_{Y_i} (a) \otimes 1_A )( U_i \otimes 1_A)^*] \oplus \tau_i (a) \otimes 1_A.$
Identifying $Y$ with $Y \otimes \lC \delta_i$ we have
\begin{align*}
\psi (a) U|_{A\oplus Y_i \otimes_{A_i} A \otimes \lC \delta_i }
&=\psi (a) (U_i \otimes 1_A)|_{A\oplus Y_i \otimes_{A_i} A } \\
&= ( U_i \otimes 1_A)(\lambda_A(a) \oplus \pi_{Y_i} (a) \otimes 1_A)|_{A\oplus Y_i \otimes_{A_i} A } \\
&=U (\lambda_A (a) \oplus \pi_Y (a) \otimes  1_{\ell^2(\cI)} )|_{A\oplus Y_i \otimes_{A_i} A \otimes \lC \delta_i }.
\end{align*}
By Proposition \ref{prop-left-right}, we also have
\begin{align*}
\psi (a) U |_{( X\ominus X_i ) \otimes_ B A \otimes \lC \delta_i }
&=( \tau_i (a) \otimes 1_A ) \rho_i^A (S_i \otimes 1_A) |_{( X\ominus X_i ) \otimes_ B A} \\
&=( \lambda_i (\pi_{X_i} (a)) \otimes 1_A ) \rho_i^A (S_i \otimes 1_A) |_{( X\ominus X_i ) \otimes_ B A} \\
&=\rho_i^A(S_i \otimes 1_A)( \lambda_i (\pi_{X_i} (a)) \otimes 1_A ) |_{( X\ominus X_i ) \otimes_ B A} \\
&=U ( \lambda_A (a) \oplus \pi_Y (a) \otimes 1_{\ell^2(\cI)} )|_{( X\ominus X_i ) \otimes_ B A \otimes \lC \delta_i }.
\end{align*}
Similarly, for any $j \in \cI$ with $j \neq i$, identifying $Y \otimes \lC \delta_j$ with $Y$ we have
\begin{align*}
\psi (a) U |_{Y \otimes \lC \delta_j }
&=( \tau_i (a) \otimes 1_A ) \rho_j^A (S_j \otimes 1_A) \\
&=( \lambda_i (\pi_{X_i} (a)) \otimes 1_A ) \rho_j^A (S_j \otimes 1_A) \\
&=\rho_j^A(S_j \otimes 1_A)( \lambda_i (\pi_{X_i} (a)) \otimes 1_A )\\
&=U ( \lambda_A (a) \oplus \pi_Y (a) \otimes 1_{\ell^2(\cI)} )|_{Y\otimes \lC \delta_j }.
\end{align*}
Hence, we get $\psi (a) U  = U ( L(a)  \oplus \pi_Y (a) \otimes 1_{\ell^2(\cI)} )$ for $a \in A_i$.
\end{proof}

%
%
We are now ready to prove Theorem \ref{thm-KK}.
\begin{proof}[Proof of Theorem \ref{thm-KK}]
Let $A:=\bigstar_{B, i \in \cI} A_i$ and $A_\re:=\bigstar_{B, i \in \cI} (A_i, E_i)$ and $\pi_\re : A \to A_\re$ be the canonical surjection.
We show that there exists $\cY \in \lE (A_\re, A)$ such that $\pi_\re \otimes_{A_\re} [ \cY] + \id_A = 0$ and $[ \cY ] \otimes_A \pi_\re + \id_{A_\re}=0$.
To simplify the notation we omit the canonical inclusion maps $\iota_i : A_i \to A$ and $\iota : B \to A$.
Put $(X_i, \pi_{X_i}, \xi_i):=(L^2(A_i, E_i) \otimes_B Z_i^\infty, \pi_{E_i} \otimes 1_{Z_i^\infty}, \xi_{E_i} \otimes \eta_i  \otimes \delta_1)$ and
$(Y_i ,\pi_{Y_i} ):= (X_i \otimes_B A_i, \pi_{X_i} \otimes 1_{A_i}) \in \Corr (A_i)$.
Define $(X,\pi_X) \in \Corr (A, B)$ by $(X, \xi_0):=\bigstar_{i \in \cI} (X_i, \xi_i)$ and $\pi_X:=\bigstar_{ i\in \cI} (\lambda_i \circ\pi_{X_i})$, where $\lambda_i : \lL_B (X_i) \to \lL_B (X)$ is the canonical left action and set $(Y, \pi_Y):=(X \otimes_B A, \pi_X \otimes 1_A)$.
We identify $X_i$ with the canonical copy of $X_i$ in $X$ and $\xi_i$ with $\xi_0$.
We also use the notation that $A_i^\circ = \ker E_i$ and $X_i^\circ = X_i \ominus \xi_0 B$.
Since $Y_i =(L^2(A_i, E_i)\otimes_B \otimes Z_i \otimes_B A_i )^\infty$, 
by Theorem \ref{thm-K-nuclear}, there exist unitaries $U^{(i)} \in \lL_{CA_i} (CA_i \oplus IY_i, IY_i)$ satisfying the following:
\begin{itemize}
\item[(1)] the triple $\cX_i=(IY_i \hoplus IY_i, U^{(i)} ( \lambda_{A_i} \otimes 1_{C_0[0,1)} \oplus \pi_{IY_i} )U^{(i)*} \hoplus \pi_{IY_i}, J_{IY_i})$ forms a Kasparov $A_i$-$IA_i$ bimodule;
\item[(2)] $\{ U^{(i)}_t \}_{0 \leq t <1}$ satisfies that $U_t ( \lambda_A(b) \oplus \pi^\infty_X(b))U_t^*=\pi_X^\infty(b)$ for all $b\in B$ and $t\in [0,1)$;
\item[(3)] the evaluation $U^{(i)}_1$ of $U$ at $1$ equals $1_{Y_i}$;
\item[(4)] the evaluations $\{U_t^{(i)} \}_{0\leq t \leq 1}$ enjoy that
$U^{(i)}_t ( \cos (\pi t) 1_{A_i} \oplus \sin (\pi t) \xi_i \otimes 1_{A_i} ) = \xi_i \otimes 1_{A_i}$ for $0 \leq t \leq 1/2$ and $U_t (0 \oplus \xi_i \otimes 1_{A_i})= \xi_i\otimes 1_{A_i}$ for $1/2 \leq t < 1$.
\end{itemize}
Let $\tau_i : A_i \to \lL_B (X\ominus X_i)$ be the compression of $\pi_{X} \circ \iota_i$ on $X\ominus X_i$.
We note that the isomorphism
$IY=IY_i \otimes_{IA_i} IA \oplus  (X\ominus X_i) \otimes_B IA$
induces
$$
(IY, \pi_{IY})
\cong
\left( IY,
\us \bigstar_{i \in \cI} \left( \pi_{I Y_i} \otimes 1_{IA}  \oplus \tau_i \otimes 1_{IA} \right) \right).
$$
Set $\psi^{(i)}:=U^{(i)} ( L_{A_i} \otimes 1_{C_0[0,1)} \oplus \pi_{IY_i})U^{(i)*} : A_i \to \lL_{IA_i} (IY_i)$.
By (2) it follows that $\psi^{(i)}|_B = \pi_{Y_i} |_B$ for all $i \in \cI$.
Hence, by the universality of $A$ we have the $*$-homomorphism
$$
\psi=\us\bigstar_{i \in \cI} \left( \psi^{(i)} \otimes 1_{IA} \oplus \tau_i \otimes 1_{IA} \right) : A \to \lL_{IA} (IY).
$$

We claim that the triple $\cX:=(IY \hoplus IY, \psi \hoplus \pi_{IY}, J_{IY} )$ defines a Kasparov $A$-$IA$ bimodule.
It suffices to show that for any $a \in A$, we have $\psi (a) - \pi_{IY} (a) \in \lK_{IA} (IY)$.
We may assume that $a$ has the form of $a_1 \cdots a_p$,
where  $p \in \lN$ and $i \in \cI_p$ and $a_k \in A_{i(k)}$.
We then have
\begin{align*}
\psi (\prod_{k=1}^p a_k ) - \pi_{IY} (\prod_{k=l}^p a_l) 
=\sum_{k=1}^p ( \prod_{n=1}^{k-1} \psi (a_n) ) ( \psi (a_k) - \pi_{IY} (a_k) ) ( \prod_{m=k+1}^p \pi_{IY} (a_m) ).
\end{align*}
This is compact since $\psi (a_k) - \pi_{IY} (a_k)$ equals 0 on $(X\ominus X_k) \otimes_B IA$ and equals $( \psi^{(i(k))} (a_k) - \pi_{IY_k} (a_k) ) \otimes 1_{IA} \in \lK_{IA_k} (IY_k) \otimes \lK_{IA} (IA)$ on $IY_k \otimes_{IA_k} IA$.

By (4) each unitary $U_0^{(i)}: A \oplus Y_i \to Y_i$ satisfies that $U_0^{(i)} (1_A \oplus 0 ) = \xi_i \ \otimes 1_{A_i}$.
Set $S_i := U^{(i)}_0|_{0\oplus Y_i}$, which is an isometry in ${}_B\lL_{A_i} (Y_i)$.
Let $\psi_t$ be the evaluation of $\psi$ at $t \in [0,1]$.
We observe that $\psi_0 (a) = \bigstar_{i \in \cI} ( U^{(i)}_0 ( L_{A_i} \oplus \pi_{Y_i}  )U_0^{(i)*} \oplus \tau_i \otimes 1_A )$. 
By Lemma \ref{lem-unitary} the operator $U := P_A \oplus \bigoplus_{i \in \cI} \rho_i^A (S_i \otimes 1_A) : A \oplus Y\otimes \ell^2 (\cI) \to Y$ is a unitary satisfying that $\psi_0 (a) = U (\lambda_A (a) \oplus \pi_Y (a) \otimes 1_{\ell^2(\cI)} ) U^*$ for all $a \in A$.
Thus, we have
\begin{align*}
\cX_0 &= ( Y \hoplus Y, \psi_0 \hoplus \pi_Y , J_{Y} ) = ( Y \hoplus Y,U (\lambda_A \oplus \pi_Y \otimes 1_{\ell^2 (\cI )}) U^* \hoplus \pi_Y , J_{Y} ) \\
& \cong \left(A \oplus Y \otimes \ell^2 (\cI)  \hoplus Y, \lambda_A \oplus \pi_Y \otimes 1_{\ell^2 (\cI)} \hoplus \pi_Y, \begin{bmatrix} 0& U^* \\ U & 0 \end{bmatrix} \right).
\end{align*}
Since $U$ is a compact perturbation of $S:=\bigoplus_{i \in \cI} \rho_i^A (S_i \otimes 1_A) \lL_A (Y \otimes \ell^2 (\cI), Y)$,
this Kasparov bimodule is homotopic to
$$
(A \hoplus 0, \lambda_A \hoplus 0, 0 ) \oplus \left( Y \otimes \ell^2( \cI)  \hoplus Y, \pi_Y  \otimes 1_{\ell^2(\cI )} \hoplus \pi_Y, \begin{bmatrix} 0& S^* \\ S & 0 \end{bmatrix} \right).
$$
Since evaluations of $\cX_i$'s at 1 are degenerate, so is the one of $\cX$.
By Lemma \ref{lem-throughAr} there exists a $*$-homomorphism $\bar{\pi}_\re : A_\re \to \lL_A (Y)$ such that $\pi_X = \bar{\pi}_X \circ \pi_\re$.
Thus, letting
\[
\cY:= \left( Y \otimes \ell^2 (\cI)  \hoplus Y, ( \bar{\pi}_X \otimes 1_A) \otimes 1_{\ell^2( \cI)} \hoplus \bar{\pi}_X \otimes 1_A , \begin{bmatrix} 0& S^* \\ S & 0 \end{bmatrix} \right)
\]
we have $0 = [ \cX_1] = [ \cX_0 ] = \id_A + \pi_\re \otimes_{A_\re} [\cY]$.

\medskip
We next prove that $[\cY ] \otimes_A \pi_\re  + \id_{A_\re}=0$ in $KK (A_\re, A_\re)$.
Our proof depends on the argument in \cite[Section 4]{Germain-duke}.
We note that $[\cY ] \otimes_A \pi_\re$ is represented by
$$
\cY_{\pi_\re} = \left( X \otimes_B A_\re \otimes \ell^2 (\cI)  \hoplus X \otimes_B A_\re, ( \bar{\pi}_X \otimes 1_{A_\re}) \otimes 1_{\ell^2( \cI)} \hoplus \bar{\pi}_X \otimes 1_{A_\re} , \begin{bmatrix} 0& S_{\pi_\re}^* \\ S_{\pi_\re} & 0 \end{bmatrix} \right).
$$
Let $\widetilde{\pi}_\re :=\pi_\re  \otimes \id_{C[0,1]} :IA \to IA_\re$ and consider the pushout of $\cX$ by $\widetilde{\pi}_\re$,
which is the Kasparov $A$-$A_\re$ bimodule
$(I(X\otimes_B A_\re) \hoplus I(X\otimes_B A_\re), \psi_{\widetilde{\pi}_\re } \hoplus (\pi_X \otimes 1_{A_\re}) \otimes 1_{C[0,1]}, J_{I(X\otimes_B A_\re)}).$

We claim that the $*$-homomorphism $\psi_{\widetilde{\pi}_\re }$ factors through $A_\re$.
For any $ a \in A$ we have $ \|\psi_{\widetilde{\pi}_\re } (a) \| = \sup_{0 \leq t \leq 1} \|( \psi_{\widetilde{\pi}_\re })_t (a) \|$.
Since $ (\psi_{\widetilde{\pi}_\re })_t = (\psi_t)_{\pi_\re } : A \to \lL_{A_\re} (X\otimes_B A_\re )$,
it is sufficient to see that  the $(\psi_t)_{\pi_\re }$ factors through $A_\re$ for every $t \in [0,1]$.
Furthermore, since $\bar{\pi}_X : A_\re \to \lL_B (X)$ is injective, it enoughs to show that $\psi_t \otimes 1_X: A \to \lL_B (Y \otimes_A X)= \lL_B (X\otimes_B X)$ factors through $A_\re$.
Recall that $\lambda_i^{X_i}: \lL_B (X_i \otimes_B X_i ) \to \lL_B (X_i \otimes_B X)$ is the $*$-homomorphism given by $x \mapsto (1_{X_i} \otimes v_i^*)(x \otimes 1_{X(\lambda, i)} ) (1_{X_i} \otimes v_i).$
Now we have
$$
\psi_t \otimes 1_X = \us\bigstar_{i \in \cI} \left[ \psi_t^{(i)} \otimes 1_X \oplus \tau_i \otimes 1_X \right]
=\us\bigstar_{i \in \cI} \left[  \lambda_{i}^{X_i}( \psi_t^{(i)} \otimes 1_{X_i}) \oplus \tau_i \otimes 1_X \right].
$$
Consider the natural isomorphism $T_i : (A_i \oplus Y_i ) \otimes_{A_i} X_i \cong X_i \oplus X_i \otimes_B X_i \cong X_i \otimes_B (B \oplus X_i)$ and set $W_t^{(t)}:=(U_t^{(i)} \otimes 1_{X_i} )\circ T_i^* : X_i \otimes_B (B \oplus X_i) \to X_i \otimes_B X_i$. 
We then have
$$
\psi_t^{(i)} \otimes 1_{X_i} = W_{t}^{(i)} (\pi_{X_i} \otimes 1_{B \oplus X_i}) W_t^{(t)*}.
$$
We note that $\xi_i \otimes \xi_i=W_t^{(i)} (\xi_i \otimes (\cos (\pi t) 1_B \oplus \sin (\pi t ) \xi_i) ) $ for $0 \leq t \leq 1/2 $ and $\xi_i \otimes \xi_i =W_t^{(i)} ( \xi_i \otimes (0 \oplus \xi_i ) $ for $1/2 \leq t \leq 1$.
Hence, Lemme \ref{lem-throughAr2} says that $\psi_{\widetilde{\pi}_\re}$ factors through $A_\re$.

Thus, there exists a $*$-homomorphism $\phi : A_\re \to \lL_{IA_\re} (I(X\otimes_B A_\re))$ satisfies that $\psi_{\widetilde{\pi}_\re } = \phi \circ \pi_\re$.
Then $\cZ=(
I(X \otimes_B A_\re) \hoplus I(X\otimes_B A_\re),
\phi \hoplus (\bar{\pi}_X \otimes 1_{A_\re}) \otimes 1_{C[0,1]},
J_{I(X\otimes_B A_\re)}
)$ forms a Kasparov $A_\re$-$IA_\re$ bimodule and satisfies that $\pi_\re^* ( \cZ ) = \cX_{\widetilde{\pi}_\re}$.
We note that $\pi_\re^* ( \cZ_0)
=(\pi_\re^* (\cZ))_0
= ( \cX_{\widetilde{\pi}_\re} )_0
=(\cX_{\widetilde{\pi}_\re })_0$,
which consists of the Hilbert $\rC^*$-module $ A_\re \oplus (X \otimes_B A_\re) \otimes \ell^2 (\cI)  \hoplus X \otimes_B A_\re$,
the left action $ \pi_\re \oplus ( \pi_X \otimes 1_{A_\re}) \otimes 1_{\ell^2 (\cI)} \hoplus ( \pi_X \otimes 1_{A_\re})$,
and the degree 1 operator $\left[ \begin{smallmatrix} 0& U_{\pi_\re}^* \\ U_{\pi_\re} & 0 \end{smallmatrix} \right]$.
Hence, the evaluation $\cZ_0$ is just
$$
\left( A_\re \oplus (X \otimes_B A_\re) \otimes \ell^2 (\cI)  \hoplus X \otimes_B A_\re, \lambda_{A_\re }\oplus ( \bar{\pi}_X \otimes 1_{A_\re}) \otimes 1_{\ell^2 (\cI)} \hoplus ( \bar{\pi}_X \otimes 1_{A_\re}), \begin{bmatrix} 0& U_{\pi_\re}^* \\ U_{\pi_\re} & 0 \end{bmatrix} \right).
$$
Since $[ \cZ_1] =0$ and $U_{\pi_\re}$ is a compact perturbation of $S_{\pi_\re}$,
it follows that $0 = [\cZ_0]= \id_{A_\re} + [\cY_{\pi_\re}] = \id_{A_\re} + [\cY] \otimes_{A} \pi_\re$.
\end{proof}

\setcounter{theorem}{0}
\renewcommand{\thetheorem}{\arabic{section}.\arabic{theorem}}

\section{Applications}\label{sec-app}
In this section, we give some applications in $KK$-theory.
The next theorem follows from Thomsen's result on full amalgamated free products (\cite[Theorem 2.7]{Thomsen}) and Corollary D.
\begin{theorem}
Let $A_1, A_2$, and $B$ be unital separable $\rC^*$-algebras, and $i_k :B \to A_k, k=1,2$ be unital embedding with nondegenerate conditional expectations $E_k : A_k \to i_k(B)$.
Set $A:=(A_1,E_1) \bigstar_B (A_2,E_2)$.
Let $j_k : A_k \to A, k=1,2$ be the canonical embeddings.
If each $A_i$ is nuclear and $B$ is finite dimensional, then for any separable $\rC^*$-algebra $D$ there are two cyclic six terms exact sequences:
\[
\begin{CD}
KK(D,B) @>(i_{1*},i_{2*}) >> KK(D,A_1) \oplus KK (D, A_2) @>j_{1*}- j_{2*}>>KK( D, A )  \\
@AAA                                    @.                                       @VVV \\
KK(SD,A  ) @<j_{1*}- j_{2*}<< KK(SD,A_1 )\oplus KK (SD, A_2) @<(i_{1*},i_{2*}) <<KK( SD, B)\\
\end{CD}\]
\[ \begin{CD}
KK(B,D) @<i_1^*- i_2^*<< KK(A_1,D) \oplus KK (A_2,D) @<j_1^* + j_2^*<< KK( A, D)  \\
@VVV                                    @.                                       @AAA \\
KK(A , SD ) @>j_1^* + j_2^*>> KK(A_1, SD )\oplus KK (A_2, SD) @>i_1^*- i_2^*>> KK( B,SD)
\end{CD}
\]
In particular, we have
$$
\begin{CD}
K_0(B) @>(i_{1*},i_{2*}) >> K_0(A_1) \oplus K_0 (A_2) @>j_{1*}- j_{2*}>>K_0(  A )  \\
@AAA                                    @.                                       @VVV \\
K_1 (A  ) @<j_{1*}- j_{2*}<< K_1 (A_1 )\oplus K_1 (A_2) @< << 0
\end{CD}
$$
\end{theorem}

We next discuss $KK$-equivalence of full and reduced HNN-extensions of $\rC^*$-algebras.
We refer to \cite{Ueda1}\cite{Ueda2} for the definition of HNN-extensions.

\begin{theorem}
Let $B \subset A$ be a unital inclusion of separable $\rC^*$-algebras and $\theta : B \to A$ be an injective $*$-homomorphism.
Assume that there exist conditional expectations $E :A \to B$ and $E_\theta : A \to \theta (B)$ such that the triple $(\lM_2(A), B \oplus \theta (B), E \oplus E_\theta)$ is strongly nuclear,
where
$$
B \oplus \theta (B) =\left \{ \begin{bmatrix} b_1 & 0 \\ 0& \theta (b_2) \end{bmatrix} \middle| b_1, b_2 \in B \right\}, \quad  E \oplus E_\theta : \begin{bmatrix} a_1 & a_2 \\ a_3& a_4 \end{bmatrix} \mapsto \begin{bmatrix} E (a_1) & 0 \\ 0& E_\theta (a_4) \end{bmatrix}.
$$
Then the canonical surjection from the full HNN-extension $A \bigstar_B^\univ \theta$ onto the reduced one $(A,E) \bigstar_B ( \theta, E_\theta)$ is a $KK$-equivalence.
\end{theorem}
\begin{proof}
Let $\cA_{\rm f}$ and $\cA_\re$ be the full and reduced amalgamated free product $\lM_2 (A) \bigstar_{B \oplus B } \lM_2 (B)$ and $(\lM_2 (A), E \oplus E_\theta) \bigstar_{B \oplus B } (\lM_2 (B), E_1)$, where the inclusion $B \oplus B \to B \oplus \theta (B) \subset \lM_2(A)$ and the conditional expectation $E \oplus E_\theta$ are as above,
and the inclusion $\iota_1 : B \oplus B \to \lM_2 (B)$ and the conditional expectation $E_1 : \lM_2 (B) \to B$ are defined by
\begin{align*}
\iota_1 :(b_1 \oplus b_2 ) \mapsto \begin{bmatrix} b_1 & 0 \\ 0 & b_2 \end{bmatrix},
\quad E_1: \begin{bmatrix} b_1 & b_2 \\ b_3& b_4 \end{bmatrix} \mapsto \begin{bmatrix} b_1 & 0 \\ 0& b_4 \end{bmatrix}.
\end{align*}
We also set $\cB_{\rm f}:=A \bigstar_B^\univ \theta$ and $\cB_\re:=(A,E) \bigstar_B ( \theta, E_\theta)$.
By remarks following to \cite[Proposition 3.1]{Ueda2} and \cite[Proposition 3.3]{Ueda2},
there exist two $*$-isomorphism $\Phi_{\rm f} : \cA_{\rm f} \to \lM_2(\cB_{\rm f})$ and $\Phi_\re : \cA_\re \to \lM_2 (\cB_\re )$ such that $\sigma_\re ^{(2)} \circ \Phi_{\rm f} = \Phi_\re \circ \pi_\re$, where $\pi_\re : \cA_{\rm f} \to \cA_\re$ and $\sigma_\re: \cB_{\rm f} \to \cB_\re$ are canonical surjections.
Since the triple $(\lM_2 (B), B \oplus B , E_1)$ is also strongly nuclear by Proposition \ref{prop-sum-tensor},
$\pi_\re$ is a $KK$-equivalence.
Consider $*$-homomorphisms $\phi_\epsilon: \cB_\epsilon \to \lM_2 (\cB_\epsilon)$ and $\psi_ \epsilon : \lM_2 (\cB_\epsilon ) \to \cB_\epsilon$ for $\epsilon \in \{\rm f, r \}$ defined by 
$$
\phi_\epsilon : b \mapsto \begin{bmatrix} b & 0 \\ 0 & 0 \end{bmatrix}, \quad \psi_\epsilon : \begin{bmatrix} b_1 & b_2 \\ b_3 & b_4 \end{bmatrix} \mapsto b_1,
$$
for $b, b_i \in \cB_\epsilon, i=1,2,3,4$.
Then we have the following commuting diagrams:
$$
\begin{CD}
\cA_{\rm f} @>\Phi_{\rm f} >> \lM_2 (\cB_{\rm f}) @<\phi_{\rm f}<< \cB_{\rm f}  \\
@V\pi_\re VV		@V\sigma_\re^{(2)} VV        @VV\sigma_\re V \\
\cA_\re @>\Phi_\re >> \lM_2 (\cB_\re) @<\phi_\re<< \cB_\re  \end{CD}
\qquad \quad
\begin{CD}
\cA_{\rm f} @>\Phi_{\rm f} >> \lM_2 (\cB_{\rm f}) @>\psi_{\rm f}>> \cB_{\rm f}  \\
@V\pi_\re VV		@V\sigma_\re^{(2)} VV        @VV\sigma_\re V \\
\cA_\re @>\Phi_\re >> \lM_2 (\cB_\re) @>\psi_\re>> \cB_\re  \end{CD}
$$
Let $\alpha \in KK(\cA_\re, \cA_{\rm f})$ be an element satisfying $\pi_\re \otimes_{\cA_\re} \alpha =\id_{\cA_{\rm f}}$ and $\alpha \otimes_{\cA_{\rm f}} \pi_\re = \id_{\cA_\re}$.
Set $\beta:=\Phi_\re^{-1} \otimes_{\cA_\re} \alpha \otimes_{\cA_{\rm f}} \Phi_{\rm f} \in KK (\lM_2 (\cB_\re), \lM_2 (\cB_{\rm f}) )$.
Then we have $\id_{\cB_{\rm f}} = \phi_{\rm f} \otimes_{\lM_2 (\cB_{\rm f} )} \sigma_\re^{(2)} \otimes_{\lM_2 (\cB_\re )} \beta \otimes_{\lM_2 (\cB_{\rm f})}  \psi_{\rm f} 
=\sigma_\re \otimes_{\cB_\re } \phi_\re \otimes_{\lM_2 ( \cB_\re )} \beta \otimes_{\lM_2(\cB_{\rm f})} \psi_{\rm f}$ and
$\id_{\cB_\re} =\phi_\re \otimes_{\lM_2 (\cB_\re)} \beta \otimes_{\lM_2 (\cB_{\rm f})} \sigma_\re^{(2)} \otimes_{\lM_2 (\cB_\re)} \psi_\re 
=\psi_\re \otimes_{\lM_2 (\cB_\re)} \beta \otimes_{\lM_2 (\cB_{\rm f})} \psi_{\rm f} \otimes_{\lM_2 (\cB_{\rm f})} \sigma_\re$,
hence the $\phi_\re \otimes_{\lM_2 (\cB_\re)} \beta \otimes_{\lM_2 (\cB_{\rm f})} \psi_{\rm f} \in KK (\cB_\re, \cB_{\rm f})$ is the desired inverse element of $\sigma_\re$.
\end{proof}
We remark that when $B=\theta (B)$, the strong nuclearity of $(\lM_2 (A), B \oplus B, E \oplus E)$ follows from the one of $(A, B, E)$.
Combining the theorem above with \cite[Proposition 4.12]{Ueda2} we obtain the following six term exact sequence for reduced HNN-extensions.
\begin{corollary}
Let $B \subset A$ be a unital inclusion of separable $\rC^*$-algebras and $\theta : B \to A$ be an injective $*$-homomorphism.
Assume that there exist conditional expectations $E :A \to B$ and $E_\theta : A \to \theta (B)$ such that the triple $(\lM_2(A), B \oplus \theta (B), E \oplus E_\theta)$ is strongly nuclear, then we have the following six term exact sequence:
$$
\begin{CD}
K_0(B) @>(\theta_{1*}- \iota_{B*}) >> K_0(A) @>\iota_{A*}>>K_0(  (A,E) \bigstar_B ( \theta, E_\theta))  \\
@AAA                                    @.                                       @VVV \\
K_1 ((A,E) \bigstar_B ( \theta, E_\theta) ) @<\iota_{A*}<< K_1 (A ) @<(\theta_{1*}- \iota_{B*}) <<K_1( B)
\end{CD}
$$
Here $\iota_B : B \to A$ and $\iota_A: A \to (A,E) \bigstar_B ( \theta, E_\theta) $ are inclusion maps.
\end{corollary}


\begin{thebibliography}{99}
\bibitem[AD1]{Delaroche}C.~Anantharaman-Delaroche, {\it Syst\`emes dynamiques non commutatifs et moyennabilit\'e.} Math. Ann. {\bf 279} (1987), no. 2, 297--315.
\bibitem[AD2]{Delaroche-JLMS} C.~Anantharaman-Delaroche, {\it Some remarks on the cone of completely positive maps between von Neumann algebras}. 
J. London Math. Soc. (2) {\bf 55} (1997), no. 1, 193--208.
\bibitem[AD3]{Delaroche2}C.~Anantharaman-Delaroche, {\it Amenable correspondences and approximation properties for von Neumann algebras.} Pacific J. Math. {\bf 171} (1995), no. 2, 309--341. 
\bibitem[ADH]{Delaroche-Havet} C.~Anantharaman-Delaroche and J.F.~Havet, {\it On approximate factorizations of completely positive maps.} J. Funct. Anal. {\bf 90} (1990), no. 2, 411--428.
\bibitem[Ar]{Arveson} W.~Arveson, {\it Notes on extensions of C$^*$-algebras.} Duke Math. J. {\bf 44} (1977), no. 2, 329--355.
\bibitem[Ba]{Bauval} A.~Bauval, {\it $RKK(X)$-nucl\'{e}arit\'{e} (d'apr\`{e}s G. Skandalis)}. $K$-Theory {\bf 13} (1998), no. 1, 23--40. 
\bibitem[Bl]{Blackadar} B.~Blackadar, {\it K-theory for Operator Algebras.} Second edition. Mathematical Sciences Research Institute Publications, 5. Cambridge University Press, Cambridge, 1998.
\bibitem[BO]{Brown-Ozawa} N.P.~Brown and N.~Ozawa, {\it C$^*$-algebras and Finite-Dimensional Approximations.} Graduate Studies in Mathematics, {\bf 88}. American Mathematical Society, Providence, RI, 2008.
\bibitem[C]{Conway} J.B.~Conway, {\it A Course in Operator Theory.} Graduate Studies in Mathematics, {\bf 21.} American Mathematical Society, Providence, RI, 2000. 
\bibitem[CE]{Choi-Effros} M.D.~Choi andE.G.~Effros, {\it Nuclear $C^*$-algebras and the approximation property}. Amer. J. Math. {\bf 100} (1978), no.\ 1, 61--79. 
\bibitem[DE]{Dadarlat-Eilers} M.~Dadarlat and S.~Eilers, {\it On the classification of nuclear $C^*$-algebras.} Proc. London Math. Soc. (3) {\bf 85} (2002), no. 1, 168--210.
\bibitem[EL]{Effros-Lance} E.G.~Effros and E.C.~Lance, {\it Tensor products of operator algebras.} Adv. Math. {\bf 25} (1977), no. 1, 1--34.
\bibitem[F]{Fell} J.M.G.~Fell, {\it The dual spaces of C$^*$-algebras}, Trans. Amer. Math. Soc. {\bf 94} (1960), no. 3, 365--403.  
\bibitem[FK]{Frank-Kirchberg}M.~Frank and E.~Kirchberg, {\it On conditional expectations of finite index.}
J. Operator Theory {\bf 40} (1998), no. 1, 87--111.
\bibitem[Ge1]{Germain-duke} E.~Germain, {\it $KK$-theory of reduced free-product $C^*$-algebras.} Duke Math. J. {\bf 82} (1996), no. 3, 707--723.
\bibitem[Ge2]{Germain-fields} E.~Germain, {\it Amalgamated free product $C^*$-algebras and $KK$-theory.} in Free Probability Theory, 89--103, Fields Inst. Commun., {\bf 12}, Amer. Math. Soc., Providence, RI, 1997.
\bibitem[Ge3]{Germain-thesis} E.~Germain,{\it KK-theory of reduced free product C*-algebras with amalgamation over C(X).} Thesis (Ph.D.)--University of California, Berkeley. 1994.
\bibitem[Gr]{Greenleaf} F.P.~Greenleaf, {\it Amenable actions of locally compact groups.} J. Functional Analysis {\bf 4} 1969 295--315.
\bibitem[H]{Haagerup} U.~Haagerup, {\it The standard form of von Neumann algebras.} Math. Scand. {\bf 37} (1975), no. 2, 271--283. 
\bibitem[I]{Izumi} M.~Izumi. {\it Inclusions of simple $C^*$-algebras.} J. Reine Angew. Math. {\bf 547} (2002), 97--138.
\bibitem[ILP]{Izumi-Longo-Popa}M.~Izumi, R.~Longo, and S.~Popa, {\it A Galois correspondence for compact groups of automorphisms of von Neumann algebras with a generalization to Kac algebras.} J. Funct. Anal. {\bf 155} (1998), no. 1, 25--63.
\bibitem[JT]{Jensen-Thomsen} K.K.~Jensen, K.~Thomsen, {\it Elements of KK-theory}. Mathematics: Theory \& Applications. Birkhauser Boston, Inc., Boston, MA, 1991.
\bibitem[JS]{Jian-Sepideh} L.~Jian and R.~Sepideh, {\it Operator-Valued Kirchberg Theory}, preprint, arXiv:1410.5379
\bibitem[Ka]{Kasparov}G.G.~Kaspraov, {\it Hilbert $C^*$-modules: theorems of Stinespring and Voiculescu.} J. Operator Theory {\bf 4} (1980), no. 1, 133--150.
\bibitem[Ki]{Kirchberg}E.~Kirchberg, {\it $C^*$-nuclearity implies CPAP.} Math. Nachr. 76 (1977), 203--212. 
\bibitem[L1]{Lance_JFA} E.C.~Lance, {\it On nuclear $C^*$-algebras.} J. Functional Analysis {\bf 12} (1973), 157--176. 
\bibitem[L2]{Lance}E.C.~Lance, {\it Hilbert C$^*$-modules. A toolkit for operator algebraists.} London Mathematical Society Lecture Note Series, {\bf 210}. Cambridge University Press, Cambridge, 1995.
\bibitem[MP]{Monod-Popa} N.~Monod and S.~Popa, {\it On co-amenability for groups and von Neumann algebras.} C. R. Math. Acad. Sci. Soc. R. Can. {\bf 25} (2003), no. 3, 82--87. 
\bibitem[OP]{Ozawa-Popa} N.~Ozawa and S.~Popa, {\it On a class of {\rm II}$_1$ factors with at most one Cartan subalgebra.} Ann. of Math. (2) {\bf 172} (2010), no. 1, 713--749.
\bibitem[Pe]{Pestov} V.~Pestov, {\it On some questions of Eymard and Bekka concerning amenability of homogeneous spaces and induced representations.} C. R. Math. Acad. Sci. Soc. R. Can. {\bf 25} (2003), no. 3, 76--81. 
\bibitem[Po]{Popa} S.~Popa, {\it Correspondences}, INCREST Preprint, 56/1986
\bibitem[Sk]{Skandalis} G.~Skandalis, {\it Une notion de nucl\'{e}arit\'{e} \'{e}n K-th\'{e}orie (d'apr\`{e}s J. Cuntz).} K-Theory {\bf 1} (1988), no. 6, 549--573. 
\bibitem[Ta]{Takesaki} M.~Takesaki, {\it On the cross-norm of the direct product of $C^*$-algebras} T\^{o}hoku Math. J. (2) {\bf 16} (1964), 111--122. 
\bibitem[Th]{Thomsen} K.~Thomsen, {\it On the $KK$-theory and the $E$-theory of amalgamated free products of $C^*$-algebras.} J. Funct. Anal. {\bf 201} (2003), no. 1, 30--56.
\bibitem[U1]{Ueda1} Y.~Ueda, {\it HNN extensions of von Neumann algebras}. J. Funct. Anal. {\bf 225} (2005), no. 2, 383--426.
\bibitem[U2]{Ueda2} Y.~Ueda, {\it Remarks on HNN extensions in operator algebras.} Illinois J. Math. {\bf 52} (2008), no. 3, 705--725. 
\bibitem[V1]{Voiculescu} D.~Voiculescu, {\it A non-commutative Weyl-von Neumann theorem.} Rev. Roumaine Math. Pures Appl. {\bf 21} (1976), no. 1, 97--113.
\bibitem[V2]{Voiculescu-free} D.~Voiculescu, {\it Symmetries of some reduced free product C*-algebras}. in Operator Algebras and Their Connections with Topology and Ergodic Theory, 556--588, Lecture Notes in Math., 1132, Springer, Berlin, 1985. 
\bibitem[W]{Watatani} Y.~Watatani, {\it Index for $C^*$-subalgebras.} Mem.\ Amer. Math. Soc. {\bf 83} (1990), no. 424.
\end{thebibliography}
\end{document}